\documentclass[11pt]{amsart}
\usepackage{latexsym,amscd,amssymb, graphicx, color, amsthm, bm, amsmath, cancel, enumitem, ytableau, float}  
\usepackage{tikz}
\usepackage[margin=1in]{geometry}
\usepackage{hyperref}
\usepackage[all,cmtip]{xy}
\usepackage{mathtools}
\usepackage{rotating}
\usetikzlibrary{math, arrows.meta}
\usetikzlibrary{calc}
\usepackage[backend=bibtex]{biblatex}
\numberwithin{equation}{section}

\newtheorem{theorem}{Theorem}[section]
\newtheorem{proposition}[theorem]{Proposition}
\newtheorem{corollary}[theorem]{Corollary}
\newtheorem{lemma}[theorem]{Lemma}
\newtheorem{conjecture}[theorem]{Conjecture}

\newtheorem{problem}[theorem]{Problem}
\newtheorem{example}[theorem]{Example}
\newtheorem{remark}[theorem]{Remark}
\newtheorem{definition}[theorem]{Definition}

\theoremstyle{definition}

\newcommand{\Hilb}{{\mathrm{Hilb}}}

\newcommand{\gl}{\mathrm{GL}}
\newcommand{\Frob}{{\mathrm{Frob}}}
\newcommand{\symm}{{\mathfrak{S}}}
\newcommand{\II}{{\mathbf{I}}}
\newcommand{\gr}{{\mathrm {gr}}}

\newcommand{\lp}[1]{\mathcal{LP}(#1)}
\newcommand{\wid}[1]{\mathrm{wid}(#1)}
\newcommand{\hori}[1]{\mathrm{HS}(#1)}
\newcommand{\horipositive}[1]{\mathrm{PHS}(#1)}
\newcommand{\horiwid}[1]{\mathrm{WHS}(#1)}
\newcommand{\leftshadow}[1]{\mathcal{LS}_{#1}}

\newcommand{\defideal}[3]{{I_{#1,#2}^{(#3)}}}

\newcommand{\RRR}{{\mathcal{R}}}

\newcommand{\grFrob}{{\mathrm{grFrob}}}

\newcommand{\ZZZ}{{\mathcal{Z}}}
\newcommand{\UZ}{{\mathcal{UZ}}}

\newcommand{\MMM}{{\mathcal{M}}}

\newcommand{\CC}{{\mathbb{C}}}

\newcommand{\ZZ}{{\mathbb{Z}}}

\newcommand{\Mat}{{\mathrm{Mat}}}

\newcommand{\xxx}{{\mathbf{x}}}

\newcommand{\mmm}{{\mathfrak{m}}}

\newcommand{\ind}{{\mathrm{Ind}}}

\newcommand{\precdot}{\mathrel{\ooalign{$\prec$\cr
  \hidewidth\raise0.001ex\hbox{$\cdot\mkern0.6mu$}\cr}}}

\title{Rook placements and orbit harmonics}

\author{Hai Zhu}
\address{Department of Mathematics, UC San Diego, La Jolla, CA, 92093, USA}
\email{haz138@ucsd.edu}
\date{\today}

\begin{document}

\begin{abstract}
    For fixed positive integers $n,m$, let $\mathrm{Mat}_{n\times m}(\mathbb{C})$ be the affine space consisting of all $n\times m$ complex matrices, and let $\mathbb{C}[\mathbf{x}_{n\times m}]$ be its coordinate ring. For $0\le r\le\min\{m,n\}$, we apply the orbit harmonics method to the finite matrix loci $\mathcal{Z}_{n,m,r}$ of rook placements with exactly $r$ rooks, yielding a graded $\mathfrak{S}_n\times\mathfrak{S}_m$-module $R(\mathcal{Z}_{n,m,r})$. We find one signed and two sign-free graded character formulae for $R(\mathcal{Z}_{n,m,r})$. We also exhibit some applications of these formulae, such as proving a concise presentation of $R(\mathcal{Z}_{n,m,r})$, and proving some module injections and isomorphisms. Some of our techniques are still valid for involution matrix loci.
\end{abstract}

\maketitle

\section{Introduction}\label{sec:intro}
Let $\xxx_N = (x_1,\dots,x_N)$ be a sequence of variables and let $\CC[\xxx_N]\coloneqq\CC[x_1,\dots,x_N]$ be the polynomial ring over these variables. For a finite locus $\ZZZ\subseteq\CC^N$, its vanishing ideal $\II(\ZZZ)\subseteq\CC[\xxx_N]$ is given by
\[\II(\ZZZ)\coloneqq\{f\in\CC[\xxx_N]\,:\,\text{$f(z)=0$ for all $z\in\ZZZ$}\}.\]
The orbit harmonics method assigns a homogeneous ideal $\gr\II(\ZZZ)$ to $\II(\ZZZ)$, yielding a graded ring $R(\ZZZ)\coloneqq\CC[\xxx_N]/\gr\II(\ZZZ)$. We have a chain of vector space isomorphisms
\begin{equation}\label{eq:intro-chain-isom-orbit-harmonics}
    \CC[\ZZZ]\cong\CC[\xxx_N]/\II(\ZZZ)\cong R(\ZZZ).
\end{equation}
If further $\ZZZ$ is closed under the action of a matrix group $G\subseteq\gl_N(\CC)$, the chain~\eqref{eq:intro-chain-isom-orbit-harmonics} is also a chain of $G$-module isomorphisms. As $R(\ZZZ)$ is a graded $G$-module, it provides a graded refinement of $\CC[\ZZZ]$, which usually has interesting and profound combinatorial interpretation relevant to $\ZZZ$.

Geometrically, the orbit harmonics method is a linear deformation from the reduced locus $\ZZZ$ to a scheme of multiplicity $\lvert\ZZZ\rvert$ supported at the origin. For example, let $\ZZZ$ be the orbit of a point (which is not the origin) under the action of the group ($\cong\symm_3$) generated by reflections in the lines below. The orbit harmonics deformation looks like:
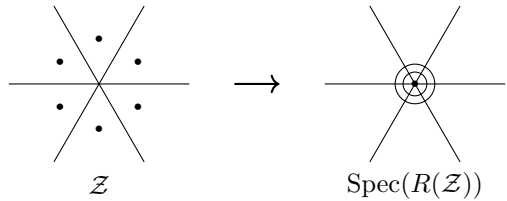
\begin{center}
 \begin{tikzpicture}[scale = 0.3]
\draw (-4,0) -- (4,0);
\draw (-2,-3.46) -- (2,3.46);
\draw (-2,3.46) -- (2,-3.46);

 \fontsize{5pt}{5pt} \selectfont
\node at (0,2) {$\bullet$};
\node at (0,-2) {$\bullet$};

\node at (-1.73,1) {$\bullet$};
\node at (-1.73,-1) {$\bullet$};
\node at (1.73,-1) {$\bullet$};
\node at (1.73,1) {$\bullet$};

\node[font=\fontsize{10}{6}\selectfont] at (0,-4.5) {$\ZZZ$};

\draw[thick, ->] (6,0) -- (8,0);

\draw (10,0) -- (18,0);
\draw (12,-3.46) -- (16,3.46);
\draw (12,3.46) -- (16,-3.46);

\draw (14,0) circle (15pt);
\draw(14,0) circle (25pt);
\node at (14,0) {$\bullet$};

\node[font=\fontsize{10}{6}\selectfont] at (14,-4.5) {$\mathrm{Spec}(R(\ZZZ))$};

 \end{tikzpicture}
\end{center}

Orbit harmonics interacts with diverse topics, such as presentations of cohomology rings \cite{MR1168926}, Macdonald theory \cite{MR4223028,MR3783430}, cyclic sieving \cite{MR4359538}, Donaldson–Thomas theory \cite{MR4675069}, and Ehrhart theory \cite{reiner2024harmonicsgradedehrharttheory}.

We are interested in the application of orbit harmonics to finite matrix loci. Throughout this paper, we fix $n,m\in\mathbb{N}\coloneqq\ZZ_{>0}$ unless specifically emphasized. Let $\xxx_{n\times m}=(x_{i,j})_{1\le i\le n,1\le j\le m}$ be an $n\times m$ matrix of variables, and consider the affine space $\Mat_{n\times m}(\CC)$ of $n\times m$ matrices with coordinate ring $\CC[\xxx_{n\times m}]$. Let $\symm_n\times\symm_m\subseteq\gl(\Mat_{n\times m}(\CC))$ act on $\Mat_{n\times m}(\CC)$ by respectively permuting rows and columns of matrices. We desire finite matrix loci $\ZZZ\subseteq\Mat_{n\times m}(\CC)$ stable under the action of $\symm_n\times\symm_m$. In this case, $\symm_n\times\symm_m$ acts on $\CC[\xxx_{n\times m}]$ by permuting rows and columns of the variable matrix $\xxx_{n\times m}$, which descends to the $\symm_n\times\symm_m$-action on $R(\ZZZ)$.

For the case $n=m$, Rhoades \cite{MR4821538} initiated the application of orbit harmonics to finite matrix loci $\ZZZ\subseteq\Mat_{n\times n}(\CC)$. He regarded the permutation matrix locus $\ZZZ=\symm_n\subseteq\Mat_{n\times n}(\CC)$, finding that algebraic properties of $R(\symm_n)$ are governed by longest increasing subsequences of permutations in $\symm_n$ and the Schensted correspondence. Additionally, he generalized Chen's log-concavity conjecture \cite{chen2008logconcavityqlogconvexityconjectureslongest} concerning the length of longest increasing subsequences to the equivariant log-concavity conjecture for $R(\symm_n)$. Details about these conjectures are given in Section~\ref{sec:conclusion}. Liu extended \cite{MR4936857} Rhoades's work to the $r$-colored permutation matrix locus $\ZZZ=(\ZZ/r\ZZ) \wr \symm_n$.

In this paper, we do not require $n=m$. Liu and Zhu extended \cite{liu2025extensionviennotsshadowrook} Rhoades's work as well as the equivariant log-concavity conjecture to the upper rook placement locus $\UZ_{n,m,r} \subseteq \Mat_{n\times m}(\CC)$ in Definition~\ref{def:rook}. We further consider each $(\symm_n\times\symm_m)$-orbit in $\UZ_{n,m,r}$, namely the rook placement locus $\ZZZ_{n,m,r} \subseteq \Mat_{n\times m}(\CC)$ given by Definition~\ref{def:rook}.
\begin{definition}\label{def:rook}
    A \emph{rook placement} $\RRR$ on the $n\times m$ board is a subset of the grid $[n]\times[m]$ such that each row and column has at most one element of $\RRR$. We call each element $(i,j)\in\RRR$ a \emph{rook} of $\RRR$. Define the \emph{size} $|\RRR|$ of $\RRR$ by
\[|\RRR|\coloneqq \text{the number of rooks of $\RRR$}.\]
Identify each rook placement $\RRR$ on the $n\times m$ board with the $n\times m$ matrix $(a_{i,j})_{1\le i\le n,1\le j\le m}$ given by
\[a_{i,j}=\begin{cases}
    1, &\text{if $(i,j)\in\RRR$} \\
    0, &\text{otherwise.}
\end{cases}\]
Write $\ZZZ_{n,m,r}\subseteq\Mat_{n\times m}(\CC)$ for the set of all rook placements on the $n\times m$ board of size $r$, and let $\UZ_{n,m,r}\coloneqq\bigsqcup_{r^\prime=r}^{\min\{m,n\}}\ZZZ_{n,m,r^\prime}$.
\end{definition}

Our results about $R(\ZZZ_{n,m,r})$ in this paper may help solve the log-concavity conjectures in Rhoades \cite{MR4821538} and Chen \cite{chen2008logconcavityqlogconvexityconjectureslongest} mentioned above. The reasons are as follows:
\begin{itemize}
    \item Note that $\ZZZ_{n,n,n} = \symm_n \subseteq \Mat_{n\times n}(\CC)$. Consequently, the rook placement locus $\ZZZ_{n,m,r}$ is a generalization of the permutation matrix locus $\symm_n$ studied by Rhoades. Interestingly, the equivariant log-concavity of $R(\symm_n)$ conjectured by Rhoades \cite{MR4821538} seems to hold for our $R(\ZZZ_{n,m,r})$ (verified by coding, see Conjecture~\ref{conj:log-concavity} for details). Therefore, with the flexibility of $\ZZZ_{n,m,r}$ arising from more parameters $n,m,r$ than the unique parameter $n$ for $\symm_n$, we may obtain more potential induction strategies to show Rhoades's equivariant log-concavity conjecture for $R(\symm_n)$. The chains of surjections in Figure~\ref{fig:graded-surj} may also provide tools for this potential induction.
    \item Moreover, we have $\ZZZ_{n,n,n}=\UZ_{n,n,n} =\symm_n\subseteq \Mat_{n\times n}(\CC)$, so they are distinct generalizations of the permutation matrix locus $\symm_n$. Both $R(\ZZZ_{n,m,r})$ and $R(\UZ_{n,m,r})$ are conjectured to be equivariantly log-concave (Conjecture~\ref{conj:log-concavity} and \cite[Conjecture 6.1]{liu2025extensionviennotsshadowrook}). Surprisingly, Proposition~\ref{prop:UZ-surj} reveals a tight connection between $R(\ZZZ_{n,m,r})$ and $R(\UZ_{n,m,r})$. This connection may thereby generate more potential tools to prove the equivariant log-concavity of $R(\symm_n)$ conjectured by Rhoades \cite{MR4821538} and hence the log-concavity conjecture of Chen \cite{chen2008logconcavityqlogconvexityconjectureslongest} mentioned above.
\end{itemize}

Our main results on the graded $\symm_n\times\symm_m$-modules $R(\ZZZ_{n,m,r})$ for $0\le r\le\min\{m,n\}$ are as follows.
\begin{itemize}
    \item We give a signed graded character formula for $R(\ZZZ_{n,m,r})$ in Theorem~\ref{thm:grad-str}, i.e.,
    \begin{align}\label{eq:intro-grad-str}
        \grFrob(R(\ZZZ_{n,m,r});q) = \sum_{d=0}^r q^d \cdot \{ SF_d - SF_{d-1} \}_{\lambda_1 \le n+m-d-r}
    \end{align}
    with the convention that $SF_{-1} = 0$. See Definition~\ref{def:specific-sym-func} for the definition of $SF_d$.
    \item As one application of Equation~\eqref{eq:intro-grad-str}, we find a concise presentation of the ring $R(\ZZZ_{n,m,r})$ in Proposition~\ref{prop:ideal-equal}. The same technique works for the involution matrix loci and yields Proposition~\ref{prop:involution-ideal}.
    \item As the other application of Equation~\eqref{eq:intro-grad-str}, we show a series of module injections with respect to the modules $R(\UZ_{n,m,r})_d$ as well as their connections with $R(\ZZZ_{n,m,r})$ in Proposition~\ref{prop:UZ-surj}.
    \item Canceling the minus signs in Equation~\eqref{eq:intro-grad-str}, we figure out a sign-free character formula for $R(\ZZZ_{n,m,r})$ in Theorem~\ref{thm:good-ver-positive-formula}, i.e.,
    \begin{align}\label{eq:intro-grad-positive}
        \grFrob(R(\ZZZ_{n,m,r});q) = \sum_{\substack{\lambda^{(1)}\vdash n \\ \lambda^{(2)}\vdash m}} \sum_{\mu\in\hori{r,\lambda^{(1)},\lambda^{(2)}}} q^{n+m-r-\wid{\mu,\lambda^{(1)},\lambda^{(2)}}}\cdot s_{\lambda^{(1)}}\otimes s_{\lambda^{(2)}}.
    \end{align}
    See details in Section~\ref{sec:sign-cancel}.
    \item As an application of Equation~\eqref{eq:intro-grad-positive}, we find a series of isomorphisms with respect to the modules $R(\ZZZ_{n,m,r})_d$ in Proposition~\ref{prop:surj-to-isom}, resulting in a family of injections among the modules $\CC[\ZZZ_{n,m,r}]$ in Corollary~\ref{cor:C[Z]-injection}.
\end{itemize}

The outline of the rest of this paper is as follows. Section~\ref{sec:background} gives some prerequisite knowledge about Gröbner theory, orbit harmonics, partitions, and the representation theory of symmetric groups. In Section~\ref{sec:module-str}, we study the graded module structure of $R(\ZZZ_{n,m,r})$. Specifically, Subsections~\ref{subsec:module-str-sym-func}, \ref{subsec:defining-dieal}, and \ref{subsec:surj} present technical results, which may be better to skip on a first read, accepting all the technical results directly. The signed graded character formula is proved in Subsection~\ref{subsec:graded-char}, and its applications are presented in Subsection~\ref{subsec:application-of-module-str}. In Section~\ref{sec:sign-cancel}, we provide sign-free graded character formulae for $R(\ZZZ_{n,m,r})$. Specifically, we introduce some notations in Subsection~\ref{subsec:notation}, then give two positive combinatorial formulae respectively in Subsections~\ref{subsec:bad-ver} and \ref{subsec:good-ver}, and finally exhibit an application of them in Subsection~\ref{subsec:application-of-sign-free-formula}. In Section~\ref{sec:conclusion}, we raise some open problems for future study.
In Appendix~\ref{sec:appendix}, we apply some of our techniques to involution matrix loci $\MMM_{n,a}\subseteq\Mat_{n\times n}(\CC)$ studied by Liu, Ma, Rhoades, Zhu \cite{MR4887799}.


\section{Background}\label{sec:background}

\subsection{Gröbner theory}\label{subsec:grobner}
Write $\xxx_N =(x_1,\dots,x_N)$ for a sequence of variables of length $N$. Let $\CC[\xxx_N]$ be the polynomial ring over these variables. A \emph{monomial order} $<$ of $\CC[\xxx_N]$ is a total order on the set of monomials in $\CC[\xxx_N]$ such that
\begin{itemize}
    \item $1\le m$ for all monomials $m\in\CC[\xxx_N]$, and
    \item for monomials $m,m_1,m_2\in\CC[\xxx_N]$ with $m_1<m_2$, we have $m\cdot m_1<m\cdot m_2$.
\end{itemize}

Given a monomial order $<$ of $\CC[\xxx_N]$, the \emph{initial monomial} $\mathrm{in}_< f$ of any nonzero polynomial $f\in\CC[\xxx_N]$ is the maximal monomial in $f$ according to $<$. For any ideal $I\subseteq\CC[\xxx_N]$, we define its \emph{initial ideal} $\mathrm{in}_< I\subseteq\CC[\xxx_N]$ by
\[\mathrm{in}_< I \coloneqq \langle \mathrm{in}_< f \,:\, f\in I, f\neq 0 \rangle.\]
A monomial $m\in\CC[\xxx_N]$ is called a \emph{standard monomial} of $I$ if $m\notin\mathrm{in}_< I$. A famous result states that the set of all the standard monomials of $I$ descends to a $\CC$-linear basis of the $\CC$-vector space $\CC[\xxx_N]/I$, namely the \emph{standard monomial basis}. See \cite{MR4952933} for details.

\subsection{Orbit harmonics}\label{subsec:orbit-harmonoics}
Let $\ZZZ\subseteq\CC^N$ be a finite locus. The \emph{vanishing ideal} $\II(\ZZZ)\subseteq\CC[\xxx_N]$ of $\ZZZ$ is defined by
\[\II(\ZZZ)\coloneqq\{f\in\CC[\xxx_N]\,:\,\text{$f(z)=0$ for all $z\in\ZZZ$}\}.\]

For any ideal $I\subseteq\CC[\xxx_N]$, its \emph{associated graded ideal} $\gr I$ is given by
\[\gr I\coloneqq \langle\tau(f)\,:\,f\in I, f\neq 0\rangle\]
where $\tau(f)$ is the top-degree homogeneous component of $f$. The \emph{orbit harmonics ring} $R(\ZZZ)$ of $\ZZZ$ is given by
\[R(\ZZZ)\coloneqq\CC[\xxx_N]/\gr\II(\ZZZ).\]
\begin{remark}
    Given a set of generators $I = \langle f_1, f_2, \dots, f_r \rangle$, we have the containment of ideals $\langle \tau(f_1),\tau(f_2), \dots, \tau(f_r) \rangle \subseteq \gr \, I$ that is strict in general. Nonetheless, if the set of generators $\{ f_1, f_2, \dots, f_r \}$ forms a Gröbner basis of $I$ with respect to some graded monomial order, then the equality $\langle \tau(f_1), \tau(f_2), \dots, \tau(f_r) \rangle = \gr \, I$ holds. In general, finding and proving a concise generating set of $\gr\II(\ZZZ)$ is difficult.
\end{remark}

Let $\CC[\ZZZ]$ be the vector space consisting of all the functions from $\ZZZ$ to $\CC$. A standard result about orbit harmonics is the chain of vector space isomorphisms
\begin{align}\label{eq:orbit-harmonics-chain-isom}\CC[\ZZZ]\cong\CC[\xxx_N]/\II(\ZZZ)\cong R(\ZZZ)\end{align}
where the first isomorphism arises from Lagrange interpolation, and the second isomorphism arises from a series of vector space isomorphisms
\begin{align}\label{eq:orbit-harmonics-graded-isom}R(\ZZZ)_d \cong \frac{\CC[\xxx_N]_{\le d}/(\II(\ZZZ)\cap\CC[\xxx_N]_{\le d})}{\CC[\xxx_N]_{\le d-1}/(\II(\ZZZ)\cap\CC[\xxx_N]_{\le d-1})}\end{align}
for all $d\ge 0$, which is induced by
\begin{align*}
    \CC[\xxx_N]_d &\longrightarrow \CC[\xxx_N]_{\le d}/(\II(\ZZZ)\cap\CC[\xxx_N]_{\le d}) \\
    f &\longmapsto f + \II(\ZZZ)\cap\CC[\xxx_N]_{\le d}.
\end{align*}
One can show that the map \eqref{eq:orbit-harmonics-graded-isom} is indeed a linear isomorphism using the standard result \cite[Lemma 3.15]{MR4821538}.

We further suppose that $\ZZZ\subseteq\CC^N$ carries an action of a subgroup $G\le \gl_N(\CC)$, which is compatible with the natural action of $\gl_N(\CC)$ on $\CC^N$. Then, the vector space isomorphisms~\eqref{eq:orbit-harmonics-chain-isom} and \eqref{eq:orbit-harmonics-graded-isom} are also $G$-module isomorphisms.

\subsection{Combinatorics}\label{subsec:combinatorics}
For any nonnegative integer $n$, a \emph{partition} of $n$ is a weakly decreasing sequence $\lambda$ of nonnegative integers
\[\lambda=(\lambda_1\ge\lambda_2\ge\dots\ge\lambda_m)\]
with $\sum_{i=1}^m =n$. Here, $n$ is called the \emph{size} of $\lambda$, and we write $\lambda\vdash n$. We allow appending several zeros to $\lambda$. That is, we have the partition identification
\[(\lambda_1,\dots,\lambda_m) = (\lambda_1,\dots,\lambda_m,0,\dots,0).\]
We always identify a partition $\lambda$ with its \emph{Young diagram} that is a diagram with $\lambda_i$ cells in the $i$-th row for all $i\ge 1$. For instance, we identify the partition $\lambda=(4,3,3,1)\vdash 11$ with its Young diagram below.
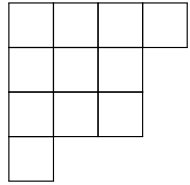
\begin{center}
    \begin{ytableau}
    \: & \: & \: & \: \cr
    \: & \: & \: \cr
    \: & \: & \: \cr
    \:
\end{ytableau}
\end{center}

A \emph{skew partition} $\lambda/\mu$ is a pair of partitions $\lambda$ and $\mu$ such that $\mu\subseteq\lambda$. We call $\mu$ (resp. $\lambda$) the \emph{inner partition} (resp. \emph{outer partition}) of $\lambda/\mu$. We also identify $\lambda/\mu$ with its \emph{Young diagram} given by removing the cells of $\mu$ from $\lambda$. The \emph{size} $\lvert\lambda/\mu\rvert$ of $\lambda/\mu$ is given by $\lvert\lambda/\mu\rvert = \lvert\lambda\rvert-\lvert\mu\rvert$. In particular, $\lambda/\mu$ is called a \emph{horizontal strip} if it possesses at most one cell in each column. For example, we consider the inner partition $\mu=(3,3,1,1)$ and the outer partition $\lambda=(4,3,3,1)$, then the Young diagram $\lambda/\mu$ is marked with bullets $\bullet$ shown below.
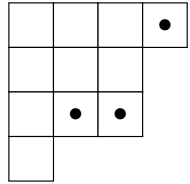
\[\ydiagram[*(white) \bullet]
    {3+1,3+0,1+2,1+0}
    *[*(white)]{3,3,1,1}\]
As $\lambda/\mu$ has at most one cell in each column, we deduce that $\lambda/\mu$ is a horizontal strip.

\subsection{Representation theory}\label{subsec:background-rep}

Let $\Lambda=\bigoplus_{n\ge 0}\Lambda_n$ be the graded ring of symmetric functions in infinitely many variables $x_1,x_2,\dots$ over the rational function field $\CC(q)$. We also define the \emph{doubly symmetric function algebra} to be $\Lambda\otimes_{\CC(q)}\Lambda$, which is naturally bigraded.

The space $\Lambda_n$ of symmetric function of degree $n$ has some well-known $\CC(q)$-vector space bases, including the \emph{monomial basis} $\{m_\lambda\}_{\lambda\vdash n}$, the \emph{complete homogeneous basis} $\{h_\lambda\}_{\lambda\vdash n}$, the \emph{elementary basis} $\{e_\lambda\}_{\lambda\vdash n}$, and the \emph{Schur basis} $\{s_\lambda\}_{\lambda\vdash n}$. See \cite{MR3443860} for their definitions. We will primarily use the complete homogeneous basis and the Schur basis, and write $h_d=h_{(d)}$ for all one-row partitions $(d)\vdash d$. Note that $s_{(d)} = h_d$.

A symmetric function $F\in\Lambda$ is \emph{Schur-positive} if 
$F = \sum_{\lambda}c_\lambda(q)\cdot s_\lambda$ where $c_\lambda(q)\in\mathbb{R}_{\ge 0}[q]$. Similarly, a doubly symmetric function $F^\prime\in\Lambda\otimes_{\CC(q)}\Lambda$ is \emph{Schur-positive} if $F^\prime = \sum_{\lambda,\mu}c_{\lambda,\mu}(q)\cdot s_\lambda\otimes s_\mu$ where $c_{\lambda,\mu}\in\mathbb{R}_{\ge 0}[q]$.
For $F,G\in\Lambda$ (resp. $F,G\in\Lambda\otimes_{\CC(q)}\Lambda$), we write $F\le G$ if and only if $G-F$ is Schur-positive.

We define the \emph{truncation operator} $\{-\}_P$ where $P$ is a condition as follows. For $F=\sum_{\lambda}c_\lambda(q)\cdot s_\lambda\in\Lambda$, we define $\{F\}_P\in\Lambda$ by
\[\{F\}_P\coloneqq\sum_{\text{$\lambda$ satisfies $P$}}c_\lambda(q)\cdot s_\lambda.\]
For $F^\prime=\sum_{\lambda,\mu}c_{\lambda,\mu}(q)\cdot s_\lambda\otimes s_\mu\in\Lambda\otimes_{\CC(q)}\Lambda$, we define $\{F^\prime\}_P\in\Lambda\otimes_{\CC(q)}\Lambda$ by
\[\{F^\prime\}_P\coloneqq\sum_{\text{both $\lambda$ and $\mu$ satisfy $P$}}c_{\lambda,\mu}(q)\cdot s_\lambda\otimes s_\mu.\]
For instance, we have
\[\{F^\prime\}_{\lambda_1\le L} = \sum_{\substack{\lambda_1^{(1)}\le L \\ \lambda_1^{(2)}\le L}}c_{\lambda^{(1)},\lambda^{(2)}}(q)\cdot s_{\lambda^{(1)}}\otimes s_{\lambda^{(2)}}\]
summing over all the pairs of partitions $\lambda^{(1)},\lambda^{(2)}$ such that $\lambda_1^{(1)}\le L$ and $\lambda_1^{(2)}\le L$.

Symmetric functions, especially Schur functions $s_\lambda$, are tightly related to the representation theory of symmetric groups. Let $\symm_n$ be the symmetric group consisting of permutations on $[n]$, and consider the representation theory of $\symm_n$ over the ground field $\CC$. It is well-known that all $\symm_n$-irreducibles $\{V^\lambda\}_{\lambda\vdash n}$ are exactly indexed by $\{\lambda\vdash n\}$, and they are called \emph{Specht modules}. For an $\symm_n$-module $V=\bigoplus_{\lambda\vdash n}c_\lambda V^\lambda$ where $c_\lambda\in\ZZ_{\ge 0}$, its \emph{Frobenius image} $\Frob(V)\in\Lambda_n$ is a Schur-positive symmetric function given by
\[\Frob(V)\coloneqq\sum_{\lambda\vdash n}c_\lambda\cdot s_\lambda.\]
Furthermore, for a graded $\symm_n$-module \[V=\bigoplus_{d=0}^m V_d,\] we define its \emph{graded Frobenius image} $\grFrob(V;q)\in\Lambda$ by
\[\grFrob(V;q)\coloneqq\sum_{d=0}^mq^d\cdot\Frob(V_d).\]

Now, we consider the representation theory of $\symm_n\times\symm_m$. All $\symm_n\times\symm_m$-irreducibles are \[\{V^\lambda\otimes V^\mu\}_{\lambda\vdash n,\mu\vdash m}\]
where the $\symm_n\times\symm_m$-action is diagonal, i.e. $(g,h)\cdot (v_1\otimes v_2) =(g\cdot v_1)\otimes( h\cdot v_2)$ for $g\in\symm_n$ and $h\in\symm_m$. Analogously, we define the \emph{Frobenius image} of an $\symm_n\times\symm_m$-module $V=\bigoplus_{\lambda\vdash n,\mu\vdash m}c_{\lambda,\mu}V^\lambda\otimes V^\mu$ by
\[\Frob(V)\coloneqq\sum_{\lambda\vdash n,\mu\vdash m}c_{\lambda,\mu}\cdot s_\lambda\otimes s_\mu \in \Lambda\otimes_{\CC(q)}\Lambda,\]
and define the \emph{graded Frobenius image} of a graded $\symm_n\times\symm_m$-module similarly to the graded $\symm_n$-module case above.

A standard result about Frobenius images is that: For an $\symm_a$-module $V$ and an $\symm_b$-module $W$, we have that
\[\Frob(\ind_{\symm_a\times\symm_b}^{\symm_{a+b}}(V\otimes W)) = \Frob(V)\cdot\Frob(W)\]
where we embed $\symm_a\times\symm_b$ into $\symm_{a+b}$ by letting $\symm_a$ permute $\{1,\dots,a\}$ and $\symm_b$ permute $\{b+1,\dots,a+b\}$.

We will frequently use the following specific doubly symmetric function.
\begin{definition}\label{def:specific-sym-func}
    For $0\le d\le\min\{m,n\}$, we define $SF_d\in\Lambda_n\otimes\Lambda_m$ by
    \[SF_d\coloneqq\sum_{\mu\vdash d}(s_\mu\cdot h_{n-d})\otimes(s_\mu\cdot h_{m-d}).\]
\end{definition}

The following result also appears in the proof of \cite[Proposition 3.5]{liu2025extensionviennotsshadowrook}.

\begin{lemma}\label{lem:ungraded-frob}
    For $0\le r\le\min\{m,n\}$, we have
    \[\Frob(R(\ZZZ_{n,m,r}))=\Frob(\CC[\ZZZ_{n,m,r}]) = SF_r.\]
\end{lemma}
\begin{proof}
    The first equal sign arises from $R(\ZZZ_{n,m,r})\cong\CC[\ZZZ_{n,m,r}]$. It remains to show the second equal sign. Consider $\symm_{r}\times\symm_{n-r}$ (resp. $\symm_r\times\symm_{m-r}$) as a subgroup of $\symm_n$ (resp. $\symm_m$) by letting $\symm_r$ and $\symm_{n-r}$ (resp. $\symm_{m-r}$) separately act on $\{1,\dots,r\}$ and $\{r+1,\dots,n\}$ (resp. $\{r+1,\dots,m\}$). Thus, we embed $(\symm_r\times\symm_{n-r})\times(\symm_{r}\times\symm_{m-r})$ into $\symm_n\times\symm_m$ as a subgroup. Note that $(\symm_r\times\symm_{n-r})\times(\symm_{r}\times\symm_{m-r})$ acts on $\CC[\symm_r]$ by
    \[((g_1,h_1),(g_2,h_2))\cdot g = g_1 g g_2^{-1}\]
    which is restricted to the $\symm_r\times\symm_r$-action. Artin-Wedderburn Theorem gives this $\symm_r\times\symm_r$-module structure by
    \begin{align}\label{eq:artin-wed}
    \CC[\symm_r]\cong\bigoplus_{\mu\vdash r}V^\mu\otimes V^\mu
    \end{align}
    where $\symm_r\times\symm_r$ acts on each $V^\mu\otimes V^\mu$ by
    $(g_1,g_2)\cdot (v_1\otimes v_2) = (g_1\cdot v_1)\otimes(g_2\cdot v_2)$. Then we have that, as $\symm_n\times\symm_m$-modules, 
    \begin{align*}
        \CC[\ZZZ_{n,m,r}] &\cong \ind_{(\symm_r\times\symm_{n-r})\times(\symm_{r}\times\symm_{m-r})}^{\symm_n\times\symm_m}\CC[\symm_r] \\ 
        &\cong \ind_{(\symm_r\times\symm_{n-r})\times(\symm_{r}\times\symm_{m-r})}^{\symm_n\times\symm_m}\bigg(\bigoplus_{\mu\vdash r}(V^\mu\otimes V^{(n-r)})\otimes( V^\mu \otimes V^{(m-r)})\bigg) \\
        &\cong \bigoplus_{\mu\vdash r}\ind_{(\symm_r\times\symm_{n-r})\times(\symm_{r}\times\symm_{m-r})}^{\symm_n\times\symm_m}\big((V^\mu\otimes V^{(n-r)})\otimes( V^\mu \otimes V^{(m-r)})\big) \\
        &\cong \bigoplus_{\mu\vdash r}\Big(\ind_{\symm_r\times\symm_{n-r}}^{\symm_n}(V^\mu\otimes V^{(n-r)})\Big)\otimes\Big(\ind_{\symm_r\times\symm_{m-r}}^{\symm_m}(V^\mu\otimes V^{(m-r)})\Big)
    \end{align*}
    where the second isomorphism arises from Equation~\eqref{eq:artin-wed}.
    Therefore, we have
    \[\Frob(\CC[\ZZZ_{n,m,r}]) = \sum_{\mu\vdash r}(s_\mu\cdot h_{n-r})\otimes(s_\mu\cdot h_{m-r}) = SF_r.\]
\end{proof}

The following result will help us assign upper bounds to both $\lambda_1$ and $\mu_1$ for all the $\symm_n\times\symm_m$-irreducibles $V^\lambda\otimes V^\mu$ appearing in some $\symm_n\times\symm_m$-module.
\begin{lemma}\label{lem:ann-length}
    Let $j\le n$ be positive integers and $\lambda\vdash n$ be a partition. Then $\eta_j\cdot V^\lambda = \{0\}$ if and only if $j>\lambda_1$.
\end{lemma}
\begin{proof}
    Note that $\eta_j\cdot V^\lambda = (V^\lambda)^{\symm_j} = \{a\in V^\lambda\,:\text{$w\cdot a =a$ for all $w\in\symm_j$}\}$. Consequently, we have that $\eta_j\cdot V^\lambda =\{0\}$ if and only if the $\symm_j$-irreducible $V^{(j)}$ does not appears in $\mathrm{Res}^{\symm_n}_{\symm_j}V^\lambda$, which is equivalent to $\lambda_1<j$ according to the branching rule (\cite{sagan2013symmetric}).
\end{proof}

\section{Module structure}\label{sec:module-str}
Our main goal is to figure out the graded $\symm_n\times\symm_m$-module structure of $R(\ZZZ_{n,m,r})$. The primary difficulty is that we cannot find an explicit basis of $R(\ZZZ_{n,m,r})$, so we must analyze the graded module structure without previously knowing the dimension of each graded component $R(\ZZZ_{n,m,r})_d$. Subsections~\ref{subsec:module-str-sym-func}, \ref{subsec:defining-dieal}, and \ref{subsec:surj} provide some technical results necessary to our graded character formula in Subsection~\ref{subsec:graded-char}. On a first read, we recommend directly accepting these technical results for their application in Subsection~\ref{subsec:graded-char} while skipping the technical proofs in Subsections~\ref{subsec:module-str-sym-func}, \ref{subsec:defining-dieal}, and \ref{subsec:surj}. Surprisingly, the strategy in \cite{MR4887799} still works in this case, which constructs a triangular diagram like Figure~\ref{fig:graded-surj} and then ``forces inequalities to be equalities''. However, constructing the symmetric function properties in Subsection~\ref{subsec:module-str-sym-func} is much more difficult than its counterpart in \cite[Subsection 5.1]{MR4887799}, as we need to deal with $\Lambda\otimes_{\CC(q)}\Lambda$ instead of $\Lambda$. Additionally, we prove a concise generating set of the defining ideal $\gr\II(\ZZZ_{n,m,r})$ in Proposition~\ref{prop:ideal-equal}, the counterpart of which was not touched by \cite{MR4887799}. Interestingly, this method also works for the defining ideal $\gr\II(\MMM_{n,a})$ in \cite{MR4887799}. 

\subsection{Symmetric function results}\label{subsec:module-str-sym-func}
We first need some symmetric function equalities. For convenience, we use $\langle s_{\lambda^{(1)}}\otimes s_{\lambda^{(2)}}\rangle F$ to denote the coefficient of $s_{\lambda^{(1)}}\otimes s_{\lambda^{(2)}}$ in the Schur expansion of $F\in \Lambda\otimes \Lambda$. We also write $\langle q^i\rangle f$ for the coefficient of $q^i$ in any generating function $f\in\CC[[q]]$. Lemmas~\ref{lem:coef} and \ref{lem:schur-interchange} are also mentioned by \cite[Lemmas 5.1 and 5.2]{liu2025extensionviennotsshadowrook}.

\begin{lemma}\label{lem:coef}
    Let $d,a,b,p,q\in\ZZ_{\ge 0}$ be nonnegative integers and write
    \[F = \sum_{\mu\vdash d}\{s_\mu\cdot h_a\}_{\lambda_1 =p}\otimes\{s_\mu\cdot h_b\}_{\lambda_1 = q}\in\Lambda\otimes_{\CC(q)}\Lambda.\]
    Let $\lambda^{(1)}\vdash a+d$ and $\lambda^{(2)}\vdash b+d$ be partitions.
    
    \noindent If we have that:
    \begin{itemize}
        \item $\lambda_1^{(1)} = p$,
        \item $\lambda_1^{(2)} = q$, and
        \item $\min\{\lambda_i^{(1)},\lambda_i^{(2)}\}\ge\max\{\lambda_{i+1}^{(1)},\lambda_{i+1}^{(2)}\}$ for all $i\ge 1$.
    \end{itemize}
    Then we have
    \[\langle s_{\lambda^{(1)}}\otimes s_{\lambda^{(2)}} \rangle F = 
        \Big\langle q^{d-\sum_{i=1}^\infty \max\{\lambda_{i+1}^{(1)},\lambda_{i+1}^{(2)}\}}\Big\rangle\Bigg(\prod_{i=1}^\infty\Bigg(\sum_{j=0}^{\min\{\lambda_i^{(1)},\lambda_i^{(2)}\} - \max\{\lambda_{i+1}^{(1)},\lambda_{i+1}^{(2)}\}}q^j\Bigg)\Bigg).
    \]
    Otherwise, $\langle s_{\lambda^{(1)}}\otimes s_{\lambda^{(2)}} \rangle F = 
        0$.
\end{lemma}
\begin{proof}
    If $s_{\lambda_1^{(1)}}\otimes s_{s_{\lambda_1^{(2)}}}$ appears in the Schur expansion of $F$, the expression of $F$ implies both $\lambda_1^{(1)}=p$ and $\lambda_1^{(2)}=q$. Additionally, we must have $\min\{\lambda_i^{(1)},\lambda_i^{(2)}\}\ge\max\{\lambda_{i+1}^{(1)},\lambda_{i+1}^{(2)}\}$ for all $i\ge 1$, since both $\lambda^{(1)}$ and $\lambda^{(2)}$ arise from a common partition $\mu$ added a horizontal strip. Calculating $\langle s_{\lambda^{(1)}}\otimes s_{\lambda^{(2)}} \rangle F$ is equivalent to counting all the partitions $\mu\vdash d$ such that: $\max\{\lambda_{i+1}^{(1)},\lambda_{i+1}^{(2)}\}\le\mu_i\le\min\{\lambda_i^{(1)},\lambda_i^{(2)}\}$ for all $i\ge 1$. Therefore, we have that
    \begin{align*}
        \langle s_{\lambda^{(1)}}\otimes s_{\lambda^{(2)}}\rangle F 
        &= \langle q^{d}\rangle\Bigg(\prod_{i=1}^\infty\Bigg(\sum_{j=\max\{\lambda_{i+1}^{(1)},\lambda_{i+1}^{(2)}\}}^{\min\{\lambda_i^{(1)},\lambda_i^{(2)}\}}q^j\Bigg)\Bigg) \\ 
        &=\langle q^{d-\sum_{i=1}^\infty \max\{\lambda_{i+1}^{(1)},\lambda_{i+1}^{(2)}\}}\rangle\Bigg(\prod_{i=1}^\infty\Bigg(\sum_{j=\max\{\lambda_{i+1}^{(1)},\lambda_{i+1}^{(2)}\}}^{\min\{\lambda_i^{(1)},\lambda_i^{(2)}\}}q^{j-\sum_{i=1}^\infty \max\{\lambda_{i+1}^{(1)},\lambda_{i+1}^{(2)}\}}\Bigg)\Bigg) \\ 
        &= \bigg\langle q^{d-\sum_{i=1}^\infty \max\{\lambda_{i+1}^{(1)},\lambda_{i+1}^{(2)}\}}\bigg\rangle\Bigg(\prod_{i=1}^\infty\Bigg(\sum_{j=0}^{\min\{\lambda_i^{(1)},\lambda_i^{(2)}\} - \max\{\lambda_{i+1}^{(1)},\lambda_{i+1}^{(2)}\}}q^j\Bigg)\Bigg).
    \end{align*}
\end{proof}

\begin{lemma}\label{lem:schur-interchange}
    For $d,a,b,p,q\in\ZZ_{\ge 0}$, we have
    \[\sum_{\mu\vdash d}\{s_\mu\cdot h_a\}_{\lambda_1 = p}\otimes\{s_\mu\cdot h_b\}_{\lambda_1 = q} = \sum_{\mu\vdash \big(d+a+b-\max\{p,q\}\big)}\{s_\mu\cdot h_{\max\{p,q\}-b}\}_{\lambda_1 = p}\otimes\{s_\mu\cdot h_{\max\{p,q\}-a}\}_{\lambda_1 = q}.\]
\end{lemma}
\begin{proof}
    We write $F$ (resp. $G$) for the left-hand side (resp. right-hand side). It suffices to show that \[\langle s_{\lambda^{(1)}}\otimes s_{\lambda^{(2)}}\rangle F = \langle s_{\lambda^{(1)}}\otimes s_{\lambda^{(2)}}\rangle G\] for any pair of partitions $\lambda^{(1)}\vdash a+d$ and $\lambda^{(2)}\vdash b+d$.

    Consider three conditions:
    \begin{itemize}
        \item $\lambda_1^{(1)} = p$,
        \item $\lambda_1^{(2)} = q$, and
        \item for all $i\ge 1$ we have $\min\{\lambda_i^{(1)},\lambda_i^{(2)}\}\ge\max\{\lambda_{i+1}^{(1)},\lambda_{i+1}^{(2)}\}$. 
    \end{itemize}
    If one of them is not satisfied, Lemma~\ref{lem:coef} indicates that
    $\langle s_{\lambda^{(1)}}\otimes s_{\lambda^{(2)}}\rangle F = 0 = \langle s_{\lambda^{(1)}}\otimes s_{\lambda^{(2)}}\rangle G$. Now, suppose that all these three conditions hold simultaneously. Then Lemma~\ref{lem:coef} yields equalities
    \[\langle s_{\lambda^{(1)}}\otimes s_{\lambda^{(2)}} \rangle F = 
        \Big\langle q^{d-\sum_{i=1}^\infty \max\{\lambda_{i+1}^{(1)},\lambda_{i+1}^{(2)}\}}\Big\rangle f\]
        and
    \[\langle s_{\lambda^{(1)}}\otimes s_{\lambda^{(2)}} \rangle G = 
        \Big\langle q^{d+a+b-\max\{p,q\}-\sum_{i=1}^\infty \max\{\lambda_{i+1}^{(1)},\lambda_{i+1}^{(2)}\}}\Big\rangle f\]  
    where $f=\prod_{i=1}^\infty\Bigg(\sum_{j=0}^{\min\{\lambda_i^{(1)},\lambda_i^{(2)}\} - \max\{\lambda_{i+1}^{(1)},\lambda_{i+1}^{(2)}\}}q^j\Bigg)$. Therefore, it remains to show that
    \[\Big\langle q^{d-\sum_{i=1}^\infty \max\{\lambda_{i+1}^{(1)},\lambda_{i+1}^{(2)}\}}\Big\rangle f = \Big\langle q^{d+a+b-\max\{p,q\}-\sum_{i=1}^\infty \max\{\lambda_{i+1}^{(1)},\lambda_{i+1}^{(2)}\}}\Big\rangle f.\]
    As $f$ is actually a product of finitely many palindromic polynomials in $q$, $f$ itself is also a palindromic polynomial in $q$. Thus, it suffices to show that
    \[\Big(d-\sum_{i=1}^\infty \max\{\lambda_{i+1}^{(1)},\lambda_{i+1}^{(2)}\}\Big)+\Big(d+a+b-\max\{p,q\}-\sum_{i=1}^\infty \max\{\lambda_{i+1}^{(1)},\lambda_{i+1}^{(2)}\}\Big) = \deg(f),\]
    which is equivalent to
    \begin{align}\label{eq:schur-interchange}2d+a+b = \max\{p,q\} + \sum_{i=1}^\infty\min\{\lambda_i^{(1)},\lambda_i^{(2)}\} + \sum_{i=1}^\infty\max\{\lambda_{i+1}^{(1)},\lambda_{i+1}^{(2)}\}.\end{align}
    Since the right-hand side of Equation~\eqref{eq:schur-interchange} equals
    \begin{align*}
        &\max\{\lambda_1^{(1)},\lambda_1^{(2)}\} + \sum_{i=1}^\infty\min\{\lambda_i^{(1)},\lambda_i^{(2)}\} + \sum_{i=1}^\infty\max\{\lambda_{i+1}^{(1)},\lambda_{i+1}^{(2)}\} \\
        =&\sum_{i=1}^\infty\min\{\lambda_i^{(1)},\lambda_i^{(2)}\} + \sum_{i=1}^\infty\max\{\lambda_{i}^{(1)},\lambda_{i}^{(2)}\} \\
        =&\sum_{i=1}^\infty \Big(\min\{\lambda_i^{(1)},\lambda_i^{(2)}\} + \max\{\lambda_{i}^{(1)},\lambda_{i}^{(2)}\}\Big) \\
        =&\sum_{i=1}^\infty\Big(\lambda_i^{(1)} + \lambda_i^{(2)}\Big) = \lvert\lambda^{(1)}\rvert +\lvert\lambda^{(2)}\rvert = (a+d) +(b+d) =2d+a+b,
    \end{align*}
    which equals the left-hand side of Equation~\eqref{eq:schur-interchange}. Therefore, Equation~\eqref{eq:schur-interchange} holds, concluding our proof.
\end{proof}

The following result gives a refinement of $\Frob(R(\ZZZ_{n,m,r}))$. We will see that this refinement exactly gives $\grFrob(R(\ZZZ_{n,m,r});q)$ in Theorem~\ref{thm:grad-str}.
\begin{corollary}\label{cor:shur-refinement}
For $0\le r\le\min\{m,n\}$, we have that
\begin{align}\label{eq:schur-refinement}
SF_r
=\sum_{d=0}^r\{SF_d - SF_{d-1}\}_{\lambda_1\le n+m-d-r}
\end{align}
with the convention that $SF_{-1} = 0$.
\end{corollary}
\begin{proof}
    If $r=0$, we have nothing to prove. We henceforth suppose that $r\ge 1$.
    The right-hand side of Equation~\eqref{eq:schur-refinement} equals
    \begin{align}
        \nonumber &\sum_{d=0}^r\big(\{SF_d\}_{\lambda_1\le n+m-d-r} - \{SF_{d-1}\}_{\lambda_1\le n+m-d-r}\big) \\ \nonumber
        = &\sum_{d=0}^r\{SF_d\}_{\lambda_1\le n+m-d-r} - \sum_{d=0}^{r-1}\{SF_d\}_{\lambda_1\le n+m-d-1-r} \\ \nonumber
        = &\sum_{d=0}^{r-1}\Big(\{SF_d\}_{\lambda_1\le n+m-d-r} - \{SF_d\}_{\lambda_1\le n+m-d-1-r}\Big) + \{SF_r\}_{\lambda_1\le n+m-2r} \\ \label{exp:interchange}
        = &\sum_{d=0}^{r-1}\sum_{\substack{p,q\ge 0\\ \max\{p,q\}=n+m-d-r}}\sum_{\mu\vdash d}\{s_\mu\cdot h_{n-d}\}_{\lambda_1=p}\otimes\{s_\mu\cdot h_{m-d}\}_{\lambda_1=q} +\{SF_r\}_{\lambda_1\le n+m-2r}.
    \end{align}
    Apply Lemma~\ref{lem:schur-interchange} to the expression~\eqref{exp:interchange} and then deduce that the right-hand side of Equation~\eqref{eq:schur-refinement} equals
    \begin{align}
        \label{exp:schur-prefinal} &\sum_{d=0}^{r-1}\sum_{\substack{p,q\ge 0\\ \max\{p,q\}=n+m-d-r}}\sum_{\mu\vdash n+m-d-\max\{p,q\}}\{s_\mu\cdot h_{\max\{p,q\}-m+d}\}_{\lambda_1=p}\otimes\{s_\mu\cdot h_{\max\{p,q\}-n+d}\}_{\lambda_1=q} \\ \nonumber &+\{SF_r\}_{\lambda_1\le n+m-2r} \\ \label{exp:shur-final}
        = &\sum_{d=0}^{r-1}\sum_{\substack{p,q\ge 0\\ \max\{p,q\}=n+m-d-r}}\sum_{\mu\vdash r}\{s_\mu\cdot h_{n-r}\}_{\lambda_1=p}\otimes\{s_\mu\cdot h_{m-r}\}_{\lambda_1=q} +\{SF_r\}_{\lambda_1\le n+m-2r}
    \end{align}
    where the equality arises from substituting $\max\{p,q\}=n+m-d-r$. This substitution makes sense because it appears under the second summation of the expression~\eqref{exp:schur-prefinal}. Now, we focus on the expression~\eqref{exp:shur-final}. As $d$ ranges over $\{0,1,\dots,r-1\}$, $\max\{p,q\} = n+m-d-r$ ranges over $\{n+m-2r+1,\dots,n+m-r\}$. However, whenever $\{s_\mu\cdot h_{n-r}\}_{\lambda_1=p}\otimes\{s_\mu\cdot h_{m-r}\}_{\lambda_1=q}\neq 0$ in \eqref{exp:shur-final}, we have that $p\le n\le n+m-r$ and $q\le m\le n+m-r$, in which case we have $\max\{p,q\}\le n+m-r$. Therefore, we can discard the restriction $d\ge 0$ under the first summation in \eqref{exp:shur-final}. That is, we can allow $\max\{p,q\}$ to range over $\ZZ_{>n+m-2r}$ in the expression~\eqref{exp:shur-final}. Consequently, the expression~\eqref{exp:shur-final} equals
    \begin{align*}
        \sum_{\substack{p,q\ge 0\\ \max\{p,q\}>n+m-2r}}\sum_{\mu\vdash r}\{s_\mu\cdot h_{n-r}\}_{\lambda_1=p}\otimes\{s_\mu\cdot h_{m-r}\}_{\lambda_1=q} + \{SF_r\}_{\lambda_1\le n+m-2r} =SF_r
    \end{align*}
    which equals the left-hand side of Equation~\eqref{eq:schur-refinement}, so our proof is finished.
\end{proof}

\begin{corollary}\label{cor:schur-sum}
    \begin{align}\label{eq:schur-sum}\sum_{d=0}^{\min\{m,n\}}\{SF_d\}_{\lambda_1\le n+m-2d} + \sum_{d=0}^{\min\{m,n\}-1}\{SF_d\}_{\lambda_1\le n+m-2d-1}
    = \sum_{r = 0}^{\min\{m,n\}}SF_r
    \end{align}
\end{corollary}
\begin{proof}
    Substituting each term of the right-hand side of Equation~\eqref{eq:schur-sum} by Equation~\eqref{eq:schur-refinement}, we see that the right-hand side of Equation~\eqref{eq:schur-sum} equals
    \begin{align*}
        &\sum_{r=0}^{\min\{m,n\}}\sum_{d=0}^r\{SF_d - SF_{d-1}\}_{\lambda_1\le n+m-d-r} 
        =\sum_{0\le d\le r\le\min\{m,n\}}\{SF_d - SF_{d-1}\}_{\lambda_1\le n+m-d-r} \\
        =&\sum_{k=0}^{2\min\{m,n\}}\sum_{\substack{0\le d\le r\le\min\{m,n\}\\d+r=k}} \{SF_d - SF_{d-1}\}_{\lambda_1\le n+m-d-r} \\
        =&\sum_{k=0}^{2\min\{m,n\}}\sum_{\substack{0\le d\le r\le\min\{m,n\}\\d+r=k}} \{SF_d - SF_{d-1}\}_{\lambda_1\le n+m-k} \\
        =&\sum_{k=0}^{2\min\{m,n\}}\sum_{\max\{k-\min\{m,n\},0\}\le d\le \lfloor k/2\rfloor} \{SF_d - SF_{d-1}\}_{\lambda_1\le n+m-k} \\
        =&\sum_{k=0}^{2\min\{m,n\}}\sum_{d\le \lfloor k/2\rfloor} \{SF_d - SF_{d-1}\}_{\lambda_1\le n+m-k}
        =\sum_{k=0}^{2\min\{m,n\}} \{SF_{\lfloor k/2\rfloor}\}_{\lambda_1\le n+m-k}
    \end{align*}
    where the next-to-last equal sign arises from the fact that: for any non-vanishing term $\{SF_d - SF_{d-1}\}_{\lambda_1\le n+m-k}$, Pieri's rule indicates that $n+m-k\ge \max\{m-d,n-d\}=\max\{m,n\}-d$ and hence $d\ge k-\min\{m,n\}$. Note that the last expression \[\sum_{k=0}^{2\min\{m,n\}} \{SF_{\lfloor k/2\rfloor}\}_{\lambda_1\le n+m-k}\] equals the left-hand side of Equation~\eqref{eq:schur-sum} as we can split the summation according to the parity of $k$, so we conclude the proof.
\end{proof}

\subsection{The ideal $\defideal{n}{m}{r}$}\label{subsec:defining-dieal}
We will encounter a lot of quotient rings of $\CC[\xxx_{n\times m}]$ spanned by the following monomials, and we will frequently use them.

\begin{definition}\label{def:monomial}
    For any rook placement $\RRR$ on the $n\times m$ board, we assign a monomial $\mmm(\RRR)\in\CC[\xxx_{n\times m}]$ to $\RRR$ given by
    \[\mmm(\RRR)\coloneqq\prod_{(i,j)\in\RRR}x_{i,j}.\]
\end{definition}

To understand the defining ideal $\gr\II(\ZZZ_{n,m,r})$ of $R(\ZZZ_{n,m,r})$, we define an ideal $\defideal{n}{m}{r}$ generated by some elements of $\gr\II(\ZZZ_{n,m,r})$. In fact, we will see that $\defideal{n}{m}{r} = \gr\II(\ZZZ_{n,m,r})$ in Proposition~\ref{prop:ideal-equal}, but it needs lots of work beforehand, including computing the graded module structure of $R(\ZZZ_{n,m,r})$.

\begin{definition}\label{def:ideal}
    For $0\le r\le\min\{m,n\}$, we let $\defideal{n}{m}{r}\subseteq\CC[\xxx_{n\times m}]$ be the ideal generated by
    \begin{itemize}
        \item the sum $\sum_{i=1}^n\sum_{j=1}^m x_{i,j}$ of all variables,
        \item all products $x_{i,j}\cdot x_{i,j^\prime}$ for $1\le i\le n$ and $1\le j,j^\prime\le m$ of variables in the same row,
        \item all products $x_{i,j}\cdot x_{i^\prime,j}$ for $1\le i,i^\prime\le n$ and $1\le j\le m$ of variables in the same column,
        \item all products $\prod_{k=1}^{n-r+1}\big(\sum_{j=1}^m x_{i_k,j}\big)$ of $n-r+1$ distinct row sums for $1\le i_1<\dots<i_{n-r+1}\le n$,
        \item all products $\prod_{k=1}^{m-r+1}\big(\sum_{i=1}^n x_{i,j_k}\big)$ of $m-r+1$ distinct column sums for $1\le j_1<\dots<j_{m-r+1}\le m$, and
        \item all monomials $\mmm(\RRR)$ for rook placements $\RRR$ on the $n\times m$ board such that $\lvert\RRR\rvert>r$.
    \end{itemize}
\end{definition}

The following spanning set of $\CC[\xxx_{n\times m}]/\defideal{n}{m}{r}$ is immediately from Definition~\ref{def:ideal}.
\begin{lemma}\label{lem:span-pre}
    The family of monomials
    \[\bigsqcup_{d=0}^r \{\mmm(\RRR)\,:\,\RRR\in\ZZZ_{n,m,d}\}\]
    descends to a spanning set of $\CC[\xxx_{n\times m}]/\defideal{n}{m}{r}$. In particular, we have
    \[(\CC[\xxx_{n\times m}]/\defideal{n}{m}{r})_d = \{0\}\]
    for $d>r$.
\end{lemma}
\begin{proof}
    Since any product of two variables in the same row or column (i.e., of the form $x_{i,j}\cdot x_{i,j^\prime}$ or $x_{i,j}\cdot x_{i^\prime,j}$) belongs to $\defideal{n}{m}{r}$, it follows that $\CC[\xxx_{n\times m}]/\defideal{n}{m}{r}$ is spanned by the set
    \[\bigsqcup_{d=0}^{\min\{m,n\}}\{\mmm(\RRR)\,:\,\RRR\in\ZZZ_{n,m,d}\}\]
    of monomials with at most one variable in each row or column. Together with the fact that
    \[\bigsqcup_{r<d\le\min\{m,n\}}\{\mmm(\RRR)\,:\,\RRR\in\ZZZ_{n,m,d}\}\subseteq\defideal{n}{m}{r},\]
    we conclude that the family of monomials
    \[\bigsqcup_{d=0}^r \{\mmm(\RRR)\,:\,\RRR\in\ZZZ_{n,m,d}\}\]
    descends to a spanning set of $\CC[\xxx_{n\times m}]/\defideal{n}{m}{r}$, indicating that $(\CC[\xxx_{n\times m}]/\defideal{n}{m}{r})_d = \{0\}$ for $d>r$.
\end{proof}

We now show that $\defideal{n}{m}{r}\subseteq\gr\II(\ZZZ_{n,m,r})$. However, we will strengthen this result and deduce that $\defideal{n}{m}{r}=\gr\II(\ZZZ_{n,m,r})$ in Proposition~\ref{prop:ideal-equal}.
\begin{lemma}\label{lem:ideal-contain}
    We have $\defideal{n}{m}{r}\subseteq\gr\II(\ZZZ_{n,m,r})$.
\end{lemma}
\begin{proof}
    It suffices to show that all the generators mentioned in Definition~\ref{def:ideal} belong to $\gr\II(\ZZZ_{n,m,r})$.
    Since any rook placement $\RRR\in\ZZZ_{n,m,r}$ has exactly $r$ rooks, we have \[\sum_{i=1}^n \sum_{j=1}^m x_{i,j} - r\in\II(\ZZZ_{n,m,r})\]
    and thus
    \[\sum_{i=1}^n \sum_{j=1}^m x_{i,j} \in \gr\II(\ZZZ_{n,m,r}).\]

    Additionally, the fact that any rook placement has no rooks in the same row or column indicates that $x_{i,j}\cdot x_{i,j^\prime}\in\II(\ZZZ_{n,m,r})$ for $1\le i\le n$ and $1\le j,j^\prime\le m$, and that $x_{i,j}\cdot x_{i^\prime,j}\in\II(\ZZZ_{n,m,r})$ for all $1\le i,i^\prime\le n$ and $1\le j\le m$. Therefore, we also have $x_{i,j}\cdot x_{i,j^\prime}\in\gr\II(\ZZZ_{n,m,r})$ and $x_{i,j}\cdot x_{i^\prime,j}\in\gr\II(\ZZZ_{n,m,r})$.

    Furthermore, for any rook placement $\RRR\in\ZZZ_{n,m,r}$ and any $n-r+1$ distinct rows indexed by $1\le i_1<\dots<i_{n-r+1}\le n$, $\RRR$ must meet some of these rows since $\lvert\RRR\rvert+(n-r+1) =n+1>n$. Consequently, we have
    \[\prod_{k=1}^{n-r+1}\Bigg(\sum_{j=1}^m x_{i_k,j}-1\Bigg)\in\II(\ZZZ_{n,m,r})\]
    and thus
    \[\prod_{k=1}^{n-r+1}\Bigg(\sum_{j=1}^m x_{i_k,j}\Bigg)\in\gr\II(\ZZZ_{n,m,r}).\]
    Similar reasoning for any $m-r+1$ distinct rows with indices $1\le j_1<\dots<j_{m-r+1}\le m$ yields
    \[\prod_{k=1}^{m-r+1}\Bigg(\sum_{i=1}^n x_{i,j_k}\Bigg)\in\gr\II(\ZZZ_{n,m,r}).\]

    Finally, for any rook placement $\RRR$ on the $n\times m$ board with more than $r$ rooks and any $\RRR^\prime\in\ZZZ_{n,r}$, it is clear that $\RRR\not\subseteq\RRR^\prime$ and hence $\mmm(\RRR)(\RRR^\prime) = 0$. Therefore, we deduce that $\mmm(\RRR)\in\II(\ZZZ_{n,m,r})$ and then $\mmm(\RRR)\in\gr\II(\ZZZ_{n,m,r})$.
\end{proof}

Lemmas~\ref{lem:span-pre} and \ref{lem:ideal-contain} immediately indicates a spanning set of $R(\ZZZ_{n,m,r})$.
\begin{corollary}\label{cor:span}
    The family of monomials
    \[\bigsqcup_{d=0}^r \{\mmm(\RRR)\,:\,\RRR\in\ZZZ_{n,m,d}\}\]
    descends to a spanning set of $R(\ZZZ_{n,m,r})$. In particular, we have
    \[R(\ZZZ_{n,m,r})_d = \{0\}\]
    for $d>r$.
\end{corollary}
\begin{proof}
    Lemma~\ref{lem:ideal-contain} yields a graded surjective $\CC$-linear map
    \[\CC[\xxx_{n\times m}]/\defideal{n}{m}{r}\twoheadrightarrow R(\ZZZ_{n,m,r}),\]
    which assigns all properties of $\CC[\xxx_{n\times m}]/\defideal{n}{m}{r}$ mentioned in Lemma~\ref{lem:span-pre} to $R(\ZZZ_{n,m,r})$.
\end{proof}

We want to apply Lemma~\ref{lem:ann-length} to each $R(\ZZZ_{n,m,r})_d$, so we want to show $R(\ZZZ_{n,m,r})_d$ is annihilated by some operators $\eta_j\otimes 1$ and $1\otimes\eta_{j^\prime}$, which will be done in Corollary~\ref{cor:ideal-annihilation}. To show Corollary~\ref{cor:ideal-annihilation}, we need to consider two new ideals $J_{n,m,r}, K_{n,m,r}\subseteq I_{n,m}^{(r)}$ allowing us to do induction in the technical result Lemma~\ref{lem:ideal-induction} before Corollary~\ref{cor:ideal-annihilation}.

\begin{definition}\label{def:intermidiate-ideal}
Let $J_{n,m,r}\subseteq\CC[\xxx_{n\times m}]$ be the ideal generated by
    \begin{itemize}
        \item all products $x_{i,j}\cdot x_{i,j^{\prime}}$ for $1\le i\le n$ and $1\le j,j^\prime\le m$ of variables in the same row, and
        \item all products $\prod_{k=1}^{m-r+1}\big(\sum_{i=1}^n x_{i,j_k}\big)$ of $m-r+1$ distinct column sums for $1\le j_1<\dots<j_{m-r+1}\le m$.
    \end{itemize}
Similarly, let $K_{n,m,r}\subseteq\CC[\xxx_{n\times m}]$ be the ideal generated by
    \begin{itemize}
        \item all products $x_{i,j}\cdot x_{i^\prime,j}$ for $1\le i,i^\prime\le n$ and $1\le j\le m$ of variables in the same row, and
        \item all products $\prod_{k=1}^{n-r+1}\big(\sum_{j=1}^n x_{i_k,j}\big)$ of $n-r+1$ distinct row sums for $1\le i_1<\dots<i_{n-r+1}\le n$.
    \end{itemize}
\end{definition}

Lemmas~\ref{lem:generator-slight-modify} and \ref{lem:ideal-induction} are also stated in \cite[Lemmas 3.7 and 3.8]{liu2025extensionviennotsshadowrook}.
\begin{lemma}\label{lem:generator-slight-modify}
    If $m-r<t\le m$ and $b_1,\dots,b_t\in[m]$ are $t$ distinct integers, then
    \[\sum_{\substack{i_1,\dots,i_t\in[n]\\\text{distinct}}}\prod_{k=1}^t x_{i_k,b_k} \in J_{n,m,r}\]
    summing over all the ordered sequences of distinct integers in $[n]$ with length $t$. Similarly, if $n-r<s\le n$ and $a_1,\dots,a_s\in[n]$ are $s$ distinct integers, then
    \[\sum_{\substack{j_1,\dots,j_s\in[m]\\\text{distinct}}}\prod_{k=1}^t x_{a_k,j_k} \in K_{n,m,r}\] summing over all the ordered sequences of distinct integers in $[m]$ with length $s$.
\end{lemma}
\begin{proof}
    As the two results are dual to each other by interchanging the roles of rows and columns of the $n\times m$ board, we only need to show the first result.
    Since any product of the form $x_{i,j}\cdot x_{i,j^\prime}$ belongs to $J_{n,m,r}$, we have that
    \[\sum_{\substack{i_1,\dots,i_t\in[n]\\\text{distinct}}}\prod_{k=1}^t x_{i_k,b_k} \equiv \sum_{i_1,\dots,i_t=1}^n\prod_{k=1}^t x_{i_k,b_k} =\prod_{k=1}^t \sum_{i=1}^n x_{i,b_k} \mod{J_{n,m,r}}.\]
    Note that $t>m-r$. Then the generators of $J_{n,m,r}$ of the form $\prod_{k=1}^{m-r+1}\big(\sum_{i=1}^n x_{i,j_k}\big)$ indicate that $\prod_{k=1}^t \sum_{i=1}^n x_{i,b_k} \in J_{n,m,r}$, completing the proof.
\end{proof}

Then we will prove the most technical result in this subsection, which immediately yields Corollary~\ref{cor:ideal-annihilation}. As we have the $\CC$-algebra identification $\CC[\symm_n\times\symm_m]\cong\CC[\symm_n]\otimes_\CC\CC[\symm_m]$, we identify the operator $\sum_{w\in\symm_p}(w,1)$ (resp. $\sum_{w\in\symm_{p\prime}}(1,w)$) with $\eta_p\otimes 1$ (resp. $1\otimes\eta_{p^\prime}$).

\begin{lemma}\label{lem:ideal-induction}
    Let $p$ and $0\le d\le\min\{m,n\}$ be two integers such that $n+m-d-r<p\le n$. Then, for any rook placement $\RRR\in\ZZZ_{n,m,d}$, we have that
    \begin{align}\label{eq:ideal-induction-1}(\eta_p \otimes 1)\cdot\mmm(\RRR)\in J_{n,m,r}.\end{align}
    Similarly, for $n+m-d-r<p^\prime\le m$, we have that
    \begin{align}\label{eq:ideal-induction-2}(1\otimes\eta_{p^\prime})\cdot\mmm(\RRR)\in K_{n,m,r}\end{align}
\end{lemma}
\begin{proof}
    As the two results \eqref{eq:ideal-induction-1} and \eqref{eq:ideal-induction-2} are essentially the same after interchanging the roles of rows and columns of the $n\times m$ board, it suffices to prove \eqref{eq:ideal-induction-1}.
    
    We prove \eqref{eq:ideal-induction-1} by induction on $n$. For $n=1$, we have nothing to prove. Now suppose that \eqref{eq:ideal-induction-1} holds for any positive integer $n^\prime<n$, and our goal is to show the same result for $n$.

    Write $\mmm(\RRR) = \prod_{k=1}^d x_{a_k,b_k}$ where $b_1,\dots,b_d\in[m]$ are distinct and $1\le a_1<\dots< a_d\le n$. Since $p+d>n+m-d-r+d=n+m-r\ge n$, it follows that $[p]\cap\{a_1,\dots,a_d\}\neq\varnothing$. Therefore, we can choose the maximal $t\in[d]$ such that $a_t\in[p]$.

    \textbf{Claim:} $t>m-r$.

    In fact, the maximality of $t$ indicates that $[p]\cap\{a_{t+1},\dots,a_d\} = \varnothing$, so we have that $[p]\sqcup\{a_{t+1},\dots,a_d\}\subseteq[n]$ and hence $p+d-t\le n$. Consequently, we deduce $t \ge p+d-n > n+m-d-r+d-n=m-r$, finishing the proof of the claim above.

    We go back to prove \eqref{eq:ideal-induction-1}. Note that
    \begin{align*}
        &(\eta_p\otimes 1)\cdot\mmm(\RRR) = (\eta_p\otimes 1)\cdot\prod_{k=1}^d x_{a_k,b_k} = (p-t)!\cdot\Bigg(\prod_{\substack{i_1,\dots,i_t\in[p]\\\text{distinct}}}x_{i_1,b_1}\dots x_{i_t,b_t}\Bigg)\cdot\prod_{k=t+1}^d x_{a_k,b_k} \\
        \equiv & (p-t)!\cdot \Bigg(-\sum_{\substack{\text{distinct $i_1,\dots,i_t\in[n]$,}\\ \{i_1,\dots,i_t\}\not\subseteq[p]}}x_{i_1,b_1}\dots x_{i_t,b_t}\Bigg)\cdot\prod_{k=t+1}^d x_{a_k,b_k} \\
        \equiv & (p-t)!\cdot \Bigg(-\sum_{\substack{\text{distinct $i_1,\dots,i_t\in[n]\setminus\{a_{t+1},\dots,a_d\}$,}\\ \{i_1,\dots,i_t\}\not\subseteq[p]}}x_{i_1,b_1}\dots x_{i_t,b_t}\Bigg)\cdot\prod_{k=t+1}^d x_{a_k,b_k} \mod{J_{n,m,r}}
    \end{align*}
    where the first $\equiv$ uses the claim above, and the second $\equiv$ uses $x_{i,j}\cdot x_{i,j^\prime}\in J_{n,m,r}$. Now, it remains to show that
    \begin{align}\label{eq:ideal-induction}
        \Bigg(\sum_{\substack{\text{distinct $i_1,\dots,i_t\in[n]\setminus\{a_{t+1},\dots,a_d\}$,}\\ \{i_1,\dots,i_t\}\not\subseteq[p]}}x_{i_1,b_1}\dots x_{i_t,b_t}\Bigg)\cdot\prod_{k=t+1}^d x_{a_k,b_k} \in J_{n,m,r}.
    \end{align}
    If the summation in \eqref{eq:ideal-induction} is empty, we have nothing to prove. Otherwise, we decompose the summation into sub-summations with smaller index sets as follows. Note that the summation in \eqref{eq:ideal-induction} requires that $\{i_1,\dots,i_t\}\not\subseteq[p]$, so we assume, without loss of generality, that $\{i_1,\dots,i_t\}\setminus[p]=\{i_{t-s+1},\dots,i_t\}$ where $0<s<t$. (Here, we have $s<t$ because $\{i_1,\dots,i_t\}\cap[p]\neq\varnothing$. Otherwise, $\{i_1,\dots,i_t\}\sqcup[p]\sqcup\{a_{t+1},\dots,a_d\}\subseteq[n]$ indicates that $t+p+d-t\le n$, i.e. $p\le n-d$. However, $p>n+m-d-r\ge n-d$, a contradiction.) Summing over such $i_1,\dots,i_t$ for fixed $i_{t-s+1}=c_{t-s+1},\dots,i_t=c_t$, we obtain a sub-summation of the summation in \eqref{eq:ideal-induction}. It suffices to show \eqref{eq:ideal-induction} after replacing its summation with these sub-summations, since adding them up yields \eqref{eq:ideal-induction} itself. Thus, it remains to show that
    \[\Bigg(\sum_{\substack{i_1,\dots,i_{t-s}\in[p]\setminus\{a_{t+1},\dots,a_d\}\\ \text{distinct}}}x_{i_1,b_1}\dots x_{i_{t-s},b_{t-s}}\Bigg)\cdot\prod_{k=t-s+1}^t x_{c_k,b_k}\cdot\prod_{k=t+1}^d x_{a_k,b_k} \in J_{n,m,r}.\] Recall that the choice of $t$ indicates that $[p]\cap\{a_{t+1},\dots,a_d\}=\varnothing$, so it remains to show that
    \begin{equation}\label{eq:ideal-induction-final}
        \Bigg(\sum_{\substack{i_1,\dots,i_{t-s}\in[p]\\ \text{distinct}}}x_{i_1,b_1}\dots x_{i_{t-s},b_{t-s}}\Bigg)\cdot\prod_{k=t-s+1}^t x_{c_k,b_k}\cdot\prod_{k=t+1}^d x_{a_k,b_k} \in J_{n,m,r}.
    \end{equation}
    Note that this expression possesses the following form:
    \begin{equation}\label{eq:rewrite}c\cdot((\eta_p\otimes 1)\cdot\mmm(\RRR^\prime))\cdot\prod_{k=t-s+1}^t x_{c_k,b_k}\end{equation} where $c\in\CC$, and $\RRR^\prime$ is a rook placement of size $d^\prime = d-s$ on the sub-board $([n]\setminus\{c_{t-s+1},\dots,c_t\})\times([m]\setminus\{b_{t-s+1},\dots,b_t\})$ of the original board $[n]\times[m]$. For the sake of induction, consider restricting Definition~\ref{def:intermidiate-ideal} to a smaller variable set:
    \begin{itemize}
        \item $n^\prime=n-s$, $m^\prime=m-s$, and $r^{\prime}=r-s$;
        \item the restricted variable matrix given by \[\xxx_{n^\prime,m^\prime}^\prime\coloneqq(x_{i,j})_{i\in[n]\setminus\{c_{t-s+1},\dots,c_t\},j\in[m]\setminus\{b_{t-s+1},\dots,b_t\}};\]
        \item the restricted counterpart $J^\prime_{n^\prime,m^\prime,r^\prime}\subseteq\CC[\xxx]$ of $J_{n,m,r}$ generated by
        \begin{itemize}
            \item any product $x_{i,j}\cdot x_{i,j^{\prime}}$ for $i\in[n]\setminus\{c_{t-s+1},\dots,c_t\}$ and $j,j^\prime\in[m]\setminus\{b_{t-s+1},\dots,x_t\}$, and
            \item any product $\prod_{k=1}^{m^\prime-r^\prime+1}(\sum_{i\in[n]\setminus\{c_{t-s+1},\dots,c_t\}}x_{i,j_k})$ for distinct $j_1,\dots,j_{m^\prime-r^\prime+1}\in[m]\setminus\{b_{t-s+1},\dots,b_t\}$.
        \end{itemize}
    \end{itemize}
    Note that $p>n+m-d-r=(n-s)+(m-s)-(d-s)-(r-s)=n^\prime+m^\prime-d^\prime-r^\prime$, and hence the induction assumption reveals that $(\eta_p\otimes 1)\cdot\mmm(\RRR^\prime)\in J_{n^\prime,m^\prime,r^\prime}^\prime$. Therefore, in order to show \eqref{eq:ideal-induction-final} using Expression~\eqref{eq:rewrite}, it suffices to show that
    \[J_{n^\prime,m^\prime,r^\prime}^\prime\cdot\prod_{k=t-s+1}^t x_{c_k,b_k} \subseteq J_{n,m,r}.\] 
    We verify this containment using the generators of $J_{n^\prime,m^\prime,r^\prime}^\prime$ above. For any product $x_{i,j}\cdot x_{i,j^{\prime}}$ for $i\in[n]\setminus\{c_{t-s+1},\dots,c_t\}$ and $j,j^\prime\in[m]\setminus\{b_{t-s+1},\dots,x_t\}$, it is clear that $x_{i,j}\cdot x_{i,j^{\prime}}\cdot\prod_{k=t-s+1}^t x_{c_k,b_k}\in J_{n,m,r}$. For any product $\prod_{k=1}^{m^\prime-r^\prime+1}(\sum_{i\in[n]\setminus\{c_{t-s+1},\dots,c_t\}}x_{i,j_k})$ for distinct $j_1,\dots,j_{m^\prime-r^\prime+1}\in[m]\setminus\{b_{t-s+1},\dots,b_t\}$, $m^\prime-r^\prime+1=(m-s)-(r-s)+1=m-r+1$ indicates that
    \begin{align*}
        &\prod_{k=1}^{m^\prime-r^\prime+1}\Bigg(\sum_{i\in[n]\setminus\{c_{t-s+1},\dots,c_t\}}x_{i,j_k}\Bigg)\cdot\prod_{k=t-s+1}^t x_{c_k,b_k} \\ = &\prod_{k=1}^{m-r+1}\Bigg(\sum_{i\in[n]\setminus\{c_{t-s+1},\dots,c_t\}}x_{i,j_k}\Bigg)\cdot\prod_{k=t-s+1}^t x_{c_k,b_k} \\
        \equiv & \prod_{k=1}^{m-r+1}\Bigg(\sum_{i\in[n]}x_{i,j_k}\Bigg)\cdot\prod_{k=t-s+1}^t x_{c_k,b_k} \equiv 0 \mod{J_{n,m,r}}
    \end{align*}
    which completes the proof.
\end{proof}

\begin{corollary}\label{cor:ideal-annihilation}
    Let $d,r,p,p^\prime$ be integers such that $0\le d\le r\le\min\{m,n\}$, $n+m-d-r<p\le n$, and $n+m-d-r<p^\prime\le m$. For $\RRR\in\ZZZ_{n,m,d}$, we have
    \[(\eta_p\otimes 1)\cdot\mmm(\RRR) \in \defideal{n}{m}{r}\]
    and
    \[(1\otimes\eta_{p^\prime})\cdot\mmm(\RRR)\in \defideal{n}{m}{r}.\]
\end{corollary}
\begin{proof}
    Both results are immediate from Lemma~\ref{lem:ideal-induction} because of the containments $J_{n,m,r},K_{n,m,r}\subseteq\defideal{n}{m}{r}$.
\end{proof}

As a result of Corollary~\ref{cor:ideal-annihilation}, we obtain an upper-bound of $\lambda_1$ and $\mu_1$ for each $\symm_n\times\symm_m$-irreducible $V^\lambda\otimes V^\mu$ appearing in $(\CC[\xxx_{n\times m}]/\defideal{n}{m}{r})_d$ or $R(\ZZZ_{n,m,r})_d$.

\begin{corollary}\label{cor:length-pre}
    Let $d,r$ be two integers such that $0\le d\le r\le\min\{m,n\}$. If an $\symm_n\times\symm_m$-irreducible $V^\lambda\otimes V^\mu$ appears in $(\CC[\xxx_{n\times m}]/\defideal{n}{m}{r})_d$, we have $\lambda_1\le n+m-d-r$ and $\mu_1\le n+m-d-r$.
\end{corollary}
\begin{proof}
    This is immediate from Lemma~\ref{lem:ann-length} and Corollary~\ref{cor:ideal-annihilation}.
\end{proof}

\begin{corollary}\label{cor:length}
    Let $d,r$ be two integers such that $0\le d\le r\le\min\{m,n\}$. If an $\symm_n\times\symm_m$-irreducible $V^\lambda\otimes V^\mu$ appears in $R(\ZZZ_{n,m,r})_d$, we have $\lambda_1\le n+m-d-r$ and $\mu_1\le n+m-d-r$.
\end{corollary}
\begin{proof}
    Lemma~\ref{lem:ideal-contain} yields a graded $\symm_n\times\symm_m$-module surjection
    \[\CC[\xxx_{n\times m}]/\defideal{n}{m}{r}\twoheadrightarrow R(\ZZZ_{n,m,r}),\]
    so the result in Corollary~\ref{cor:length-pre} also holds for $R(\ZZZ_{n,m,r})$.
\end{proof}

\subsection{Surjection of $\symm_n\times\symm_m$-modules}\label{subsec:surj}
In this section, our goal is to construct chains of $\symm_n\times\symm_m$-equivariant surjections $R(\ZZZ_{n,m,r})_d\twoheadrightarrow R(\ZZZ_{n,m,r+1})_d$. We want to initially understand $R(\ZZZ_{n,m,r})$ through $\CC[\xxx_{n\times m}]/\II(\ZZZ_{n,m,r})$, because they are connected by Equations~\eqref{eq:orbit-harmonics-chain-isom} and \eqref{eq:orbit-harmonics-graded-isom}.

First, we state some linear relations (Lemma~\ref{lem:linear-relation-ungraded}) on $\CC[\xxx_{n\times m}]/\II(\ZZZ_{n,m,r})$ as a means to obtain a spanning set (Lemma~\ref{lem:spanning-ungraded}) of $\CC[\xxx_{n\times m}]_{\le d}/(\II(\ZZZ_{n,m,r})\cap\CC[\xxx_{n\times m}]_{\le d})$. We will also provide a linear basis of $\CC[\xxx_{n\times m}]/\II(\ZZZ_{n,m,r})$ in Lemma~\ref{lem:basis-ungraded}.

We assign a partial order $\preceq$ to $\bigsqcup_{d=0}^{\min\{m,n\}}\ZZZ_{n,m,d}$. Write $\RRR\preceq\RRR^\prime$ if and only if $\RRR\subseteq\RRR^\prime$. It is clear that $\preceq$ is generated by the covering relation $\precdot$ where $\RRR\precdot\RRR^\prime$ if and only if $\RRR\subseteq\RRR^\prime$ and $\lvert\RRR^\prime\rvert = \lvert\RRR\rvert + 1$.

\begin{lemma}\label{lem:linear-relation-ungraded}
    Given $0\le d<r\le\min\{m,n\}$ and $\RRR\in\ZZZ_{n,m,d}$, we have
    \[\mmm(\RRR) \equiv \frac{1}{r-d}\sum_{\RRR\precdot\RRR^\prime}\mmm(\RRR^\prime) \mod{\II(\ZZZ_{n,m,r})}.\]
\end{lemma}
\begin{proof}
    Identifying $\CC[\xxx_{n\times m}]/\II(\ZZZ_{n,m,r})$ with $\CC[\ZZZ_{n,m,r}]$ as a $\CC$-vector space, it follows that $\mmm(\tilde{\RRR}) + \II(\ZZZ_{n,m,r})\in\CC[\xxx_{n\times m}]/\II(\ZZZ_{n,m,r})$ is identified with $\mathbf{1}_{\tilde{\RRR}}\in\CC[\ZZZ_{n,m,r}]$ given by
    \[\mathbf{1}_{\tilde{\RRR}}(\tilde{\RRR}^\prime)=\begin{cases}
        1, &\text{if $\tilde{\RRR}\preceq\tilde{\RRR}^\prime$} \\
        0, &\text{otherwise.}
    \end{cases}\]
    for all $\tilde{\RRR}^\prime\in\ZZZ_{n,m,r}$. Therefore, we have
    \begin{align}\label{eq:linear-relation-identify}
        \mmm(\tilde{\RRR})\equiv\sum_{\substack{\tilde{\RRR}\preceq\tilde{\RRR}^\prime \\ \tilde{\RRR}^\prime\in\ZZZ_{n,m,r}}}\mmm(\tilde{\RRR}^\prime) \mod{\II(\ZZZ_{n,m,r})}.
    \end{align}
    Equation~\eqref{eq:linear-relation-identify} indicates that for a fixed $\RRR\in\ZZZ_{n,m,d}$ we have
    \begin{align*}
        &\sum_{\RRR\precdot\RRR^\prime} \mmm(\RRR^\prime) \equiv \sum_{\RRR\precdot\RRR^\prime} \sum_{\substack{\RRR^\prime\preceq\RRR^{\prime\prime} \\ \RRR^{\prime\prime}\in\ZZZ_{n,m,r}}}\mmm(\RRR^{\prime\prime}) \equiv \sum_{\substack{\RRR\preceq\RRR^{\prime\prime} \\ \RRR^{\prime\prime}\in\ZZZ_{n,m,r}}}\Bigg(\sum_{\substack{\RRR\precdot\RRR^{\prime}\preceq\RRR^{\prime\prime}}}1\Bigg)\cdot\mmm(\RRR^{\prime\prime}) \\ = & (r-d)\cdot\sum_{\substack{\RRR\preceq\RRR^{\prime\prime} \\ \RRR^{\prime\prime}\in\ZZZ_{n,m,r}}} \mmm(\RRR^{\prime\prime})
        \equiv (r-d)\cdot\mmm(\RRR) \mod{\II(\ZZZ_{n,m,r})}
    \end{align*}
    where the first and the last signs $\equiv$ arise from Equation~\eqref{eq:linear-relation-identify}.
\end{proof}

\begin{lemma}\label{lem:spanning-ungraded}
    For $0\le d\le r\le\min\{m,n\}$, the family of monomials $\{\mmm(\RRR)\,:\,\RRR\in\ZZZ_{n,m,d}\}$ descends to a spanning set of $\CC[\xxx_{n\times m}]_{\le d}/(\II(\ZZZ_{n,m,r})\cap\CC[\xxx_{n\times m}]_{\le d})$.
\end{lemma}
\begin{proof}
    Since $\II(\ZZZ_{n,m,r})$ contains all the products of the form $x_{i,j}\cdot x_{i,j^\prime}$ or $x_{i,j}\cdot x_{i^\prime,j}$ killing monomials with variables in the same row or column, it follows that the family of monomials
    \[\bigsqcup_{d^\prime =0}^d\{\mmm(\RRR)\,:\,\RRR\in\ZZZ_{n,m,d^\prime}\}\]
    descends to a spanning set of $\CC[\xxx_{n\times m}]_{\le d}/(\II(\ZZZ_{n,m,r})\cap\CC[\xxx_{n\times m}]_{\le d})$. However, iterating Lemma~\ref{lem:linear-relation-ungraded} indicates that for all $\RRR\in\ZZZ_{n,m,d^\prime}$ with $d^\prime<d$ we can expand $\mmm(\RRR) + \II(\ZZZ_{n,m,r})$ as a linear combination of $\{\mmm(\RRR^\prime)+\II(\ZZZ_{n,m,r})\,:\,\RRR^\prime\in\ZZZ_{n,m,d}\}$. Therefore, the set $\{\mmm(\RRR) + \II(\ZZZ_{n,m,r}) \,:\,\RRR\in\ZZZ_{n,m,d}\}$ is enough to span $\CC[\xxx_{n\times m}]_{\le d}/(\II(\ZZZ_{n,m,r})\cap\CC[\xxx_{n\times m}]_{\le d})$.
\end{proof}

\begin{lemma}\label{lem:basis-ungraded}
    For $0\le r\le\min\{m,n\}$, the family of monomials $\{\mmm(\RRR)\,:\,\RRR\in\ZZZ_{n,m,r}\}$ descends to a basis of $\CC[\xxx_{n\times m}]/\II(\ZZZ_{n,m,r})$.
\end{lemma}
\begin{proof}
    This is immediate from the vector space identification
    \[\CC[\xxx_{n\times m}]/\II(\ZZZ_{n,m,r})\cong\CC[\ZZZ_{n,m,r}].\]
\end{proof}

Now we use the basis in Lemma~\ref{lem:basis-ungraded} to construct $\symm_n\times\symm_m$-module surjections
\[\CC[\xxx_{n\times m}]_{\le d}/(\II(\ZZZ_{n,m,r})\cap\CC[\xxx_{n\times m}]_{\le d})\twoheadrightarrow\CC[\xxx_{n\times m}]_{\le d}/(\II(\ZZZ_{n,m,r+1})\cap\CC[\xxx_{n\times m}]_{\le d})\]
and then take quotients to obtain
\[R(\ZZZ_{n,m,r})_d\twoheadrightarrow R(\ZZZ_{n,m,r+1})_d.\]

Let $0\le r<\min\{m,n\}$ be an integer. Consider the linear map \begin{align}\label{eq:ungraded-map}
\varphi_{n,m,r}\,:\,\CC[\xxx_{n\times m}]/\II(\ZZZ_{n,m,r})&\longrightarrow \CC[\xxx_{n\times m}]_{\le r}/(\II(\ZZZ_{n,m,r+1})\cap\CC[\xxx_{n\times m}]_{\le r}) \\
\nonumber \mmm(\RRR)+\II(\ZZZ_{n,m,r})&\longmapsto\mmm(\RRR)+\II(\ZZZ_{n,m,r+1})
\end{align}
for all $\RRR\in\ZZZ_{n,m,r}$. This map is well-defined since \eqref{eq:ungraded-map} assigns an image to each element of the basis of $\CC[\xxx_{n\times m}]/\II(\ZZZ_{n,m,r})$ given by Lemma~\ref{lem:basis-ungraded}. This map is $\symm_n\times\symm_m$-equivariant.

Note the embedding $\CC[\xxx_{n\times m}]_{\le d}/(\II(\ZZZ_{n,m,r})\cap\CC[\xxx_{n\times m}]_{\le d})\subseteq\CC[\xxx_{n\times m}]/\II(\ZZZ_{n,m,r})$. We can consider restricting $\varphi_{n,m,r}$ to $\CC[\xxx_{n\times m}]_{\le d}/(\II(\ZZZ_{n,m,r})\cap\CC[\xxx_{n\times m}]_{\le d})$ and figure out its image.

\begin{lemma}\label{lem:ungraded-surj}
    For $0\le d\le r$,
    $\varphi_{n,m,r}$ is restricted to an $\symm_n\times\symm_m$-module surjection
    \begin{align*}\CC[\xxx_{n\times m}]_{\le d}/(\II(\ZZZ_{n,m,r})\cap\CC[\xxx_{n\times m}]_{\le d}) &\twoheadrightarrow \CC[\xxx_{n\times m}]_{\le d}/(\II(\ZZZ_{n,m,r+1})\cap\CC[\xxx_{n\times m}]_{\le d}) \\
    \mmm(\RRR)+\II(\ZZZ_{n,m,r})&\mapsto (r-d+1)\cdot\mmm(\RRR) +\II(\ZZZ_{n,m,r+1})
    \end{align*}
    for all $\RRR\in\ZZZ_{n,m,r}$.
\end{lemma}
\begin{proof}
    Once we could show
    \begin{align}\label{eq:ungraded-image}
        \varphi_{n,m,r}(\mmm(\RRR) + \II(\ZZZ_{n,m,r})) = (r-d+1)\cdot\mmm(\RRR) +\II(\ZZZ_{n,m,r+1})
    \end{align}
    for all $\RRR\in\ZZZ_{n,m,r}$, then Lemma~\ref{lem:spanning-ungraded} would indicate
    \[\varphi_{n,m,r}\bigg(\CC[\xxx_{n\times m}]_{\le d}/(\II(\ZZZ_{n,m,r})\cap\CC[\xxx_{n\times m}]_{\le d})\bigg) = \CC[\xxx_{n\times m}]_{\le d}/(\II(\ZZZ_{n,m,r+1})\cap\CC[\xxx_{n\times m}]_{\le d}).\]
    Therefore, it suffices to show Equation~\eqref{eq:ungraded-image}.

    We prove Equation~\eqref{eq:ungraded-image} by inverse induction on $d$. For $d=r$, Equation~\eqref{eq:ungraded-image} holds by the definition of $\varphi_{n,m,r}$. Suppose for the sake of induction that Equation~\eqref{eq:ungraded-image} holds for $d=k$ where $0<k\le r$. It remains to show Equation~\eqref{eq:ungraded-image} for $d=k-1$. In fact, for $\RRR\in\ZZZ_{n,m,k-1}$ we have
    \begin{align*}
        &\varphi_{n,m,r}(\mmm(\RRR)+\II(\ZZZ_{n,m,r})) = \varphi_{n,m,r}\Bigg(\frac{1}{r-k+1}\cdot\sum_{\RRR\precdot\RRR^\prime}\mmm(\RRR^\prime)+\II(\ZZZ_{n,m,r})\Bigg) \\
        =&\frac{1}{r-k+1}\cdot\sum_{\RRR\precdot\RRR^\prime}\varphi_{n,m,r}(\mmm(\RRR^\prime)+\II(\ZZZ_{n,m,r})) \\
        =& \frac{1}{r-k+1}\cdot\sum_{\RRR\precdot\RRR^\prime}(r-k+1)\cdot\mmm(\RRR^\prime)+\II(\ZZZ_{n,m,r+1}) \\
        =&\sum_{\RRR\precdot\RRR^\prime}\mmm(\RRR^\prime)+\II(\ZZZ_{n,m,r+1}) = (r-k+2)\cdot\mmm(\RRR) + \II(\ZZZ_{n,m,r+1})
    \end{align*}
    where the first and last equal signs arise from Lemma~\ref{lem:linear-relation-ungraded}, while the third equal sign arises from the induction assumption for $d=k$. We henceforth deduce Equation~\eqref{eq:ungraded-image} for $d=k-1$, finishing our proof.
\end{proof}

We finally obtain $R(\ZZZ_{n,m,r})_d\twoheadrightarrow R(\ZZZ_{n,m,r+1})_d$ as follows.
\begin{corollary}\label{cor:graded-surj}
    For $0\le d\le r<\min\{m,n\}$, we have an $\symm_n\times\symm_m$-module surjection
    \[R(\ZZZ_{n,m,r})_d\twoheadrightarrow R(\ZZZ_{n,m,r+1})_d.\]
\end{corollary}
\begin{proof}
    Recall that $R(\ZZZ_{n,m,r})_{\le d}\cong\CC[\xxx_{n\times m}]_{\le d}/(\II(\ZZZ_{n,m,r})\cap\CC[\xxx_{n\times m}]_{\le d})$ as an $\symm_n\times\symm_m$-module. Therefore, Lemma~\ref{lem:ungraded-surj} induces an $\symm_n\times\symm_m$-module surjection
    \[R(\ZZZ_{n,m,r})_{\le d}\twoheadrightarrow R(\ZZZ_{n,m,r+1})_{\le d}\]
    restricted to
    \[R(\ZZZ_{n,m,r})_{\le d-1}\twoheadrightarrow R(\ZZZ_{n,m,r+1})_{\le d-1}.\]
    Then we take quotients of both sides and obtain
    \[R(\ZZZ_{n,m,r})_{\le d}/R(\ZZZ_{n,m,r})_{\le d-1}\twoheadrightarrow R(\ZZZ_{n,m,r+1})_{\le d}/R(\ZZZ_{n,m,r+1})_{\le d-1}\]
    which is equivalent to
    \[R(\ZZZ_{n,m,r})_d\twoheadrightarrow R(\ZZZ_{n,m,r+1})_d.\]
\end{proof}
\begin{remark}\label{rmk:graded-surj}
Write $q\coloneqq\min\{m,n\}$. Then Corollary~\ref{cor:graded-surj} gives the diagram shown in Figure~\ref{fig:graded-surj}.
\begin{figure}[H]
    \centering
    \begin{scriptsize}
      \begin{displaymath}\displaystyle \scalebox{0.93}{
        \xymatrix{
        &&&&&R_n(\ZZZ_{n,m,q})_{q}\\
        &&&&R(\ZZZ_{n,m,q-1})_{q-1}\ar@{->>}[r] &R(\ZZZ_{n,m,q})_{q-1}\\
        &&& R(\ZZZ_{n,m,q-2})_{q-2}
        \ar@{->>}[r] &R(\ZZZ_{n,m,q-1})_{q-2}\ar@{->>}[r] &R(\ZZZ_{n,m,q})_{q-2}\\
        &&\begin{sideways}$\ddots$\end{sideways}\quad   &\vdots &\vdots &\vdots\\
        &R(\ZZZ_{n,m,1})_{1}\ar@{->>}[r] &\dots\ar@{->>}[r] &R(\ZZZ_{n,m,q-2})_{1}\ar@{->>}[r] &R(\ZZZ_{n,m,q-1})_{1}\ar@{->>}[r] &R(\ZZZ_{n,m,q})_{1}\\
        R(\ZZZ_{n,m,0})_{0}\ar@{->>}[r] &R(\ZZZ_{n,m,1})_{0}\ar@{->>}[r] &\dots\ar@{->>}[r] &R(\ZZZ_{n,m,q-2})_{0}\ar@{->>}[r] &R(\ZZZ_{n,m,q-1})_{0}\ar@{->>}[r] &R(\ZZZ_{n,m,q})_{0}
        }}
      \end{displaymath}
    \end{scriptsize}
    \caption{Chains of surjections}
    \label{fig:graded-surj}
\end{figure}

\end{remark}

\subsection{The graded character of $R(\ZZZ_{n,m,r})$}\label{subsec:graded-char}

We want to determine $\grFrob(R(\ZZZ_{n,m,r});q)$. Roughly speaking, we will use Figure~\ref{fig:graded-surj} and Corollary~\ref{cor:length} to deduce a family of inequalities of symmetric functions, and then use Corollary~\ref{cor:schur-sum} to force all the inequalities to be equalities. These equalities are sufficient to determine $\grFrob(R(\ZZZ_{n,m,r});q)$.

First of all, we consider the direct sum of the modules in each column $R(\ZZZ_{n,m,r})_*$ in Figure~\ref{fig:graded-surj}.
\begin{lemma}\label{lem:col-sum}
    For $0\le r\le\min\{m,n\}$, we have
    \[\sum_{d=0}^r\Frob(R(\ZZZ_{n,m,r})_d) = SF_r.\]
\end{lemma}
\begin{proof}
    Corollary~\ref{cor:span} indicates that $R(\ZZZ_{n,m,r})_d=\{0\}$ for $d>r$. It follows that
    \[\sum_{d=0}^r\Frob(R(\ZZZ_{n,m,r})_d) = \Frob\Bigg(\bigoplus_{d=0}^r R(\ZZZ_{n,m,r})_d\Bigg) = \Frob(R(\ZZZ_{n,m,r})) = \Frob(\CC[\ZZZ_{n,m,r}]),\]
    which, together with Lemma~\ref{lem:ungraded-frob}, finishes the proof.
\end{proof}

Then we construct inequalities using Figure~\ref{fig:graded-surj}.
\begin{lemma}\label{lem:ineq-according-to-parity}
    Suppose $0\le d\le r\le\min\{m,n\}$.
    \begin{itemize}
        \item[(1)] If $d\equiv r \mod{2}$, we have
        \[\Frob(R(\ZZZ_{n,m,r})_d) \le \{\Frob(R(\ZZZ_{n,m,\frac{d+r}{2}})_d)\}_{\lambda_1\le n+m-d-r}.\]
        \item[(2)] If $d\not\equiv r \mod{2}$, we have
        \[\Frob(R(\ZZZ_{n,m,r})_d) \le \{\Frob(R(\ZZZ_{n,m,\frac{d+r-1}{2}})_d)\}_{\lambda_1\le n+m-d-r}.\]
    \end{itemize}
\end{lemma}
We provide a pictorial interpretation for Lemma~\ref{lem:ineq-according-to-parity}. For convenience, we replace modules in Figure~\ref{fig:graded-surj} with dots and specifically denote the module $R(\ZZZ_{n,m,r})_d$ in Lemma~\ref{lem:ineq-according-to-parity} by the symbol $\bullet$. Consider the straight line $L$ passing $\bullet$ with slope $-1$ and mark the northwest-most module in $L$ by a $\blacktriangle$. Then we decorate the module in the same column with $\blacktriangle$ and the same row with $\bullet$ with a $\star$.
\begin{itemize}
    \item[(1)] If $d\equiv r\mod{2}$, $\blacktriangle$ is in the northwestern boundary of the big triangle in Figure~\ref{fig:graded-surj}, which is shown as follows. Note that $\star$ exactly represents the module $R(\ZZZ_{n,m,\frac{d+r}{2}})_d$ in Lemma~\ref{lem:ineq-according-to-parity}.
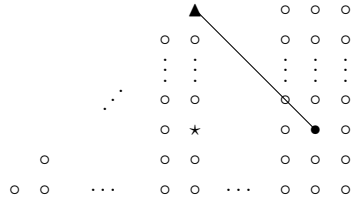
    \begin{scriptsize}
        \begin{center}
            \begin{tikzpicture}[scale = 0.2]
        
            \node (n,0) at (0,0) {$\circ$};
        
            \node (n-2,0) at (2,0) {$\circ$};
            \node (n-2,1) at (2,2) {$\circ$};

            \node (dots) at (6,0) {$\dots$};
            \node (dots) at (6,6) {\begin{sideways} $\ddots$ \end{sideways}};
        
            \node at (10,0) {$\circ$};
            \node at (10,2) {$\circ$};
            \node at (10,4) {$\circ$};
            \node at (10,6) {$\circ$};
            \node at (10,8.5) {$\vdots$};
            \node at (10,10){$\circ$};

            \node at (12,0) {$\circ$};
            \node at (12,2) {$\circ$};
            \node at (12,4) {$\star$};
            \node at (12,6) {$\circ$};
            \node at (12,8.5) {$\vdots$};
            \node at (12,10){$\circ$};
            \node at (12,12){$\blacktriangle$};
        
            \node at (15,0) {$\dots$};
            \node at (18,0) {$\circ$};
            \node at (18,2) {$\circ$};
            \node at (18,4) {$\circ$};
            \node at (18,6) {$\circ$};
            \node at (18,8.5) {$\vdots$};
            \node at (18,10){$\circ$};
            \node at (18,12){$\circ$};
        
            \node at (20,0) {$\circ$};
            \node at (20,2) {$\circ$};
            \node at (20,4) {$\bullet$};
            \node at (20,6) {$\circ$};
            \node at (20,8.5) {$\vdots$};
            \node at (20,10){$\circ$};
            \node at (20,12){$\circ$};

            \node at (22,0) {$\circ$};
            \node at (22,2) {$\circ$};
            \node at (22,4) {$\circ$};
            \node at (22,6) {$\circ$};
            \node at (22,8.5) {$\vdots$};
            \node at (22,10){$\circ$};
            \node at (22,12){$\circ$};

            \draw [-] (20,4) -- (12,12);
                
            \end{tikzpicture}
        \end{center}
    \end{scriptsize}
    
    \item[(2)] If $d\not\equiv r \mod{2}$, $\blacktriangle$ is not in but adjacent to the northwestern boundary of the big triangle in Figure~\ref{fig:graded-surj}, which is shown as follows. Note that $\star$ exactly represents the module $R(\ZZZ_{n,m,\frac{d+r-1}{2}})_d$ in Lemma~\ref{lem:ineq-according-to-parity}.
    
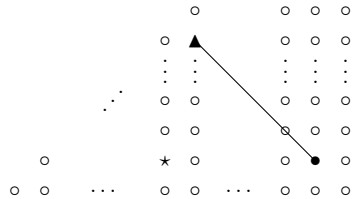
    \begin{scriptsize}
        \begin{center}
            \begin{tikzpicture}[scale = 0.2]
        
            \node (n,0) at (0,0) {$\circ$};
        
            \node (n-2,0) at (2,0) {$\circ$};
            \node (n-2,1) at (2,2) {$\circ$};

            \node (dots) at (6,0) {$\dots$};
            \node (dots) at (6,6) {\begin{sideways} $\ddots$ \end{sideways}};
        
            \node at (10,0) {$\circ$};
            \node at (10,2) {$\star$};
            \node at (10,4) {$\circ$};
            \node at (10,6) {$\circ$};
            \node at (10,8.5) {$\vdots$};
            \node at (10,10){$\circ$};
            
            \node at (12,0) {$\circ$};
            \node at (12,2) {$\circ$};
            \node at (12,4) {$\circ$};
            \node at (12,6) {$\circ$};
            \node at (12,8.5) {$\vdots$};
            \node at (12,10){$\blacktriangle$};
            \node at (12,12){$\circ$};
        
            \node at (15,0) {$\dots$};
            \node at (18,0) {$\circ$};
            \node at (18,2) {$\circ$};
            \node at (18,4) {$\circ$};
            \node at (18,6) {$\circ$};
            \node at (18,8.5) {$\vdots$};
            \node at (18,10){$\circ$};
            \node at (18,12){$\circ$};
        
            \node at (20,0) {$\circ$};
            \node at (20,2) {$\bullet$};
            \node at (20,4) {$\circ$};
            \node at (20,6) {$\circ$};
            \node at (20,8.5) {$\vdots$};
            \node at (20,10){$\circ$};
            \node at (20,12){$\circ$};

            \node at (22,0) {$\circ$};
            \node at (22,2) {$\circ$};
            \node at (22,4) {$\circ$};
            \node at (22,6) {$\circ$};
            \node at (22,8.5) {$\vdots$};
            \node at (22,10){$\circ$};
            \node at (22,12){$\circ$};
        
            \draw [-] (20,2) -- (12,10);
                
            \end{tikzpicture}
        \end{center}
    \end{scriptsize}
\end{itemize}
In both cases, Lemma~\ref{lem:ineq-according-to-parity} states that
\[\Frob(\bullet)\le \{\Frob(\star)\}_{\lambda_1\le n+m-d-r}.\]
This inequality is immediate from Corollaries~\ref{cor:length} and \ref{cor:graded-surj} as follows.
\begin{proof}{\em (of Lemma~\ref{lem:ineq-according-to-parity})}
    \begin{itemize}
        \item[(1)] If $d\equiv r \mod{2}$, we have $\frac{d+r}{2}\in\ZZ$. The fact $d\le r$ indicates $\frac{d+r}{2}\le r$, so we can iteratively use Corollary~\ref{cor:graded-surj} to obtain an $\symm_n\times\symm_m$-module surjection
        \[R(\ZZZ_{n,m,\frac{d+r}{2}})_d\twoheadrightarrow R(\ZZZ_{n,m,r})_d\]
        and hence
        \[\Frob(R(\ZZZ_{n,m,r})_d) \le \Frob(R(\ZZZ_{n,m,\frac{d+r}{2}})_d).\]
        Then Corollary~\ref{cor:length} strengthens this inequality and yields
        \[\Frob(R(\ZZZ_{n,m,r})_d) \le \{\Frob(R(\ZZZ_{n,m,\frac{d+r}{2}})_d)\}_{\lambda_1\le n+m-d-r}.\]
        \item[(2)] If $d\not\equiv r \mod{2}$, we have $\frac{d+r-1}{2}\in\ZZ$. The fact $d\le r$ indicates $\frac{d+r-1}{2}\le r-\frac{1}{2}$ and thus $\frac{d+r-1}{2} < r$, so iterating Corollary~\ref{cor:graded-surj} yields an $\symm_n\times\symm_m$-module surjection
        \[R(\ZZZ_{n,m,\frac{d+r-1}{2}})_d\twoheadrightarrow R(\ZZZ_{n,m,r})_d\]
        and then
        \[\Frob(R(\ZZZ_{n,m,r})_d)\le \Frob(R(\ZZZ_{n,m,\frac{d+r-1}{2}})_d).\]
        Therefore, Corollary~\ref{cor:length} strengthens this inequality to
        \[\Frob(R(\ZZZ_{n,m,r})_d)\le \{\Frob(R(\ZZZ_{n,m,\frac{d+r-1}{2}})_d)\}_{\lambda_1\le n+m-d-r}.\]
    \end{itemize}
\end{proof}

Surprisingly, the inequalities in Lemma~\ref{lem:ineq-according-to-parity} turn out to be equalities, which will be proved in Lemma~\ref{lem:eq-according-to-parity}. Specifically, we will add all the inequalities in Lemma~\ref{lem:ineq-according-to-parity} for $0\le d\le r\le\min\{m,n\}$, obtaining a new inequality. We will demonstrate that this new inequality is actually an equality, thereby forcing all the inequalities to be equalities.
\begin{lemma}\label{lem:eq-according-to-parity}
    Suppose $0\le d\le r\le\min\{m,n\}$.
    \begin{itemize}
        \item[(1)] If $d\equiv r \mod{2}$, we have
        \[\Frob(R(\ZZZ_{n,m,r})_d) = \{\Frob(R(\ZZZ_{n,m,\frac{d+r}{2}})_d)\}_{\lambda_1\le n+m-d-r}.\]
        \item[(2)] If $d\not\equiv r \mod{2}$, we have
        \[\Frob(R(\ZZZ_{n,m,r})_d) = \{\Frob(R(\ZZZ_{n,m,\frac{d+r-1}{2}})_d)\}_{\lambda_1\le n+m-d-r}.\]
    \end{itemize}
\end{lemma}
\begin{proof}
    For $0\le r\le\min\{m,n\}$, we define two graded $\symm_n\times\symm_m$-modules $P_r,Q_r$ by
    \begin{align*}
        P_r&\coloneqq R(\ZZZ_{n,m,r})_r\oplus R(\ZZZ_{n,m,r+1})_{r-1}\oplus R(\ZZZ_{n,m,r+2})_{r-2}\oplus\dots \\
        Q_r&\coloneqq R(\ZZZ_{n,m,r+1})_r\oplus R(\ZZZ_{n,m,r+2})_{r-1}\oplus R(\ZZZ_{n,m,r+3})_{r-2}\oplus\dots.
    \end{align*}
    where the index of each graded component is give by \[(P_r)_d = R(\ZZZ_{n,m,2r-d})_d \quad \text{and} \quad (Q_r)_d=R(\ZZZ_{n,m,2r+1-d})_d.\]
    Intuitively, both $P_r,Q_r$ are the direct sums of all the modules in the same northwest-to-southeast diagonal in Figure~\ref{fig:graded-surj}. For instance, we let $n=5$ and $m=8$, then the modules $P_r$ (resp. $Q_r$) are the direct sums of all the modules in the same red (resp. blue) rectangles in the following diagram.
    
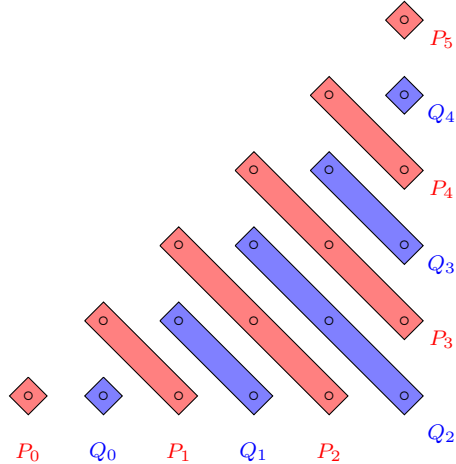
\begin{center}
    \begin{scriptsize}
        \begin{tikzpicture}[scale = 0.5]
                \draw [fill=red!50] (0,0.5) -- (0.5,0) -- (0,-0.5) -- (-0.5,0) -- (0,0.5);
                \draw [fill=blue!50] (2,0.5) -- (2.5,0) -- (2,-0.5) -- (1.5,0) -- (2,0.5);
                \draw [fill=red!50] (1.5,2) -- (2,2.5) -- (4.5,0) -- (4,-0.5) -- (1.5,2);
                \draw [fill=blue!50] (3.5,2) -- (4,2.5) -- (6.5,0) -- (6,-0.5) -- (3.5,2);
                \draw [fill=red!50] (3.5,4) -- (4,4.5) -- (8.5,0) -- (8,-0.5) -- (3.5,4);
                \draw [fill=blue!50] (5.5,4) -- (6,4.5) -- (10.5,0) -- (10,-0.5) -- (5.5,4);
                \draw [fill=red!50] (5.5,6) -- (6,6.5) -- (10.5,2) -- (10,1.5) -- (5.5,6);
                \draw [fill=blue!50] (7.5,6) -- (8,6.5) -- (10.5,4) -- (10,3.5) -- (7.5,6);
                \draw [fill=red!50] (7.5,8) -- (8,8.5) -- (10.5,6) -- (10,5.5) -- (7.5,8);
                \draw [fill=blue!50] (9.5,8) -- (10,8.5) -- (10.5,8) -- (10,7.5) -- (9.5,8);
                \draw [fill=red!50] (9.5,10) -- (10,10.5) -- (10.5,10) -- (10,9.5) -- (9.5,10);
                
                \node at (0,0) {$\circ$};

                \node at (2,0) {$\circ$};
                \node at (2,2) {$\circ$};

                \node at (4,0) {$\circ$};
                \node at (4,2) {$\circ$};
                \node at (4,4) {$\circ$};

                \node at (6,0) {$\circ$};
                \node at (6,2) {$\circ$};
                \node at (6,4) {$\circ$};
                \node at (6,6) {$\circ$};

                \node at (8,0) {$\circ$};
                \node at (8,2) {$\circ$};
                \node at (8,4) {$\circ$};
                \node at (8,6) {$\circ$};
                \node at (8,8) {$\circ$};

                \node at (10,0) {$\circ$};
                \node at (10,2) {$\circ$};
                \node at (10,4) {$\circ$};
                \node at (10,6) {$\circ$};
                \node at (10,8) {$\circ$};
                \node at (10,10) {$\circ$};

                \node at (0,-1.5) {{\color{red} $P_0$}};
                \node at (2,-1.5) {{\color{blue} $Q_{0}$}}; 
                \node at (4,-1.5) {{\color{red} $P_1$}};
                \node at (6,-1.5) {{\color{blue} $Q_1$}};
                \node at (8,-1.5) {{\color{red} $P_2$}};
                \node at (11,-1) {{\color{blue} $Q_2$}};
                \node at (11,1.5) {{\color{red} $P_3$}};
                \node at (11,3.5) {{\color{blue} $Q_3$}};
                \node at (11,5.5) {{\color{red} $P_4$}};
                \node at (11,7.5) {{\color{blue} $Q_4$}};
                \node at (11,9.5) {{\color{red} $P_5$}};
        \end{tikzpicture}
    \end{scriptsize} 
    \end{center}
    In particular, we have $Q_{\min\{m,n\}} = \{0\}$. Additionally, we have $(P_r)_d = (Q_r)_d = \{0\}$ for $d>r$ and for $d<0$. Note that $P_r,Q_r$ for $0\le r\le\min\{m,n\}$ together form a partition of the big triangle, which has another partition separating it into columns. Since Corollary~\ref{cor:span} indicates $R(\ZZZ_{n,m,r})_d =\{0\}$ for $d>r$, the direct sum of modules in the $r$-th column is the whole $R(\ZZZ_{n,m,r})$. Therefore, two different partitions above of the big triangle reveal
    \begin{align}\label{eq:two-partitions}
        &\sum_{r=0}^{\min\{m,n\}}\grFrob(P_r;1) + \sum_{r=0}^{\min\{m,n\}-1}\grFrob(Q_r;1) = \sum_{r=0}^{\min\{m,n\}}\Frob(R(\ZZZ_{n,m,r})) \\ \nonumber
        =&\sum_{r=0}^{\min\{m,n\}}\Frob(\CC[\ZZZ_{n,m,r}]) = \sum_{r=0}^{\min\{m,n\}}SF_r
    \end{align}
    where the last equal sign arises from Lemma~\ref{lem:ungraded-frob}.

    Lemma~\ref{lem:ineq-according-to-parity} gives graded upper bounds on $P_r$ and $Q_r$:
    \begin{align}
        \label{eq:upper-bound-P} \grFrob(P_r;q)&\le\{\grFrob(R(\ZZZ_{n,m,r});q)\}_{\lambda_1\le n+m-2r} \\
        \label{eq:upper-bound-Q} \grFrob(Q_r;q)&\le\{\grFrob(R(\ZZZ_{n,m,r});q)\}_{\lambda_1\le n+m-2r-1}
    \end{align}
    which can be understood using the two diagrams below. The inequality~\eqref{eq:upper-bound-P} compares the red rectangle with the green rectangle shown in the left diagram, while the inequality~\eqref{eq:upper-bound-Q} compares the blue rectangle with the green rectangle shown in the right diagram.
    \begin{center}
    \begin{scriptsize}
        \begin{tikzpicture} [scale = 0.4]
                \draw [fill=blue!50] (27.5,6) -- (28,6.5) -- (30.5,4) -- (30,3.5) -- (27.5,6);
                \draw [fill=green!50] (25.65,6.35) -- (26.35,6.35) -- (26.35,-0.35) -- (25.65,-0.35) -- (25.65,6.35);
                
                \node at (20,0) {$\circ$};

                \node at (22,0) {$\circ$};
                \node at (22,2) {$\circ$};

                \node at (24,0) {$\circ$};
                \node at (24,2) {$\circ$};
                \node at (24,4) {$\circ$};

                \node at (26,0) {$\circ$};
                \node at (26,2) {$\circ$};
                \node at (26,4) {$\circ$};
                \node at (26,6) {$\circ$};

                \node at (28,0) {$\circ$};
                \node at (28,2) {$\circ$};
                \node at (28,4) {$\circ$};
                \node at (28,6) {$\circ$};
                \node at (28,8) {$\circ$};

                \node at (30,0) {$\circ$};
                \node at (30,2) {$\circ$};
                \node at (30,4) {$\circ$};
                \node at (30,6) {$\circ$};
                \node at (30,8) {$\circ$};
                \node at (30,10) {$\circ$};

                \draw [fill=green!50] (5.65,6.35) -- (6.35,6.35) -- (6.35,-0.35) -- (5.65,-0.35) -- (5.65,6.35);
                \draw [fill=red!50] (5.5,6) -- (6,6.5) -- (10.5,2) -- (10,1.5) -- (5.5,6);

                \node at (0,0) {$\circ$};

                \node at (2,0) {$\circ$};
                \node at (2,2) {$\circ$};

                \node at (4,0) {$\circ$};
                \node at (4,2) {$\circ$};
                \node at (4,4) {$\circ$};

                \node at (6,0) {$\circ$};
                \node at (6,2) {$\circ$};
                \node at (6,4) {$\circ$};
                \node at (6,6) {$\circ$};

                \node at (8,0) {$\circ$};
                \node at (8,2) {$\circ$};
                \node at (8,4) {$\circ$};
                \node at (8,6) {$\circ$};
                \node at (8,8) {$\circ$};

                \node at (10,0) {$\circ$};
                \node at (10,2) {$\circ$};
                \node at (10,4) {$\circ$};
                \node at (10,6) {$\circ$};
                \node at (10,8) {$\circ$};
                \node at (10,10) {$\circ$};
        \end{tikzpicture}
    \end{scriptsize} 
    \end{center}
    Substituting $q=1$ respectively in \eqref{eq:upper-bound-P} and \eqref{eq:upper-bound-Q}, we obtain
    \begin{align}
        \label{ineq:ungraded-upper-bound-P} \grFrob(P_r;1)&\le\{\grFrob(R(\ZZZ_{n,m,r});1)\}_{\lambda_1\le n+m-2r} = \{SF_r\}_{\lambda_1 \le n+m-2r} \\
        \label{ineq:ungraded-upper-bound-Q} \grFrob(Q_r;1)&\le\{\grFrob(R(\ZZZ_{n,m,r});1)\}_{\lambda_1\le n+m-2r-1} = \{SF_r\}_{\lambda_1 \le n+m-2r-1}
    \end{align}
    where the last equal signs of both inequalities use Lemma~\ref{lem:ungraded-frob}.

    Summing both sides of inequalities~\eqref{ineq:ungraded-upper-bound-P} for $0\le r\le\min\{m,n\}$ and \eqref{ineq:ungraded-upper-bound-Q} for $0\le r\le \min\{m,n\}-1$, we deduce
    \begin{align}
        \label{ineq:sum-ungraded}\sum_{r=0}^{\min\{m,n\}}\grFrob(P_r;1) + \sum_{r=0}^{\min\{m,n\}-1}\grFrob(Q_r;1) \le \\ \nonumber \sum_{r=0}^{\min\{m,n\}}\{SF_r\}_{\lambda_1\le n+m-2r} + \sum_{r=0}^{\min\{m,n\}-1}\{SF_r\}_{n+m-2r-1}.
    \end{align}
    However, Equation~\eqref{eq:two-partitions} and Corollary~\ref{cor:schur-sum} imply
    \begin{align*}
        &\sum_{r=0}^{\min\{m,n\}}\grFrob(P_r;1) + \sum_{r=0}^{\min\{m,n\}-1}\grFrob(Q_r;1)
        = \sum_{r=0}^{\min\{m,n\}}SF_r \\ = & \sum_{d=0}^{\min\{m,n\}}\{SF_d\}_{\lambda_1\le n+m-2d} + \sum_{d=0}^{\min\{m,n\}-1}\{SF_d\}_{\lambda_1\le n+m-2d-1},
    \end{align*}
    indicating that \eqref{ineq:sum-ungraded} is actually an equality. Therefore, all the inequalities above, including \eqref{eq:upper-bound-P} and \eqref{eq:upper-bound-Q}, are forced to be equalities, concluding the proof.
\end{proof}

Iterating Lemma~\ref{lem:eq-according-to-parity} gives the following result, which will play a crucial role in the induction of the proof of the graded module structure of $R(\ZZZ_{n,m,r})$ in Theorem~\ref{thm:grad-str}.

\begin{lemma}\label{lem:induction-tool-of-module-str}
    For $0\le d\le r\le\min\{m,n\}$, we have
    \[\Frob(R(\ZZZ_{n,m,r})_d) = \{\Frob(R(\ZZZ_{n,m,d})_d)\}_{\lambda_1\le n+m-d-r}.\]
\end{lemma}
\begin{proof}
    We prove this result for any fixed $d$ by induction on $r$. If $r = d$, Lemma~\ref{lem:eq-according-to-parity} indicates
    \[\Frob(R(\ZZZ_{n,m,d})_d) = \{\Frob(R(\ZZZ_{n,m,d})_d\}_{\lambda_1\le n+m-2d}\]
    and hence Lemma~\ref{lem:induction-tool-of-module-str} holds.
    Suppose, for the sake of induction, that there exists an integer $d<k\le\min\{m,n\}$ such that Lemma~\ref{lem:induction-tool-of-module-str} holds for all $r<k$. It remains to show that Lemma~\ref{lem:induction-tool-of-module-str} holds for $r=k$.

        \noindent\textbf{Case (1):} If $d\equiv k\mod{2}$, Lemma~\ref{lem:eq-according-to-parity} indicates
        \begin{align}\label{eq:induction-tool-1}
            \Frob(R(\ZZZ_{n,m,k})_d) = \{\Frob(R(\ZZZ_{n,m,\frac{d+k}{2}})_d)\}_{\lambda_1\le n+m-d-k}.
        \end{align}
        Since $\frac{d+k}{2}<\frac{k+k}{2} = k$, the induction assumption indicates
        \begin{align}\label{eq:indcution-tool-1'}
            \Frob(R(\ZZZ_{n,m,\frac{d+k}{2}})_d) = \{\Frob(R(\ZZZ_{n,m,d})_d) \}_{\lambda_1\le n+m-d-\frac{d+k}{2}}.
        \end{align}
        Equations~\eqref{eq:induction-tool-1} and \eqref{eq:indcution-tool-1'} together indicate
        \begin{align*}
            &\Frob(R(\ZZZ_{n,m,k})_d) = \{\Frob(R(\ZZZ_{n,m,\frac{d+k}{2}})_d)\}_{\lambda_1\le n+m-d-k} \\
            =&\{\{\Frob(R(\ZZZ_{n,m,d})_d) \}_{\lambda_1\le n+m-d-\frac{d+k}{2}}\}_{\lambda_1\le n+m-d-k} \\
            =&\{\Frob(R(\ZZZ_{n,m,d})_d)\}_{\lambda_1\le n+m-d-k}
        \end{align*}
        where the last equal sign arises from the fact that $n+m-d-k<n+m-d-\frac{d+k}{2}$.

        \noindent\textbf{Case (2):} If $d\not\equiv k \mod{2}$, Lemma~\ref{lem:eq-according-to-parity} indicates
        \begin{align}\label{eq:induction-tool-2}
            \Frob(R(\ZZZ_{n,m,k})_d) =\{\Frob(R(\ZZZ_{n,m,\frac{d+k-1}{2}})_d)\}_{\lambda_1\le n+m-d-k}.
        \end{align}
        Since $\frac{d+k-1}{2}<\frac{k+k-1}{2}<k$, the induction assumption indicates
        \begin{align}\label{eq:induction-tooll-2'}
            \Frob(R(\ZZZ_{n,m,\frac{d+k-1}{2}})_d) = \{\Frob(R(\ZZZ_{n,m,d})_d\}_{\lambda_1\le n+m-d-\frac{d+k-1}{2}}.
        \end{align}
        Equations~\eqref{eq:induction-tool-2} and \eqref{eq:induction-tooll-2'} together indicate
        \begin{align*}
            &\Frob(R(\ZZZ_{n,m,k})_d) = \{\Frob(R(\ZZZ_{n,m,\frac{d+k-1}{2}})_d)\}_{\lambda_1\le n+m-d-k} \\
            =&\{\{\Frob(R(\ZZZ_{n,m,d})_d) \}_{\lambda_1\le n+m-d-\frac{d+k-1}{2}}\}_{\lambda_1\le n+m-d-k} \\
            =&\{\Frob(R(\ZZZ_{n,m,d})_d)\}_{\lambda_1\le n+m-d-k}
        \end{align*}
        where the last equal sign arises from the fact that $n+m-d-k<n+m-d-\frac{d+k-1}{2}$.
    
    To sum up, in both cases we have shown that Lemma~\ref{lem:induction-tool-of-module-str} holds for $d=k$, finishing our proof.
\end{proof}
We are ready to prove the graded module structure of $R(\ZZZ_{n,m,r})$.

\begin{theorem}\label{thm:grad-str}
    For $0 \le r \le \min\{m,n\}$, we have
    \[\grFrob(R(\ZZZ_{n,m,r});q) = \sum_{d=0}^r q^d \cdot \{ SF_d - SF_{d-1} \}_{\lambda_1 \le n+m-d-r}\]
    with the convention that $SF_{-1} = 0$.
\end{theorem}
\begin{proof}
    Recall that for $d>r$ we have $R(\ZZZ_{n,m,r})_d=\{0\}$ by Corollary~\ref{cor:span}.
    It suffices to show that for all integers $d,r$ such that $0\le d\le r \le\min\{m,n\}$ we have
    \begin{align}\label{eq:grad-mod-str-induction}
        \Frob(R(\ZZZ_{n,m,r})_d) = \{ SF_d - SF_{d-1} \}_{\lambda_1 \le n+m-d-r}.
    \end{align}
    We prove Equation~\eqref{eq:grad-mod-str-induction} by induction on $d$. If $d=0$, we have nothing to show because $R(\ZZZ_{n,m,r})_0=\CC$. Suppose, for the sake of induction, that there exists an integer $0<k\le\min\{m,n\}$ such that Equation~\eqref{eq:grad-mod-str-induction} holds whenever $d<k$. It remains to show that Equation~\eqref{eq:grad-mod-str-induction} holds whenever $d=k$.

    As Corollary~\ref{cor:span} tells us that $R(\ZZZ_{n,m,r})_d=\{0\}$ for $d>r$, we have
    \begin{align*}
        \Frob(R(\ZZZ_{n,m,k})_k) = \Frob(R(\ZZZ_{n,m,k})) - \sum_{d=0}^{k-1} \Frob(R(\ZZZ_{n,m,k})_d)
    \end{align*}
    where we can apply the substitution $\Frob(R(\ZZZ_{n,m,k})_d) = \{SF_d-SF_{d-1}\}_{\lambda_1\le n+m-d-k}$ to the right-hand side according to the induction assumption. We thereby obtain
    \begin{align}\label{eq:grad-induction-left-endpoint}
        &\Frob(R(\ZZZ_{n,m,k})_k) = \Frob(R(\ZZZ_{n,m,k})) - \sum_{d=0}^{k-1} \{SF_d-SF_{d-1}\}_{\lambda_1\le n+m-d-k} \\ \nonumber
        =&SF_k - \sum_{d=0}^{k-1} \{SF_d-SF_{d-1}\}_{\lambda_1\le n+m-d-k} = \{ SF_k - SF_{k-1} \}_{\lambda_1 \le n+m-2k}
    \end{align}
    where the second equal sign arises from Lemma~\ref{lem:ungraded-frob}, and the last equal sign arises from Corollary~\ref{cor:shur-refinement}. Then Lemma~\ref{lem:induction-tool-of-module-str} and Equation~\eqref{eq:grad-induction-left-endpoint} together indicate that for $k\le r\le\min\{m,n\}$ we have
    \begin{align*}
        &\Frob(R(\ZZZ_{n,m,r})_k) = \{\Frob(R(\ZZZ_{n,m,k})_k)\}_{\lambda_1\le n+m-k-r} \\
        =&\{\{ SF_k - SF_{k-1} \}_{\lambda_1 \le n+m-2k}\}_{\lambda_1\le n+m-k-r} = \{ SF_k - SF_{k-1} \}_{\lambda_1\le n+m-k-r}
    \end{align*}
    where the last equal sign uses the fact that $n+m-k-r\le n+m-2k$. Therefore, Equation~\eqref{eq:grad-mod-str-induction} holds whenever $d=k$, completing our proof.
\end{proof}
\begin{remark}\label{rmk:grad-str-looks-like-trancation-intuitively}
    Intuitively, Theorem~\ref{thm:grad-str} means that each $\twoheadrightarrow$ in Figure~\ref{fig:graded-surj} looks like a truncation
    \[\{SF_d - SF_{d-1}\}_{\lambda_1\le n+m-d-r}\twoheadrightarrow\{SF_{d} - SF_{d-1}\}_{\lambda_1\le n+m-d-r-1}\]
    decreasing the upper bound of $\lambda_1$ by $1$.
\end{remark}
\begin{remark}\label{rmk:grad-str-positivity}
    Theorem~\ref{thm:grad-str} implies that $\{SF_d-SF_{d-1}\}_{\lambda_1\le n+m-d-r}$ is Schur-positive for $0\le d\le r\le \min \{ m , n \}$, which is not combinatorially obvious. We will not only combinatorially prove this Schur-positivity but also give a positive combinatorial expression of $\{SF_d-SF_{d-1}\}_{\lambda_1\le n+m-d-r}$ for $0\le d\le r\le \min \{ m , n \}$ in Section~\ref{sec:sign-cancel}.
\end{remark}

\subsection{Applications}\label{subsec:application-of-module-str}

In this subsection, we introduce some applications of results in Subsection~\ref{subsec:graded-char}. Initially, we strengthen the containment in Lemma~\ref{lem:ideal-contain} into an equality, obtaining a concise generating set of the defining ideal $\gr\II(\ZZZ_{n,m,r})$. We need the following strategic result beforehand.
\begin{lemma}\label{lem:inverse-surj-endpoint}
    For $0\le r\le\min\{m,n\}$, we have
    \[\Frob((\CC[\xxx_{n\times m}]/\defideal{n}{m}{r})_r) \le \Frob(R(\ZZZ_{n,m,r})_r).\]
\end{lemma}
\begin{proof}
    For convenience, we write
    \[L_{r,d}\coloneqq\CC\cdot\{\mathbf{1}_\RRR\,:\,\RRR\in\ZZZ_{n,m,d}\}\subseteq \CC[\ZZZ_{n,m,r}]\]
    for $0\le d\le r\le\min\{m,n\}$, where we have $\mathbf{1}_\RRR(\RRR^\prime)\coloneqq\begin{cases}
        1, &\text{if $\RRR\subseteq\RRR^\prime$}\\
        0, &\text{otherwise.}
    \end{cases}$ for all $\RRR^\prime\in\ZZZ_{n,m,r}$.
    It suffices to find an $\symm_n\times\symm_m$-module surjection
    \begin{equation}\label{eq:highest-surj}
        R(\ZZZ_{n,m,r})_r \twoheadrightarrow (\CC[\xxx_{n\times m}]/\defideal{n}{m}{r})_r.
    \end{equation}
    Recall the $\symm_n\times\symm_m$-module isomorphism
    \[R(\ZZZ_{n,m,r})_r \cong \frac{\CC[\xxx_{n\times m}]_{\le r}/(\II(\ZZZ_{n,m,r})\cap\CC[\xxx_{n\times m}]_{\le r})}{\CC[\xxx_{n\times m}]_{\le r-1}/(\II(\ZZZ_{n,m,r})\cap\CC[\xxx_{n\times m}]_{\le r-1})}.\]
    Furthermore, the $\symm_n\times\symm_m$-module identification
    \[\CC[\xxx_{n\times m}]/\II(\ZZZ_{n,m,r}) \cong \CC[\ZZZ_{n,m,r}]\]
    identifies $\CC[\xxx_{n\times m}]_{\le d}/(\II(\ZZZ_{n,m,r})\cap\CC[\xxx_{n\times m}]_{\le d}) \subseteq \CC[\xxx_{n\times m}]/\II(\ZZZ_{n,m,r})$ with $L_{r,d}\subseteq\CC[\ZZZ_{n,m,r}]$ by Lemma~\ref{lem:spanning-ungraded}. Therefore, we have the $\symm_n\times\symm_m$-module isomorphism
    \[R(\ZZZ_{n,m,r})_r \cong L_{r,r}/L_{r,r-1}.\]
    To show the surjection\eqref{eq:highest-surj}, it suffices to construct an $\symm_n\times\symm_m$-surjection
    \[\tau_{r}\,:\,L_{r,r}/L_{r,r-1} \twoheadrightarrow (\CC[\xxx_{n\times m}]/\defideal{n}{m}{r})_r.\]

    Now we start to construct $\tau_r$. Consider the $\symm_n\times\symm_m$-equivariant map
    \begin{align*}
    \xi_r\,:\,L_{r,r} &\longrightarrow \CC[\xxx_{n\times m}]_r \\
    \mathbf{1}_\RRR &\longmapsto \mmm(\RRR)
    \end{align*}
    for all $\RRR\in\ZZZ_{n,m,r}$. For each $\tilde{\RRR}\in\ZZZ_{n,m,r-1}$, we have
    \begin{align*}
        \xi_r(\mathbf{1}_{\tilde{\RRR}}) &= \xi_r\Bigg(\sum_{\substack{\tilde{\RRR}\subseteq\RRR \\ \RRR\in\ZZZ_{n,m,r}}} \mathbf{1}_\RRR \Bigg) = \sum_{\substack{\tilde{\RRR}\subseteq\RRR \\ \RRR\in\ZZZ_{n,m,r}}} \xi_r(\mathbf{1}_\RRR) = \sum_{\substack{\tilde{\RRR}\subseteq\RRR \\ \RRR\in\ZZZ_{n,m,r}}} \mmm(\RRR) \\
        &=\mmm(\tilde{\RRR}) \cdot \sum_{\substack{\{x=i\}\cap\tilde{\RRR} = \varnothing \\ \{y=j\}\cap\tilde{\RRR} = \varnothing}} x_{i,j}
        \equiv \mmm(\tilde{\RRR}) \cdot \sum_{i=1}^n \sum_{j=1}^m x_{i,j}
        \equiv 0 \mod{\defideal{n}{m}{r}}
    \end{align*}
    where the first $\equiv$ arises from the fact that $x_{i,j}\cdot x_{i,j^\prime}, x_{i,j}\cdot x_{i^\prime,j} \in \defideal{n}{m}{r}$ and the second $\equiv$ uses the fact that $\sum_{i=1}^n \sum_{j=1}^m x_{i,j} \in \defideal{n}{m}{r}$. Thereby, we have
    $\xi_r(L_{r,r-1}) \subseteq \defideal{n}{m}{r}$, so $\xi_r$ descends to an $\symm_n\times\symm_m$-module homomorphism
    \[\tau_r\,:\,L_{r,r}/L_{r,r-1} \twoheadrightarrow (\CC[\xxx_{n\times m}]/\defideal{n}{m}{r})_r\]
    which is surjective by Lemma~\ref{lem:span-pre} as we desire.
\end{proof}

We are ready to calculate the defining ideal $\gr\II(\ZZZ_{n,m,r})$.

\begin{proposition}\label{prop:ideal-equal}
    For $0\le r\le\min\{m,n\}$, we have $\gr\II(\ZZZ_{n,m,r})=\defideal{n}{m}{r}$.
\end{proposition}
\begin{proof}
    Lemma~\ref{lem:ideal-contain} implies the $\symm_n\times\symm_m$-module surjection
    \[\CC[\xxx_{n\times m}]/\defideal{n}{m}{r} \twoheadrightarrow R(\ZZZ_{n,m,r}),\]
    so it suffices to force this surjection to be an isomorphism. That is, we only need to show
    \begin{align}\label{ineq:ideal-inverse}
        \Frob((\CC[\xxx_{n\times m}]/\defideal{n}{m}{r})_d) \le \Frob(R(\ZZZ_{n,m,r})_d)
    \end{align}
    for all $d\ge 0$.

        \noindent\textbf{Case (1):} If $d>r$, Lemma~\ref{lem:span-pre} and Corollary~\ref{cor:span} indicate
        \[\Frob((\CC[\xxx_{n\times m}]/\defideal{n}{m}{r})_d) = \Frob(R(\ZZZ_{n,m,r})_d) = 0\]
        and thus \eqref{ineq:ideal-inverse} holds.
        
        \noindent\textbf{Case (2):} If $0\le d\le r$, Definition~\ref{def:ideal} indicates $(\defideal{n}{m}{d})_d\subseteq(\defideal{n}{m}{r})_d$ and hence
        \begin{align*}
            \Frob((\CC[\xxx_{n\times m}]/\defideal{n}{m}{r})_d)\le\Frob((\CC[\xxx_{n\times m}]/\defideal{n}{m}{d})_d),
        \end{align*}
        which, together with Lemma~\ref{lem:inverse-surj-endpoint}, yields
        \[\Frob((\CC[\xxx_{n\times m}]/\defideal{n}{m}{r})_d)\le\Frob(R(\ZZZ_{n,m,d})_d).\]
        Then we apply Corollary~\ref{cor:length-pre} to the left-hand side of this inequality and thereby obtain
        \begin{align*}
            \Frob((\CC[\xxx_{n\times m}]/\defideal{n}{m}{r})_d)\le\{\Frob(R(\ZZZ_{n,m,d})_d)\}_{\lambda_1 \le n+m-d-r} = \Frob(R(\ZZZ_{n,m,r})_d)
        \end{align*}
        where the equal sign arises from Lemma~\ref{lem:induction-tool-of-module-str}. Therefore, \eqref{ineq:ideal-inverse} holds.
   
    To sum up, the inequality~\eqref{ineq:ideal-inverse} holds in both cases, concluding the proof.
\end{proof}
\begin{remark}
    The strategy to show Proposition~\ref{prop:ideal-equal} also works for the generating set of the defining ideal $\gr\II(\MMM_{n,a})$ in \cite{MR4887799}. See relevant results in Appendix~\ref{sec:appendix}.
\end{remark}


Another application is about module embedding. Liu and Zhu \cite{liu2025extensionviennotsshadowrook} considered the \emph{upper rook placement} locus $\UZ_{n,m,r}$ for $0\le r\le\min\{m,n\}$ given by
\[\UZ_{n,m,r}\coloneqq\bigsqcup_{r^\prime = r}^{\min\{m,n\}} \ZZZ_{n,m,r^\prime}\subseteq\Mat_{n\times m}(\CC).\]
They found the standard monomial basis of $R(\UZ_{n,m,r})$. Consequently, the following embedding~\eqref{eq:embedding-ZZZ-UZ} and identification~\eqref{eq:identification-ZZZ-UZ} may help us to find the standard monomial basis of $R(\ZZZ_{n,m,r})$.

\begin{proposition}\label{prop:UZ-surj}
    For $0\le d<r<\min\{m,n\}$, the map
    \begin{align}\label{eq:embedding-ZZZ-UZ}
        \mathrm{mult}_{d,r}\,:\, R(\UZ_{n,m,r+1})_d &\longrightarrow R(\UZ_{n,m,r})_{d+1} \\ \nonumber
        f + \gr\II(\UZ_{n,m,r+1}) &\longmapsto f\cdot\sum_{i=1}^n \sum_{j=1}^m x_{i,j} + \gr\II(\UZ_{n,m,r})
    \end{align}
    is injective. Furthermore, this embedding gives the $\symm_n\times\symm_m$-module identification
    \begin{equation}\label{eq:identification-ZZZ-UZ}
        R(\UZ_{n,m,r})_{d+1}/\mathrm{mult}_{d,r}(R(\UZ_{n,m,r+1})_d) \cong R(\ZZZ_{n,m,r})_{d+1}.
    \end{equation}
\end{proposition}
\begin{proof}
    We firstly show that the map~\eqref{eq:embedding-ZZZ-UZ} is well-defined.
    Liu and Zhu figured out \cite[Theorem 4.12]{liu2025extensionviennotsshadowrook} that the defining ideal $\gr\II(\UZ_{n,m,r})\subseteq\CC[\xxx_{n\times m}]$ of the orbit harmonics ring $R(\UZ_{n,m,r})$ is generated by
    \begin{itemize}
        \item any product $x_{i,j}\cdot x_{i,j^\prime} \, (1\le i\le n ; \, 1\le j,j^\prime\le m)$ of variables in the same row,
        \item any product $x_{i,j}\cdot x_{i^\prime,j}\, (1\le i,i'\le n ; \, 1\le j\le m)$ of variables in the same column,
        \item any product $\prod_{k=1}^{n-r+1}\big(\sum_{j=1}^m x_{i_k,j}\big)$ of $n-r+1$ distinct row sums for $1\le i_1<\dots<i_{n-r+1}\le n$, and
        \item any product $\prod_{k=1}^{m-r+1}\big(\sum_{i=1}^n x_{i,j_k}\big)$ of $m-r+1$ distinct column sums for $1\le j_1<\dots<j_{m-r+1}\le m$.
    \end{itemize}
    We have a relevant observation:
    
    \textbf{Claim:} $\gr\II(\UZ_{n,m,r+1})\cdot\sum_{i=1}^n\sum_{j=1}^m x_{i,j} \subseteq \gr\II(\UZ_{n,m,r})$.

    In fact, to show this claim, it suffices to show that
    \[\Bigg(\prod_{k=1}^{n-r}\bigg(\sum_{j=1}^m x_{i_k,j}\bigg)\Bigg)\cdot\sum_{i=1}^n\sum_{j=1}^m x_{i,j} \in \gr\II(\UZ_{n,m,r})\]
    for $1\le i_1<\dots<i_{n-r}\le n$, and that
    \[\Bigg(\prod_{k=1}^{m-r}\bigg(\sum_{i=1}^n x_{i,j_k}\bigg)\Bigg)\cdot\sum_{i=1}^n\sum_{j=1}^m x_{i,j} \in \gr\II(\UZ_{n,m,r})\]
    for $1\le j_1<\dots<j_{m-r}\le m$. As their proofs are essentially the same, we only show the first one. The first one follows from
    \begin{align*}
        &\Bigg(\prod_{k=1}^{n-r}\bigg(\sum_{j=1}^m x_{i_k,j}\bigg)\Bigg)\cdot\sum_{i=1}^n\sum_{j=1}^m x_{i,j} = \sum_{i=1}^n\Bigg(\prod_{k=1}^{n-r}\bigg(\sum_{j=1}^m x_{i_k,j}\bigg)\Bigg)\cdot\sum_{j=1}^m x_{i,j} \\
        \equiv &\sum_{i\in[n]\setminus\{i_1,\dots,i_{n-r}\}}\Bigg(\prod_{k=1}^{n-r}\bigg(\sum_{j=1}^m x_{i_k,j}\bigg)\Bigg)\cdot\sum_{j=1}^m x_{i,j} \equiv 0 \mod{\gr\II(\UZ_{n,m,r})}
    \end{align*}
    where the first $\equiv$ arises from the fact that $x_{i,j}\cdot x_{i,j^\prime}\in\gr\II(\UZ_{n,m,r})$, and the second $\equiv$ arises from the fact that $\big(\prod_{k=1}^{n-r}\big(\sum_{j=1}^m x_{i_k,j}\big)\big)\cdot\sum_{j=1}^m x_{i,j}\in\gr\II(\UZ_{n,m,r})$ whenever $i\in[n]\setminus\{i_1,\dots,i_{n-r}\}$. Now the proof of the claim above is finished.

    According to the claim above, the map
    \begin{align*}
        \CC[\xxx_{n \times m}] &\longrightarrow \CC[\xxx_{n \times m}] \\
        f &\longmapsto f\cdot\sum_{i=1}^n\sum_{j=1}^m x_{i,j}
    \end{align*}
    descends to a well-defined $\symm_n\times\symm_m$-module homomorphism
    \begin{align*}
        (\CC[\xxx_{n \times m}]/\gr\II(\UZ_{n,m,r+1}))_d &\longrightarrow (\CC[\xxx_{n \times m}]/\gr\II(\UZ_{n,m,r}))_{d+1} \\
        f &\longmapsto f\cdot\sum_{i=1}^n\sum_{j=1}^m x_{i,j}
    \end{align*}
    which means the map $\mathrm{mult}_{d,r}$ given by \eqref{eq:embedding-ZZZ-UZ} is well-defined.

    The definition of $\mathrm{mult}_{d,r}$ indicates
    \begin{align*}
        R(\UZ_{n,m,r})_{d+1}/\mathrm{mult}_{d,r}(R(\UZ_{n,m,r+1})_d) &\cong \CC[\xxx_{n\times m}]_{d+1}/\Bigg(\gr\II(\UZ_{n,m,r}) + \bigg\langle \sum_{i=1}^n \sum_{j=1}^m x_{i,j} \bigg\rangle\Bigg)_{d+1} \\
        &= \CC[\xxx_{n\times m}]_{d+1}/(\gr\II(\ZZZ_{n,m,r}))_{d+1} = R(\ZZZ_{n,m,r})_{d+1}
    \end{align*}
    where we derive the equal sign from the generating set above of $\gr\II(\UZ_{n,m,r})$ and Proposition~\ref{prop:ideal-equal}. Therefore, the identification~\eqref{eq:identification-ZZZ-UZ} holds.
    
    It remains to show that $\mathrm{mult}_{d,r}$ is injective.
    It follows from the identification~\eqref{eq:identification-ZZZ-UZ} that
    \begin{align*}
        \Frob(R(\UZ_{n,m,r})_{d+1}/\mathrm{mult}_{d,r}(R(\UZ_{n,m,r+1})_d)) = \Frob(R(\ZZZ_{n,m,r})_{d+1}) \\ = \{SF_{d+1} - SF_d \}_{\lambda_1\le n+m-d-r-1}
    \end{align*}
    where the last equal sign arises from Theorem~\ref{thm:grad-str}. That is,
    \[\Frob(R(\UZ_{n,m,r})_{d+1}) - \Frob(\mathrm{mult}_{d,r}(R(\UZ_{n,m,r+1})_d)) = \{SF_{d+1} - SF_d \}_{\lambda_1\le n+m-d-r-1},\]
    indicating
    \begin{align*}
        &\Frob(\mathrm{mult}_{d,r}(R(\UZ_{n,m,r+1})_d)) \\ =& \Frob(R(\UZ_{n,m,r})_{d+1}) - \{SF_{d+1} - SF_d \}_{\lambda_1\le n+m-d-r-1} \\
        =&\{ SF_{d+1} \}_{\lambda_1\le n+m-d-r-1} - \{SF_{d+1} - SF_d \}_{\lambda_1\le n+m-d-r-1} \\ = &\{ SF_d \}_{\lambda_1\le n+m -d-r-1} = \Frob(R(\UZ_{n,m,r+1})_d)
    \end{align*}
    where \cite[Theorem 5.4]{liu2025extensionviennotsshadowrook} yields the second and fourth equal signs.
    Now we obtain \[\Frob(\mathrm{mult}_{d,r}(R(\UZ_{n,m,r+1})_d)) = \Frob(R(\UZ_{n,m,r+1})_d)\]
    which means that the map $\mathrm{mult}_{d,r}$ is injective, completing the proof.
\end{proof}

\section{Sign cancellation}\label{sec:sign-cancel}

While Theorem~\ref{thm:grad-str} gives a signed graded character formula for $R(\ZZZ_{n,m,r})$, we want to cancel the minus signs to obtain a positive combinatorial formula, which will be given by Proposition~\ref{prop:bad-ver-pos-formula}. We will assign a lattice path to each pair of horizontal strips so that we can describe how to embed the negative parts into the positive parts in Theorem~\ref{thm:grad-str}. Similar techniques appear in \cite{zhu2025positivecombinatorialformulaeinvolution} in which the index set consists of single horizontal strips, but our reasoning is more subtle, as each index in our case is a pair of horizontal strips rather than a single horizontal strip.

\begin{remark}\label{rmk:put-partition-in-grid}
    Throughout this section, we put all partitions in the grid $\ZZ_{>0}\times\ZZ_{>0}$ such that the upper left corners of them (including the grid) coincide. Therefore, whenever we say ``the $i$-th column'' we mean the $i$-th column of the grid $\ZZ_{>0}\times\ZZ_{>0}$. We may be concerned about whether a column meets a partition or a horizontal strip.
\end{remark} 

\subsection{Notations}\label{subsec:notation}

We generalize \cite{zhu2025positivecombinatorialformulaeinvolution} the definition of the lattice path associated with a horizontal strip by allowing horizontal steps in a lattice path, and thereby assign a lattice path to each pair of horizontal strips with the same inner partition.

\begin{definition}\label{def:lattice-path}
    Given partitions $\mu,\lambda^{(1)},\lambda^{(2)}$ such that $\lambda^{(i)}/\mu$ is a horizontal strip ($i=1,2$), we associate the triple $(\mu,\lambda^{(1)},\lambda^{(2)})$ with a lattice path $\lp{\mu,\lambda^{(1)},\lambda^{(2)}}$ sequentially given by
    \begin{itemize}
         \item $\lp{\mu,\lambda^{(1)},\lambda^{(2)}}$ starts at the origin $(0,0)$ and extends rightwards with infinitely many steps.
         \item Put $\mu,\lambda^{(1)},\lambda^{(2)}$ in an $\ZZ_{> 0}\times\ZZ_{> 0}$ grid such that the upper left corners of them (including the grid) coincide.
         \item Append steps (i.e., vectors $(1,1),(1,0),(1,-1)$, which are respectively called an \emph{NE}, \emph{HE}, and \emph{SE} step) to  $\lp{\mu,\lambda^{(1)},\lambda^{(2)}}$ as follows:
         \begin{itemize}
             \item If the $i$-th column meets both $\lambda^{(1)}/\mu$ and $\lambda^{(2)}/\mu$, the $i$-th step of $\lp{\mu,\lambda^{(1)},\lambda^{(2)}}$ is the vector $(1,1)$.
             \item If the $i$-th column meets neither $\lambda^{(1)}/\mu$ nor $\lambda^{(2)}/\mu$, the $i$-th step of $\lp{\mu,\lambda^{(1)},\lambda^{(2)}}$ is the vector $(1,-1)$.
             \item Otherwise, the $i$-th step of $\lp{\mu,\lambda^{(1)},\lambda^{(2)}}$ is the vector $(1,0)$.
         \end{itemize}
    \end{itemize}
\end{definition}

\begin{example}\label{ex:lattice-path}
    Let $\mu=(6,3,1),\lambda^{(1)}=(6,5,2),\lambda^{(2)}=(6,4,3)$. We obtain a pair of horizontal strips $\lambda^{(1)}/\mu,\lambda^{(2)}/\mu$ shown as follows, where both horizontal strips are decorated with $\bullet$.
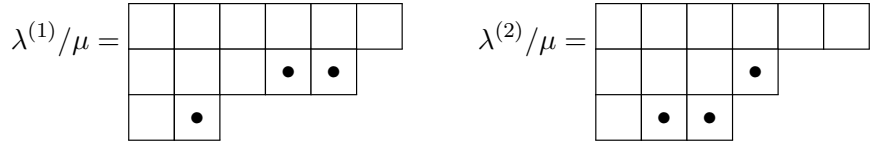
    \[\lambda^{(1)}/\mu =  \ydiagram[*(white) \bullet]
    {6+0,3+2,1+1}
    *[*(white)]{6,3,1} \quad \lambda^{(2)}/\mu =  \ydiagram[*(white) \bullet]
    {6+0,3+1,1+2}
    *[*(white)]{6,3,1}\]
    Then the lattice path $\lp{\mu,\lambda^{(1)},\lambda^{(2)}}$ is shown in the following picture.

    \def\sequence{-1,1,0,1,0,-1,-1}
    \def\gridscale{1}
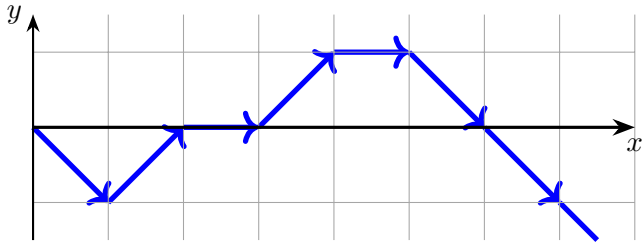
    \begin{center}
    \begin{tikzpicture}[scale=\gridscale]
        \tikzmath{
        integer \x,\y,\miny,\maxy;
        \x=0;\y=0;\miny=0;\maxy=0;
        for \s in \sequence {
        {\draw[->,line width=2pt,color = blue] (\x,\y) -- (\x+1,\y+\s);};
        \x=\x+1;\y=\y+\s;
        if \y>\maxy then {\maxy = \y;};
        if \y<\miny then {\miny = \y;};
        };
        {\draw[line width=2pt,color=blue] (\x,\y) -- (\x+0.5,\y-0.5);
        \draw[gray!70, thin] (0,\miny-0.5) grid (\x+1,\maxy+0.5);
        \draw[very thick, -Stealth] (0,0) -- (\x+1,0) node[below] {$x$};
        \draw[thick, -Stealth] (0,\miny-0.5) -- (0,\maxy+0.5) node[left]  {$y$};};
        }
    \end{tikzpicture}
    \end{center}
\end{example}

We provide an immediate result of the lattice path $\lp{\mu,\lambda^{(1)},\lambda^{(2)}}$.
\begin{lemma}\label{lem:height-of-right-end-of-lattice-path}
    Let $L$ be an integer with $L\ge \max\{\lambda_1^{(1)},\lambda_1^{(2)}\}$.
    The $y$-coordinate of $\lp{\mu,\lambda^{(1)},\lambda^{(2)}}$ at $x=L$ is $\lvert\lambda^{(1)}\rvert\ + \lvert\lambda^{(2)}\rvert -2\lvert\mu\rvert - L$.
\end{lemma}
\begin{proof}
    We write
    \[A = \{i\in[L]\,:\,\text{the $i$-th column meets $\lambda^{(1)}/\mu$}\}\]
    and
    \[B = \{i\in[L]\,:\,\text{the $i$-th column meets $\lambda^{(2)}/\mu$}\}.\]
    Then $\lp{\mu,\lambda^{(1)},\lambda^{(2)}}\mid_{0\le x\le L}$ totally has $\lvert A\cap B\rvert$ NE steps and $L-\lvert A\cup B\rvert$ SE steps.
    Therefore, the $y$-coordinate of $\lp{\mu,\lambda^{(1)},\lambda^{(2)}}$ at $x=L$ equals
    \begin{align*}
        &\lvert A\cap B\rvert - (L-\lvert A\cup B\rvert) = \lvert A\cap B\rvert + \lvert A\cup B\rvert -L 
        = \lvert A \rvert + \lvert B \rvert -L \\ =& \lvert\lambda^{(1)}/\mu\rvert + \lvert\lambda^{(2)}/\mu\rvert - L = \lvert\lambda^{(1)}\rvert\ + \lvert\lambda^{(2)}\rvert -2\lvert\mu\rvert - L.
    \end{align*}
\end{proof}

\begin{definition}\label{def:reflection-pair}
    For a lattice path $\mathcal{L}$ starting at the point $(0,0)$ and consisting of vectors $(1,1)$, $(1,0)$, and $(1,-1)$, a \emph{reflection pair} of $\mathcal{L}$ is a pair of positive integers $(i<j)$ such that the $i$-th (resp. $j$-th) step of $\mathcal{L}$ is an NE (resp. SE) step, and these two steps have the same height (i.e., the starting height of the $i$-th step coincides with the ending height of the $j$-th step) and are horizontally visible from each other without being obstructed by $\mathcal{L}$.
\end{definition}

\begin{example}\label{ex:reflection-pair}
    Consider the lattice path $\lp{\mu,\lambda^{(1)},\lambda^{(2)}}$ in Example~\ref{ex:lattice-path}. All the reflection pairs are $(2,7)$ and $(4,6)$ shown as follows.

    \def\sequence{-1,1,0,1,0,-1,-1}
    \def\gridscale{1}
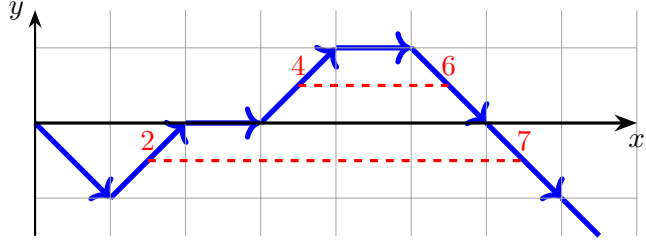
    \begin{center}
    \begin{tikzpicture}[scale=\gridscale]
        \tikzmath{
        integer \x,\y,\miny,\maxy;
        \x=0;\y=0;\miny=0;\maxy=0;
        for \s in \sequence {
        {\draw[->,line width=2pt,color = blue] (\x,\y) -- (\x+1,\y+\s);};
        \x=\x+1;\y=\y+\s;
        if \y>\maxy then {\maxy = \y;};
        if \y<\miny then {\miny = \y;};
        };
        {\draw[line width=2pt,color=blue] (\x,\y) -- (\x+0.5,\y-0.5);
        \draw[gray!70, thin] (0,\miny-0.5) grid (\x+1,\maxy+0.5);
        \draw[very thick, -Stealth] (0,0) -- (\x+1,0) node[below] {$x$};
        \draw[thick, -Stealth] (0,\miny-0.5) -- (0,\maxy+0.5) node[left]  {$y$};};
        }

        \node[above, red] at (1.5,-0.5) {$2$};
        \node[above, red] at (3.5,0.5) {$4$};
        \node[above, red] at (5.5,0.5) {$6$};
        \node[above, red] at (6.5,-0.5) {$7$};

        \draw[red,dashed,very thick] (1.5,-0.5) -- (6.5,-0.5);
        \draw[red,dashed,very thick] (3.5,0.5) -- (5.5,0.5);
    \end{tikzpicture}
    \end{center}
\end{example}

\begin{definition}\label{def:wid}
    For horizontal strips $\lambda^{(1)}/\mu,\lambda^{(2)}/\mu$ with the same inner partition $\mu$, we write
    \[M = \max\{j\in\ZZ_{>0}\,:\,\text{$(i,j)$ is a reflection pair of $\lp{\mu,\lambda^{(1)},\lambda^{(2)}}$ for some $i\in\ZZ_{>0}$}\}\]
    and define the \emph{width} $\wid{\mu,\lambda^{(1)},\lambda^{(2)}}$ of the triple $(\mu,\lambda^{(1)},\lambda^{(2)})$ by
    \[\wid{\mu,\lambda^{(1)},\lambda^{(2)}}\coloneqq\max\{M,\lambda_1^{(1)},\lambda_1^{(2)}\}.\]
\end{definition}

\begin{example}\label{ex:wid}
    For $\mu,\lambda^{(1)},\lambda^{(2)}$ in Example~\ref{ex:lattice-path}, Example~\ref{ex:reflection-pair} indicates \[\wid{\mu,\lambda^{(1)},\lambda^{(2)}} =\max\{M,\lambda_1^{(1)},\lambda_1^{(2)}\} = \max\{7,6,6\} = 7.\]
\end{example}

\subsection{Bad version: a positive combinatorial formula depending on the graded component}\label{subsec:bad-ver}

In this subsection, our goal is to embed the Schur expansion of $\{SF_{d-1}\}_{\lambda_1\le n+m-d-r}$ into the Schur expansion of $\{SF_{d}\}_{\lambda_1\le n+m-d-r}$ for $0\le d\le r\le\min\{m,n\}$, and thereby yield a positive combinatorial formula for $\Frob(R(\ZZZ_{n,m,r})_d) = \{SF_d - SF_{d-1}\}_{\lambda_1 \le n+m-d-r}$ stated in Theorem~\ref{thm:grad-str}.

In the rest of this subsection, we fix an integer $r$ with $0\le r\le \min\{m,n\}$. Let $0\le d\le r$ be an integer. Let $\lambda^{(1)}\vdash n$ and $\lambda^{(2)}\vdash m$ be two partitions with $\lambda^{(i)}_1\le n+m-d-r$ for $i=1,2$. We define two sets of partitions for convenience:
\begin{align}
\label{eq:def-horizontal-stripe-set}
    \hori{d,\lambda^{(1)},\lambda^{(2)}} & \coloneqq \{\mu\vdash d\,:\, \text{both $\lambda^{(1)}/\mu$ and $\lambda^{(2)}/\mu$ are horizontal strips}
    \} \\ \nonumber \\
\label{eq:def-positive-horizontal-stripe-set}
    \horipositive{d,\lambda^{(1)},\lambda^{(2)}} & \coloneqq \Bigg\{\mu\in\hori{d,\lambda^{(1)},\lambda^{(2)}}\,:\!\!\!\!\! \begin{array}{cc}
         & \text{the first $\max\{\lambda_1^{(1)},\lambda_1^{(2)}\}$ steps of $\lp{\mu,\lambda^{(1)},\lambda^{(2)}}$}  \\
         & \text{are weakly higher than the $x$-axis} 
    \end{array} \Bigg\}.
\end{align}
\begin{remark}\label{rmk:check_horipositive}
    Let $\lambda^{(1)}\vdash n,\lambda^{(2)}\vdash m$ with $\lambda^{(i)}_1\le n+m-d-r$ for $i=1,2$. For $\mu\vdash d$ with $\mu\subseteq\lambda^{(i)}$ for $i=1,2$, verifying $\mu\in\horipositive{d,\lambda^{(1)},\lambda^{(2)}}$ is equivalent to verifying the first $n+m-d-r$ steps of $\lp{\mu,\lambda^{(1)},\lambda^{(2)}}$ are weakly higher than the $x$-axis. There are two reasons as follows. On one hand, the fact that $\max\{\lambda^{(1)}_1,\lambda^{(2)}_1\}\le n+m-d-r$ means that checking the first $n+m-d-r$ steps is enough. On the other hand, Lemma~\ref{lem:height-of-right-end-of-lattice-path} indicates that: if $\mu\in\horipositive{d,\lambda^{(1)},\lambda^{(2)}}$, then the $y$-coordinate of $\lp{\mu,\lambda^{(1)},\lambda^{(2)}}$ at $x=n+m-d-r$ equals $n+m-2d-(n+m-d-r) = r-d \ge 0$; but all the steps on the right of $x=\max\{\lambda^{(1)}_1,\lambda^{(2)}_1\}$ are SE steps, so the first $n+m-d-r$ steps are all weakly higher than the $x$-axis.
\end{remark}

\begin{example}\label{ex:set-of-horizontal-stripes}
    Consider the partitions $\mu,\lambda^{(1)},\lambda^{(2)}$ in Example~\ref{ex:lattice-path}. Note that $\lp{\mu,\lambda^{(1)},\lambda^{(2)}}$ goes below the $x$-axis at the first step, so $\mu\in\hori{10,\lambda^{(1)},\lambda^{(2)}}\setminus\horipositive{10,\lambda^{(1)},\lambda^{(2)}}$.
\end{example}

We further assume that $d>0$, and then construct a map (in fact a bijection by Lemma~\ref{lem:bad-version-bijection})
\begin{align}\label{eq:bad-version-bijection}
    \Phi_{d,\lambda^{(1)},\lambda^{(2)}}\,:\,\hori{d,\lambda^{(1)},\lambda^{(2)}}\setminus\horipositive{d,\lambda^{(1)},\lambda^{(2)}} \longrightarrow \hori{d-1,\lambda^{(1)},\lambda^{(2)}}
\end{align}
That is, for each $\mu\in\hori{d,\lambda^{(1)},\lambda^{(2)}}\setminus\horipositive{d,\lambda^{(1)},\lambda^{(2)}}$, we construct its image $\Phi_{d,\lambda^{(1)},\lambda^{(2)}}(\mu)\in\hori{d-1,\lambda^{(1)},\lambda^{(2)}}$ sequentially as follows:
\begin{itemize}
    \item Choose the minimal $x_0\in\ZZ_{\ge 0}$ such that $\lp{\mu,\lambda^{(1)},\lambda^{(2)}}\mid_{0\le x\le n+m-d-r}$ attains its lowest point at $x=x_0$.
    \item Remove the lowest cell from the $x_0$-th column of $\mu$, yielding a new diagram $\nu$ with $d-1$ cells. (We will see that $\mu_1\ge x_0$ from Lemma~\ref{lem:bad-ver-well-defined}, so this removal is well-defined. We will also show that $\nu\vdash d-1$ is indeed a partition from Lemma~\ref{lem:bad-ver-well-defined}.)
    \item Let $\Phi_{d,\lambda^{(1)},\lambda^{(2)}}(\mu) \coloneqq \nu$.
\end{itemize}

\begin{example}\label{ex:bad-ver-bijection}
    We use a concrete example to illustrate how $\Phi_{d,\lambda^{(1)},\lambda^{(2)}}$ works. Let $d=36$ and $\mu = (13,8,6,5,2,2)\vdash d$. Let $\lambda^{(1)} = (13,12,7,6,2,2,1)$ and $\lambda^{(2)} = (13,13,7,5,2,2,2)$. The horizontal strips $\lambda^{(1)}/\mu,\lambda^{(2)}/\mu$ are shown in the diagram below.
    \begin{align*}\lambda^{(1)}/\mu &=  \ydiagram[*(white) \bullet]
    {13+0,8+4,6+1,5+1,2+0,2+0,0+1}
    *[*(white)]{13,8,6,5,2,2} \\ \lambda^{(2)}/\mu &=  \ydiagram[*(white) \bullet]
    {13+0,8+5,6+1,5+0,2+0,2+0,0+2}
    *[*(white)]{13,8,6,5,2,2}\end{align*}
    The lattice path $\lp{\mu,\lambda^{(1)},\lambda^{(2)}}$ is the dashed lower lattice path in the diagram below, which attains its first lowest point at $x=5$ and hence $x_0=5$. Therefore, we remove the lowest cell in the fifth column of $\mu$ and thereby obtain $\nu = \Phi_{d,\lambda^{(1)},\lambda^{(2)}}(\mu)$. A pictorial interpretation of this removal uses the lattice paths, i.e., replacing the $x_0$-th step of $\lp{\mu,\lambda^{(1)},\lambda^{(2)}}$ with an NE step generates $\lp{\nu,\lambda^{(1)},\lambda^{(2)}}$, the higher lattice path in the diagram below.
    
    \def\gridscale{0.7}
     \def\sequence{1,0,-1,-1,1,0,1,-1,1,1,1,1,0,-1,-1,-1}
     \def\oldsequence{-1,0,1,-1,1,1,1,1,0,-1,-1,-1}
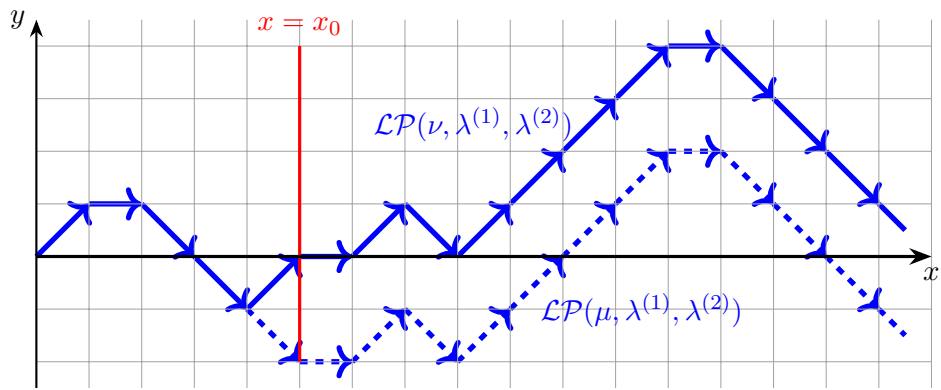
     \begin{center}
     \begin{tikzpicture}[scale=\gridscale]
        \tikzmath{
        integer \x,\y,\miny,\maxy;
        \x=0;\y=0;\miny=0;\maxy=0;
        for \s in \sequence {
        {\draw[->,line width=2pt,color = blue] (\x,\y) -- (\x+1,\y+\s);};
        \x=\x+1;\y=\y+\s;
        if \y>\maxy then {\maxy = \y;};
        if \y<\miny then {\miny = \y;};
        };
        {\draw[line width=2pt,color=blue] (\x,\y) -- (\x+0.5,\y-0.5);};
        \x=4; \y=-1;
        for \s in \oldsequence {
        {\draw[->,dashed,line width=2pt,color = blue] (\x,\y) -- (\x+1,\y+\s);};
        \x=\x+1;\y=\y+\s;
        if \y>\maxy then {\maxy = \y;};
        if \y<\miny then {\miny = \y;};
        };
        {\draw[dashed,line width=2pt,color=blue] (\x,\y) -- (\x+0.5,\y-0.5);
        \draw[gray!70, thin] (0,\miny-0.5) grid (\x+1,\maxy+0.5);
        \draw[red,very thick] (5,\miny) -- (5,\maxy) node[above] {$x=x_0$};
        \draw[very thick, -Stealth] (0,0) -- (\x+1,0) node[below] {$x$};
        \draw[thick, -Stealth] (0,\miny-0.5) -- (0,\maxy+0.5) node[left]  {$y$};};
        }

        \node[blue,below] at (11.5,-0.5) {$\lp{\mu,\lambda^{(1)},\lambda^{(2)}}$};
        \node[above,blue] at (8.3,2) {$\lp{\nu,\lambda^{(1)},\lambda^{(2)}}$};
    \end{tikzpicture}
    \end{center}
    The new horizontal strips $\lambda^{(1)}/\nu,\lambda^{(2)}/\nu$ are shown below. We remove the cell decorated with a $\textcolor{red}{\bullet}$ from $\mu$ to get $\nu=\Phi_{d,\lambda^{(1)},\lambda^{(2)}}(\mu)$.
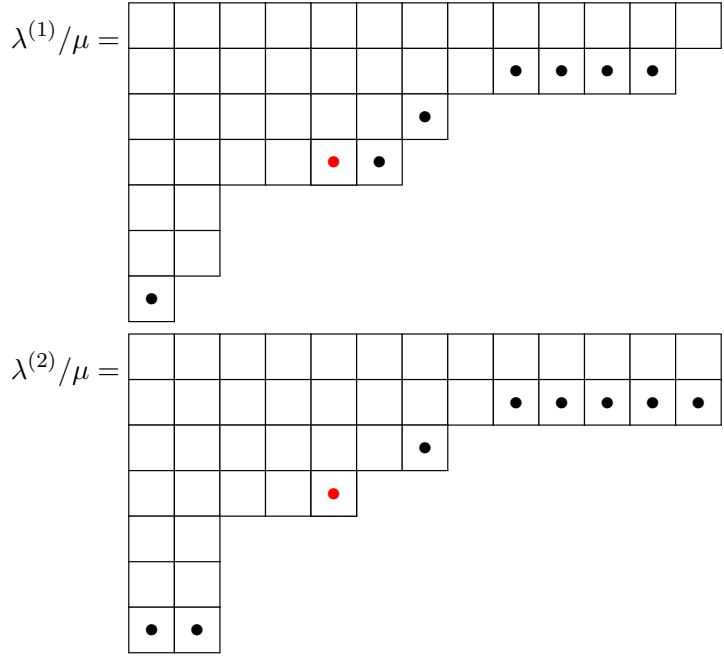
    \begin{align*}\lambda^{(1)}/\mu &=  \ydiagram[*(white) \textcolor{red}{\bullet}]{13+0,12+0,7+0,4+1,2+0,2+0}*[*(white) \bullet]
    {13+0,8+4,6+1,5+1,2+0,2+0,0+1}
    *[*(white)]{13,8,6,5,2,2} \\ \lambda^{(2)}/\mu &=  \ydiagram[*(white) \textcolor{red}{\bullet}]{13+0,12+0,7+0,4+1,2+0,2+0}*[*(white) \bullet]
    {13+0,8+5,6+1,5+0,2+0,2+0,0+2}
    *[*(white)]{13,8,6,5,2,2}\end{align*}
\end{example}

Before studying $\Phi_{d,\lambda^{(1)},\lambda^{(2)}}$, we need to show that it is well-defined using Lemmas~\ref{lem:x_0-range} and \ref{lem:bad-ver-well-defined}.

\begin{lemma}\label{lem:x_0-range}
    The integer $x_0$ above satisfies $0<x_0<n+m-d-r$.
\end{lemma}
\begin{proof}
    Since the fact that $\mu\in\hori{d,\lambda^{(1)},\lambda^{(2)}}\setminus\horipositive{d,\lambda^{(1)},\lambda^{(2)}}$ indicates the lowest point of $\lp{\mu,\lambda^{(1)},\lambda^{(2)}}\mid_{0\le x\le n+m-d-r}$ is strongly lower than the $x$-axis, we have $x_0>0$. It remains to show that $x_0<n+m-d-r$. Lemma~\ref{lem:height-of-right-end-of-lattice-path} indicates that the $y$-coordinate of $\lp{\mu,\lambda^{(1)},\lambda^{(2)}}$ at $n+m-d-r$ equals
    \begin{align*}
        n+m-2d-(n+m-d-r) = r-d \ge 0,
    \end{align*}
    which means that $\lp{\mu,\lambda^{(1)},\lambda^{(2)}}\mid_{0\le x\le n+m-d-r}$ does not attain its lowest point at $x= n+m-d-r$ as its lowest point is strongly lower than the $x$-axis. Therefore, we have $x_0<n+m-d-r$.
\end{proof}

\begin{lemma}\label{lem:bad-ver-well-defined}
    The map $\Phi_{d,\lambda^{(1)},\lambda^{(2)}}$ is well-defined, i.e., $\mu$ has at least one cell in the $x_0$-th column, and the diagram $\Phi_{d,\lambda^{(1)},\lambda^{(2)}}(\mu) = \nu$ defined above is indeed a partition in $\hori{d-1,\lambda^{(1)},\lambda^{(2)}}$.
\end{lemma}
\begin{proof}
    First, we need to show that $\nu$ is a partition. Since $\lp{\mu,\lambda^{(1)},\lambda^{(2)}}\mid_{0\le x\le n+m-d-r}$ attains its lowest point at $x=x_0$ where $0<x_0<n+m-d-r$ by Lemma~\ref{lem:x_0-range}, the $(x_0+1)$-th step of $\lp{\mu,\lambda^{(1)},\lambda^{(2)}}$ must be an NE or HE step. In addition, the minimality of $x_0$ means that the $x_0$-th step is an SE step. In conclusion, $\mu$ must have a row of length $x_0$, so the removal of the lowest cell from the $x_0$-th column of $\mu$ is permissible and yields a partition $\nu\vdash d-1$. We further know that the $x_0$-th column meets neither $\lambda^{(1)}/\mu$ nor $\lambda^{(2)}/\mu$, as we have shown that the $x_0$-th step is an SE step. Therefore, both $\lambda^{(1)}/\nu$ and $\lambda^{(2)}/\nu$ are still horizontal strips, indicating that $\nu\in\hori{d-1,\lambda^{(1)},\lambda^{(2)}}$.
\end{proof}

As $\Phi_{d,\lambda^{(1)},\lambda^{(2)}}$ is shown to be well-defined, we are ready to show that $\Phi_{d,\lambda^{(1)},\lambda^{(2)}}$ is bijective. It would be helpful to understand $\Phi_{d,\lambda^{(1)},\lambda^{(2)}}$ pictorially using lattice paths as stated in Example~\ref{ex:bad-ver-bijection}.

\begin{lemma}\label{lem:bad-version-bijection}
    The map $\Phi_{d,\lambda^{(1)},\lambda^{(2)}}$ is bijective.
\end{lemma}
\begin{proof}
    \textbf{Injectivity:} Suppose that for $\mu^{(1)},\mu^{(2)}\in\hori{d,\lambda^{(1)},\lambda^{(2)}}\setminus\horipositive{d,\lambda^{(1)},\lambda^{(2)}}$ we have $\Phi_{d,\lambda^{(1)},\lambda^{(2)}}(\mu^{(1)}) = \Phi_{d,\lambda^{(1)},\lambda^{(2)}}(\mu^{(2)})$. It suffices to show that $\mu^{(1)}=\mu^{(2)}$. Choose the minimal $x_1\ge 0$ (resp. $x_2\ge 0$) such that $\lp{\mu^{(1)},\lambda^{(1)},\lambda^{(2)}}\mid_{0\le x\le n+m-d-r}$ (resp. $\lp{\mu^{(2)},\lambda^{(1)},\lambda^{(2)}}\mid_{0\le x\le n+m-d-r}$) attains its lowest point at $x=x_1$ (resp. $x=x_2$). Then, for $i=1,2$ the new lattice path $\lp{\Phi_{d,\lambda^{(1)},\lambda^{(2)}}(\mu^{(i)}),\lambda^{(1)},\lambda^{(2)}}$ arises from $\lp{\mu^{(i)},\lambda^{(1)},\lambda^{(2)}}$ by replacing the $x_i$-th step (an SE step) with an NE step. As Lemma~\ref{lem:x_0-range} states that $0<x_i<n+m-d-r$ for $i=1,2$, the minimality of $x_i$ indicates that $\lp{\Phi_{d,\lambda^{(1)},\lambda^{(2)}}(\mu^{(i)}),\lambda^{(1)},\lambda^{(2)}}\mid_{0\le x\le n+m-d-r}$ attains its last lowest point at $x=x_i-1$. Therefore, $\Phi_{d,\lambda^{(1)},\lambda^{(2)}}(\mu^{(1)}) = \Phi_{d,\lambda^{(1)},\lambda^{(2)}}(\mu^{(2)})$ implies $x_1-1=x_2-1$ and hence $x_1=x_2$. Note that $\mu^{(i)}$ arises from appending a cell to the $x_i$-th column of $\Phi_{d,\lambda^{(1)},\lambda^{(2)}}(\mu^{(i)})$ for $i=1,2$. We thereby conclude that $\mu^{(1)} = \mu^{(2)}$.

    \textbf{Surjectivity:} Given $\nu\in\hori{d-1,\lambda^{(1)},\lambda^{(2)}}$, it suffices to find a partition $\mu\in\hori{d,\lambda^{(1)},\lambda^{(2)}}\setminus\horipositive{d,\lambda^{(1)},\lambda^{(2)}}$ with $\Phi_{d,\lambda^{(1)},\lambda^{(2)}}(\mu) = \nu$. Choose the maximal $x_0^\prime\ge 0$ such that the restricted lattice path $\lp{\nu,\lambda^{(1)},\lambda^{(2)}}\mid_{0\le x\le n+m-d-r}$ attains its lowest point at $x=x_0^\prime$. We initially give an upper bound to $x_0^\prime$:
    
    \textbf{Claim:} $x_0^\prime< n+m-d-r$.

    In fact, Lemma~\ref{lem:height-of-right-end-of-lattice-path} indicates that the $y$-coordinate of $\lp{\nu,\lambda^{(1)},\lambda^{(2)}}$ at $x=n+m-d-r$ is $n+m-2(d-1)-(n+m-d-r) = r-d+2\ge 2>0$. However, the $y$-coordinate of $\lp{\nu,\lambda^{(1)},\lambda^{(2)}}$ at $x=0$ is $0$. Therefore, $\lp{\nu,\lambda^{(1)},\lambda^{(2)}}\mid_{0\le x\le n+m-d-r}$ does not attain its lowest point at $x=n+m-d-r$, meaning that $x_0^\prime<n+m-d-r$.

    The maximality of $x_0^\prime$ and the claim above force the $(x_0^\prime+1)$-th step of $\lp{\nu,\lambda^{(1)},\lambda^{(2)}}$ to be an NE step. Furthermore, the $x_0^\prime$-th step of $\lp{\nu,\lambda^{(1)},\lambda^{(2)}}$ cannot be an NE step by the definition of $x_0^\prime$. Therefore, $\nu$ has a row of length $x_0^\prime$, so appending a cell in the $(x_0^\prime+1)$-th row of $\nu$ generates a new partition $\mu\vdash d$. We further deduce that $\mu\in\hori{d,\lambda^{(1)},\lambda^{(2)}}$, since we have shown that the $(x_0^\prime+1)$-th step of $\lp{\nu,\lambda^{(1)},\lambda^{(2)}}$ is an NE step.

    Now we show that $\mu\notin\horipositive{d,\lambda^{(1)},\lambda^{(2)}}$. the $y$-coordinate of $\lp{\nu,\lambda^{(1)},\lambda^{(2)}}$ at $x=0$ is $0$, so the $y$-coordinate of $\lp{\nu,\lambda^{(1)},\lambda^{(2)}}$ at $x=x_0^\prime$ is no greater than $0$. Note that $\lp{\mu,\lambda^{(1)},\lambda^{(2)}}$ arises from replacing the $(x_0^\prime+1)$-th step (an NE step) of $\lp{\nu,\lambda^{(1)},\lambda^{(2)}}$ with an SE step. Consequently, the $y$-coordinate of $\lp{\mu,\lambda^{(1)},\lambda^{(2)}}$ at $x=x_0^\prime+1$ is less than $0$, indicating that $\mu\notin\horipositive{d,\lambda^{(1)},\lambda^{(2)}}$ by Remark~\ref{rmk:check_horipositive}.

    To sum up, we have shown that $\mu\in\hori{d,\lambda^{(1)},\lambda^{(2)}}\setminus\horipositive{d,\lambda^{(1)},\lambda^{(2)}}$. It remains to show that $\Phi_{d,\lambda^{(1)},\lambda^{(2)}}(\mu)=\nu$. Recall two facts:
    \begin{itemize}
        \item $\lp{\mu,\lambda^{(1)},\lambda^{(2)}}$ arises from replacing the $(x_0^\prime+1)$-th step (an NE step) of $\lp{\nu,\lambda^{(1)},\lambda^{(2)}}$ with an SE step.
        \item $\lp{\nu,\lambda^{(1)},\lambda^{(2)}}\mid_{0\le n+m-d-r}$ attains its last lowest point at $x=x_0^\prime<n+m-d-r$.
    \end{itemize}
    As a result, $\lp{\mu,\lambda^{(1)},\lambda^{(2)}}\mid_{0\le x\le n+m-d-r}$ attains its first lowest point at $x=x_0^\prime + 1$. Moreover, we note that the removal of the lowest cell in the $(x_0^\prime+1)$-th column yields $\nu$. Therefore, the definition of $\Phi_{d,\lambda^{(1)},\lambda^{(2)}}$ indicates that $\Phi_{d,\lambda^{(1)},\lambda^{(2)}}(\mu) = \nu$.
\end{proof}

The bijections $\Phi_{d,\lambda^{(1)},\lambda^{(2)}}$ for $0<d\le r$ give our first positive combinatorial formula for the graded character of $R(\ZZZ_{n,m,r})$.

\begin{proposition}\label{prop:bad-ver-pos-formula}
    For $0\le d\le r$, we have
    \[\Frob(R(\ZZZ_{n,m,r})_d) = \sum_{\substack{\lambda^{(1)}\vdash n \\ \lambda^{(1)}_1\le n+m-d-r \\ \lambda^{(2)}\vdash m \\ \lambda^{(2)}_1\le n+m-d-r}}\sum_{\mu\in\horipositive{d,\lambda^{(1)},\lambda^{(2)}}} s_{\lambda^{(1)}}\otimes s_{\lambda^{(2)}}.\]
\end{proposition}
\begin{proof}
    By Theorem~\ref{thm:grad-str}, it suffices to show
    \begin{align}\label{eq:bad-ver-pos-formula-intermediate}
        \{SF_d - SF_{d-1}\}_{\lambda_1 \le n+m-d-r} = \sum_{\substack{\lambda^{(1)}\vdash n \\ \lambda^{(1)}_1\le n+m-d-r \\ \lambda^{(2)}\vdash m \\ \lambda^{(2)}_1\le n+m-d-r}}\sum_{\mu\in\horipositive{d,\lambda^{(1)},\lambda^{(2)}}} s_{\lambda^{(1)}}\otimes s_{\lambda^{(2)}}.
    \end{align}
    We have
    \begin{align*}
        &\{SF_d - SF_{d-1}\}_{\lambda_1 \le n+m-d-r} = \{SF_d\}_{\lambda_1\le n+m-d-r} - \{SF_{d-1}\}_{\lambda_1 \le n+m-d-r} \\
        =&\sum_{\substack{\lambda^{(1)}\vdash n \\ \lambda^{(1)}_1\le n+m-d-r \\ \lambda^{(2)}\vdash m \\ \lambda^{(2)}_1\le n+m-d-r}} \sum_{\mu\in\hori{d,\lambda^{(1)},\lambda^{(2)}}} s_{\lambda^{(1)}}\otimes s_{\lambda^{(2)}} - \sum_{\substack{\lambda^{(1)}\vdash n \\ \lambda^{(1)}_1\le n+m-d-r \\ \lambda^{(2)}\vdash m \\ \lambda^{(2)}_1\le n+m-d-r}} \sum_{\mu\in\hori{d-1,\lambda^{(1)},\lambda^{(2)}}} s_{\lambda^{(1)}}\otimes s_{\lambda^{(2)}}
    \end{align*}
    where the second equal sign arises from Pieri's rule. Thanks to Lemma~\ref{lem:bad-version-bijection}, we can replace the index set $\hori{d-1,\lambda^{(1)},\lambda^{(2)}}$ under the last summation above with $\Phi_{d,\lambda^{(1)},\lambda^{(2)}}\,:\,\hori{d,\lambda^{(1)},\lambda^{(2)}}\setminus\horipositive{d,\lambda^{(1)},\lambda^{(2)}}$, and thus obtain
    \begin{align*}
        &\{SF_d - SF_{d-1}\}_{\lambda_1 \le n+m-d-r} \\
        =&\sum_{\substack{\lambda^{(1)}\vdash n \\ \lambda^{(1)}_1\le n+m-d-r \\ \lambda^{(2)}\vdash m \\ \lambda^{(2)}_1\le n+m-d-r}} \sum_{\mu\in\hori{d,\lambda^{(1)},\lambda^{(2)}}} s_{\lambda^{(1)}}\otimes s_{\lambda^{(2)}} - \sum_{\substack{\lambda^{(1)}\vdash n \\ \lambda^{(1)}_1\le n+m-d-r \\ \lambda^{(2)}\vdash m \\ \lambda^{(2)}_1\le n+m-d-r}} \sum_{\mu\in\hori{d,\lambda^{(1)},\lambda^{(2)}}\setminus\horipositive{d,\lambda^{(1)},\lambda^{(2)}}} s_{\lambda^{(1)}}\otimes s_{\lambda^{(2)}} \\
        =& \sum_{\substack{\lambda^{(1)}\vdash n \\ \lambda^{(1)}_1\le n+m-d-r \\ \lambda^{(2)}\vdash m \\ \lambda^{(2)}_1\le n+m-d-r}} \Bigg(\sum_{\mu\in\hori{d,\lambda^{(1)},\lambda^{(2)}}} s_{\lambda^{(1)}}\otimes s_{\lambda^{(2)}} - \sum_{\mu\in\hori{d,\lambda^{(1)},\lambda^{(2)}}\setminus\horipositive{d,\lambda^{(1)},\lambda^{(2)}}} s_{\lambda^{(1)}}\otimes s_{\lambda^{(2)}}\Bigg) \\
        =&\sum_{\substack{\lambda^{(1)}\vdash n \\ \lambda^{(1)}_1\le n+m-d-r \\ \lambda^{(2)}\vdash m \\ \lambda^{(2)}_1\le n+m-d-r}} \sum_{\mu\in\horipositive{d,\lambda^{(1)},\lambda^{(2)}}} s_{\lambda^{(1)}}\otimes s_{\lambda^{(2)}}.
    \end{align*}
    Consequently, we conclude Equation~\eqref{eq:bad-ver-pos-formula-intermediate}, completing our proof.
\end{proof}

\subsection{Good version: a positive combinatorial formula as a refinement of $\Frob(R(\ZZZ_{n,m,r}))$}\label{subsec:good-ver}
We have found a positive combinatorial character formula (Proposition~\ref{prop:bad-ver-pos-formula}) for the modules $R(\ZZZ_{n,m,r})_d$ in Subsection~\ref{subsec:bad-ver}, which algebraically form a graded refinement of the module $R(\ZZZ_{n,m,r})$. In this subsection, we want to understand this refinement combinatorially. That is, we want to combinatorially embed the sum in Proposition~\ref{prop:bad-ver-pos-formula} into $\Frob(R(\ZZZ_{n,m,r}))$ given by Lemma~\ref{lem:ungraded-frob} and thus obtain a combinatorial refinement of $\Frob(R(\ZZZ_{n,m,r}))$ given by Theorem~\ref{thm:good-ver-positive-formula}.

Throughout the rest of this subsection, we fix an integer $r$ with $0\le r \le \min\{m,n\}$. Let $0\le d\le r$ be an integer and $\lambda^{(1)}\vdash n, \lambda^{(2)}\vdash m$ be two partitions. We define:
\begin{align}\label{eq:whs-set}
    \horiwid{d,\lambda^{(1)},\lambda^{(2)}} \coloneqq \{\mu\in\hori{r,\lambda^{(1)},\lambda^{(2)}}\,:\, \wid{\mu,\lambda^{(1)},\lambda^{(2)}} = n+m-d-r\}.
\end{align}
Our goal is to construct a one-to-one correspondence between the two sets $\horipositive{d,\lambda^{(1)},\lambda^{(2)}}$ and $\horiwid{d,\lambda^{(1)},\lambda^{(2)}}$. First, we define the \emph{left shadow map}
\begin{align}\label{eq:left-shadow-map}
    \leftshadow{d,\lambda^{(1)},\lambda^{(2)}}\,:\,\horipositive{d,\lambda^{(1)},\lambda^{(2)}}\longrightarrow\horiwid{d,\lambda^{(1)},\lambda^{(2)}}
\end{align}
through the following steps sequentially:
\begin{itemize}
    \item Given $\mu\in\horipositive{d,\lambda^{(1)},\lambda^{(2)}}$, we define a set $S\subseteq[n+m-d-r]$ by
    \[S\!\coloneqq\!\{i\in[n+m-d-r]\,:\,\text{$(i,j)$ is a reflection pair of $\lp{\mu,\lambda^{(1)},\lambda^{(2)}}$ for some $j\in[n+m-d-r]$}\}.\]
    \item Append one cell to the $i$-th column of $\mu$ for all $i\in [n+m-d-r]\setminus S$ and obtain a new diagram $\nu$.
    \item Define $\leftshadow{d,\lambda^{(1)},\lambda^{(2)}}(\mu) \coloneqq \nu$.
\end{itemize}

\begin{remark}\label{rmk:left-shadow}
    Let $\mu,\nu,S$ be what we use to define $\leftshadow{d,\lambda^{(1)},\lambda^{(2)}}$ above.
    It is helpful to pictorially understand $\leftshadow{d,\lambda^{(1)},\lambda^{(2)}}$ using lattice paths. We put a light source on the right of the restricted lattice path $\lp{\mu,\lambda^{(1)},\lambda^{(2)}}\mid_{0\le x\le n+m-d-r}$ which is weakly higher than the $x$-axis by Remark~\ref{rmk:check_horipositive}. Let the light source emit horizontal light rays leftwards. See Figure~\ref{fig:left-ray} for example. All the NE steps touched by these light rays are marked in red.
    \def\sequence{1,-1,1,1,1,-1,1,-1,-1,-1,1,1,1,0,0,-1,1,1,1,1,1,-1,0,0,-1,-1}
        \def\gridscale{0.5}
        \begin{figure}[H]
            \centering
            \begin{tikzpicture}[scale=\gridscale]
            \tikzmath{
            integer \x,\y,\miny,\maxy;
            \x=0;\y=0;\miny=0;\maxy=0;
            for \s in \sequence {
            {\draw[->,line width=2pt,color = blue] (\x,\y) -- (\x+1,\y+\s);};
            \x=\x+1;\y=\y+\s;
            if \y>\maxy then {\maxy = \y;};
            if \y<\miny then {\miny = \y;};
            };
            {
            \draw[gray!70, thin] (0,\miny-0.5) grid (\x+1,\maxy+0.5);
            \draw[very thick, -Stealth] (0,0) -- (\x+1,0) node[below] {$x$};
            \draw[thick, -Stealth] (0,\miny-0.5) -- (0,\maxy+0.5) node[left]  {$y$};
            \draw[->,very thick,red] (27,0.5) -- (11,0.5);
            \draw[line width=2pt,red] (10,0) -- (12,2);
            \draw[->,very thick,red] (27,1.5) -- (12,1.5);
            \draw[->,very thick,red] (27,2.5) -- (17,2.5);
            \draw[line width=2pt,red] (16,2) -- (18,4);
            \draw[->,very thick,red] (27,3.5) -- (18,3.5);
            \draw[purple,very thick] (26,\miny) -- (26,\maxy+0.2) node[above] {\small $x=n+m-d-r$};
            };
            }
        \end{tikzpicture}
            \caption{$\lp{\mu,\lambda^{(1)},\lambda^{(2)}}\mid_{0\le x\le n+m-d-r}$}
            \label{fig:left-ray}
        \end{figure}
        
    Note that all the NE steps not touched by these light rays perfectly match all the SE steps of $\lp{\mu,\lambda^{(1)},\lambda^{(2)}}\mid_{0\le x\le n+m-d-r}$, forming reflection pairs. This is because $\mu\in\horipositive{d,\lambda^{(1)},\lambda^{(2)}}$, meaning that $\lp{\mu,\lambda^{(1)},\lambda^{(2)}}\mid_{0\le x\le n+m-d-r}$ is weakly higher than the $x$-axis. Therefore, all the NE steps touched by light rays are exactly indexed by $[n+m-d-r]\setminus S$. Consequently, what $\leftshadow{d,\lambda^{(1)},\lambda^{(2)}}$ does is append a cell to each column of $\mu$ with indices of these NE steps touched by light rays, or, equivalently, remove a cell from from each column with indices of these NE steps touched by light rays of both $\lambda^{(1)}/\mu$ and $\lambda^{(2)}/\mu$. This means that $\lp{\nu,\lambda^{(1)},\lambda^{(2)}}\mid_{0\le x\le n+m-d-r}$ arises from $\lp{\mu,\lambda^{(1)},\lambda^{(2)}}\mid_{0\le x\le n+m-d-r}$ by replacing all the NE steps touched by light rays with SE steps, shown in Figure~\ref{fig:after-left-ray}. Aside from these new SE steps, other steps of $\lp{\nu,\lambda^{(1)},\lambda^{(2)}}\mid_{0\le x\le n+m-d-r}$ still perfectly match each other as in $\lp{\mu,\lambda^{(1)},\lambda^{(2)}}\mid_{0\le x\le n+m-d-r}$, and form all the reflection pairs of $\lp{\nu,\lambda^{(1)},\lambda^{(2)}}$. 
    \begin{figure}[H]
        \centering
        \def\sequence{1,-1,1,1,1,-1,1,-1,-1,-1,-1,-1,1,0,0,-1,-1,-1,1,1,1,-1,0,0,-1,-1}
        \def\gridscale{0.5}
        \begin{tikzpicture}[scale=\gridscale]
            \tikzmath{
            integer \x,\y,\miny,\maxy;
            \x=0;\y=0;\miny=0;\maxy=0;
            for \s in \sequence {
            {\draw[->,line width=2pt,color = blue] (\x,\y) -- (\x+1,\y+\s);};
            \x=\x+1;\y=\y+\s;
            if \y>\maxy then {\maxy = \y;};
            if \y<\miny then {\miny = \y;};
            };
            {
            \draw[gray!70, thin] (0,\miny-0.5) grid (\x+1,\maxy+0.5);
            \draw[very thick, -Stealth] (0,0) -- (\x+1,0) node[below] {$x$};
            \draw[thick, -Stealth] (0,\miny-0.5) -- (0,\maxy+0.5) node[left]  {$y$};
            \draw[line width=2pt,red] (10,0) -- (12,-2);
            \draw[line width=2pt,red] (16,-2) -- (18,-4);
            \draw[purple,very thick] (26,\miny) -- (26,\maxy+0.2) node[above] {\small $x=n+m-d-r$};
            \draw[red,dashed,very thick] (0.5,0.5) -- (1.5,0.5);
            \draw[red,dashed,very thick] (2.5,0.5) -- (9.5,0.5);
            \draw[red,dashed,very thick] (3.5,1.5) -- (8.5,1.5);
            \draw[red,dashed,very thick] (4.5,2.5) -- (5.5,2.5);
            \draw[red,dashed,very thick] (6.5,2.5) -- (7.5,2.5);
            \draw[red,dashed,very thick] (12.5,-1.5) -- (15.5,-1.5);
            \draw[red,dashed,very thick] (18.5,-3.5) -- (25.5,-3.5);
            \draw[red,dashed,very thick] (19.5,-2.5) -- (24.5,-2.5);
            \draw[red,dashed,very thick] (20.5,-1.5) -- (21.5,-1.5);
            };
            }
        \end{tikzpicture}
        \caption{$\lp{\nu,\lambda^{(1)},\lambda^{(2)}}\mid_{\lambda_1\le n+m-d-r}$}
        \label{fig:after-left-ray}
    \end{figure}
    All the red steps in Figure~\ref{fig:left-ray} and all the red steps in Figure~\ref{fig:after-left-ray} are symmetric with respect to the $x$-axis.
\end{remark}

We need to show that $\leftshadow{d,\lambda^{(1)},\lambda^{(2)}}$ is well-defined.
\begin{lemma}\label{lem:left-shadow-map-well-defined}
    The map $\leftshadow{d,\lambda^{(1)},\lambda^{(2)}}$ is well-defined, i.e., the diagram $\nu$ above is indeed a partition in $\horiwid{d,\lambda^{(1)},\lambda^{(2)}}$.
\end{lemma}
\begin{proof}
    First, we need to show that $\nu$ is a partition. As in Remark~\ref{rmk:left-shadow}, we put a light source on the right of $\lp{\mu,\lambda^{(1)},\lambda^{(2)}}\mid_{0\le x\le n+m-d-r}$ which emits horizontal light rays leftwards. Then the index set $[n+m-d-r]\setminus S$ coincides with
    \[\Bigg\{i\in[n+m-d-r]\,:\,\begin{array}{cc}
         & \text{the $i$-th step of $\lp{\mu,\lambda^{(1)},\lambda^{(2)}}\mid_{0\le x\le n+m-d-r}$} \\
         & \text{is an NE step touched by light rays}
    \end{array} \Bigg\}\]
    marked in red in Figure~\ref{fig:left-ray}. Therefore, all the NE steps of $\lp{\mu,\lambda^{(1)},\lambda^{(2)}}\mid_{0\le x\le n+m-d-r}$ with indices in $[n+m-d-r]\setminus S$ are partitioned into connected components, where the left-adjacent step of each connected component is an SE or HE step. In other words, $[n+m-d-r]\setminus S$ is partitioned into intervals, where $\mu$ has a row of length $i-1$ for the minimal element $i$ in each interval. Therefore, appending a cell to each column of $\mu$ with indices in $[n+m-d-r]\setminus S$ still results in a partition, i.e., $\nu$ is a partition.

    Additionally, we need to show that $\lambda^{(i)}/\nu$ is a horizontal strip for $i=1,2$. We have shown that all the NE steps of $\lp{\mu,\lambda^{(1)},\lambda^{(2)}}\mid_{0\le x\le n+m-d-r}$ with indices in $[n+m-d-r]\setminus S$ are partitioned into connected components, where the left-adjacent step of each connected component is an SE or HE step. Thereby, for $i=1,2$ we can construct $\lambda^{(i)}/\nu$ by removing the leftmost several cells from each row of the horizontal strip $\lambda^{(i)}/\mu$, which means that $\lambda^{(i)}/\nu$ is still a horizontal strip.

    Furthermore, we want to show that $\lvert\nu\rvert = r$. By Lemma~\ref{lem:height-of-right-end-of-lattice-path}, we only need to calculate the $y$-coordinate of $\lp{\nu,\lambda^{(1)},\lambda^{(2)}}$ at $x=n+m-d-r$. Lemma~\ref{lem:height-of-right-end-of-lattice-path} indicates that the $y$-coordinate of $\lp{\mu,\lambda^{(1)},\lambda^{(2)}}$ at $x=n+m-d-r$ is $n+m-2d-(n+m-d-r) = r-d$. That is, the number of the NE steps of $\lp{\mu,\lambda^{(1)},\lambda^{(2)}}\mid_{0\le x\le n+m-d-r}$ touched by light rays (marked in red in Figure~\ref{fig:left-ray}) is $r-d$. Remark~\ref{rmk:left-shadow} indicates that replacing these NE steps with SE steps yields $\lp{\nu,\lambda^{(1)},\lambda^{(2)}}\mid_{0\le x\le n+m-d-r}$, so the $y$-coordinate of $\lp{\nu,\lambda^{(1)},\lambda^{(2)}}$ at $x=n+m-d-r$ is $-(r-d) =d-r$. Then, the implementation of Lemma~\ref{lem:height-of-right-end-of-lattice-path} to $(\nu,\lambda^{(1)},\lambda^{(2)})$ gives
    \[d-r = n+m-2\lvert\nu\rvert-(n+m-d-r),\]
    yielding $\lvert\nu\rvert = r$.

    To sum up, we have shown that $\nu\in\hori{r,\lambda^{(1)},\lambda^{(2)}}$. It remains to show $\wid{\nu,\lambda^{(1)},\lambda^{(2)}} = n+m-d-r$. As mentioned in Remark~\ref{rmk:left-shadow}, all the reflection pairs $(i,j)$ of $\lp{\nu,\lambda^{(1)},\lambda^{(2)}}$ satisfies $j\le n+m-d-r$ (see Figure~\ref{fig:after-left-ray} for example), so we have
    \[n+m-d-r\ge M\coloneqq \max\{j\in\ZZ_{>0}\,:\,\text{$(i,j)$ is a reflection pair of $\lp{\nu,\lambda^{(1)},\lambda^{(2)}}$ for some $i\in\ZZ_{>0}$}\}.\]
    Recalling that $\lambda^{(1)}_1\le n+m-d-r$ and $\lambda^{(2)}_1\le n+m-d-r$, we further obtain
    \begin{equation}\label{eq:wid-ineq}
        n+m-d-r\ge\max\{M,\lambda^{(1)}_1,\lambda^{(2)}_1\}.
    \end{equation}
    It suffices to force this inequality to be an equality.
    
        \noindent\textbf{Case (1):} If $n+m-d-r = \lambda^{(1)}_1$ or $n+m-d-r=\lambda_1^{(2)}$, the inequality of \eqref{eq:wid-ineq} is forced to be an equality.
        
        \noindent \textbf{Case (2):} Otherwise, we have $\max\{\lambda^{(1)}_1,\lambda_1^{(2)}\}<n+m-d-r$. In this case, the $(n+m-d-r)$-th step of $\lp{\mu,\lambda^{(1)},\lambda^{(2)}}$ is an SE step, and thereby matches some NE step of $\lp{\mu,\lambda^{(1)},\lambda^{(2)}}$ to form a reflection pair $(i,n+m-d-r)$. This is because the fact that $\mu\in\horipositive{d,\lambda^{(1)},\lambda^{(2)}}$ and Remark~\ref{rmk:check_horipositive} together indicate that $\lp{\mu,\lambda^{(1)},\lambda^{(2)}}\mid_{\lambda_1\le n+m-d-r}$ is weakly higher than the $x$-axis. As mentioned in Remark~\ref{rmk:left-shadow}, $(i,n+m-d-r)$ is also a reflection pair of $\lp{\nu,\lambda^{(1)},\lambda^{(2)}}$, which means that $n+m-d-r\le M$, forcing the inequality~\eqref{eq:wid-ineq} to be an equality.
\end{proof}

As we have shown that $\leftshadow{d,\lambda^{(1)},\lambda^{(2)}}$ is well-defined, we are ready to show the bijectivity of $\leftshadow{d,\lambda^{(1)},\lambda^{(2)}}$. Intuitively, the bijectivity arises from the relationship between Figures~\ref{fig:left-ray} and \ref{fig:after-left-ray}: it is easy to construct one from the other. 

\begin{lemma}\label{lem:left-shadow-map-bijectivity}
    The map $\leftshadow{d,\lambda^{(1)},\lambda^{(2)}}$ is bijective.
\end{lemma}
\begin{proof}
    \textbf{Injectivity:} We suppose that $\leftshadow{d,\lambda^{(1)},\lambda^{(2)}}(\mu^{(1)}) = \leftshadow{d,\lambda^{(1)},\lambda^{(2)}}(\mu^{(2)})$ for some partitions $\mu^{(1)},\mu^{(2)}\in\horipositive{d,\lambda^{(1)},\lambda^{(2)}}$. It suffices to show $\mu^{(1)} =\mu^{(2)}$. Recall that what $\leftshadow{d,\lambda^{(1)},\lambda^{(2)}}$ does is append a cell to some columns of $\mu^{(i)}$, and we can figure out the indices of these columns using leftwards light rays as in Remark~\ref{rmk:left-shadow}. That is, we put a light source on the right of $\lp{\mu,\lambda^{(1)},\lambda^{(2)}}\mid_{0\le x\le n+m-d-r}$ which emits horizontal light rays leftwards, then the indices of all the NE steps touched by light rays are exactly what we desire. 
    
    It suffices to show that the set of all these indices is uniquely determined by $\leftshadow{d,\lambda^{(1)},\lambda^{(2)}}(\mu^{(i)})$. Remark~\ref{rmk:left-shadow} indicates that the lattice path $\lp{\leftshadow{d,\lambda^{(1)},\lambda^{(2)}}(\mu^{(i)}),\lambda^{(1)},\lambda^{(2)}}$ arises from converting all the NE steps of $\lp{\mu^{(i)},\lambda^{(1)},\lambda^{(2)}}$ with indices above (marked in red in Figure~\ref{fig:left-ray}) into SE steps (and we only need to show that $\leftshadow{d,\lambda^{(1)},\lambda^{(2)}}(\mu^{(i)})$ uniquely determines the indices of these new SE steps). These new SE steps (marked in red in Figure~\ref{fig:after-left-ray}) can be identified as follows:
    \begin{itemize}
        \item Put a light source on the left of $\lp{\leftshadow{d,\lambda^{(1)},\lambda^{(2)}}(\mu^{(i)}),\lambda^{(1)},\lambda^{(2)}}$, which emits horizontal light rays rightwards.
        \item All the SE steps of $\lp{\leftshadow{d,\lambda^{(1)},\lambda^{(2)}}(\mu^{(i)}),\lambda^{(1)},\lambda^{(2)}}\mid_{0\le x\le n+m-d-r}$ are exactly what we desire.
    \end{itemize}
    To sum up, only using $\lp{\leftshadow{d,\lambda^{(1)},\lambda^{(2)}}(\mu^{(i)}),\lambda^{(1)},\lambda^{(2)}}$, we can identify all the columns of $\mu^{(i)}$ where we append a cell when constructing $\leftshadow{d,\lambda^{(1)},\lambda^{(2)}}(\mu^{(i)})$. This fact yields the injectivity of $\leftshadow{d,\lambda^{(1)},\lambda^{(2)}}$.

    \textbf{Surjectivity:} Given $\nu\in\horiwid{d,\lambda^{(1)},\lambda^{(2)}}$, it suffices to find a partition $\mu\in\horipositive{d,\lambda^{(1)},\lambda^{(2)}}$ such that $\leftshadow{d,\lambda^{(1)},\lambda^{(2)}}(\mu) = \nu$. We construct an index set $T\subseteq[n+m-d-r]$ as follows:
    \begin{itemize}
        \item Put a light source on the left of $\lp{\nu,\lambda^{(1)},\lambda^{(2)}}$, which emits horizontal light rays rightwards.
        \item Let $T\coloneqq\Bigg\{i\in[n+m-d-r]\,:\,\begin{array}{cc}
             & \text{the $i$-th step of $\lp{\nu,\lambda^{(1)},\lambda^{(2)}}$ is} \\
             & \text{an SE step touched by light rays}
        \end{array}\Bigg\}$.
    \end{itemize}  
    Note that $T$ is a disjoint union of intervals, and the right-adjacent integer of each interval indexes an NE or HE step of $\lp{\nu,\lambda^{(1)},\lambda^{(2)}}$. Therefore, for the maximal integer $j$ of each interval, $\nu$ has a row of length $j$. Consequently, for all $i\in T$ we can remove the lowest cell from the $i$-th column of $\nu$, yielding a new partition $\mu$. In addition, both $\lambda^{(1)}/\mu$ and $\lambda^{(2)}/\mu$ are horizontal strips, because all the steps of $\lp{\nu,\lambda^{(1)},\lambda^{(2)}}$ with indices in $T$ are SE steps, meaning that for all $i\in T$ the $i$-th columns of both $\lambda^{(1)}/\nu$ and $\lambda^{(2)}/\nu$ are empty.

    To guarantee that $\mu\in\horipositive{d,\lambda^{(1)},\lambda^{(2)}}$, it remains to show $\lvert\mu\rvert=d$ and that the restricted lattice path $\lp{\mu,\lambda^{(1)},\lambda^{(2)}}\mid_{0\le x\le n+m-d-r}$ is weakly higher than the $x$-axis as mentioned in Remark~\ref{rmk:check_horipositive}. Before proving these two results, we introduce a fact immediate from the construction of $\mu$. \\
    \textbf{Fact:} Converting the all the SE steps of $\lp{\nu,\lambda^{(1)},\lambda^{(2)}}$ with indices in $T$ into NE steps yields $\lp{\mu,\lambda^{(1)},\lambda^{(2)}}$. Pictorially, this procedure converts the lattice path in Figure~\ref{fig:after-left-ray} into the lattice path in Figure~\ref{fig:left-ray}.

    Initially, let us show $\lvert\mu\rvert=d$. Lemma~\ref{lem:height-of-right-end-of-lattice-path} indicates that the $y$-coordinate of $\lp{\nu,\lambda^{(1)},\lambda^{(2)}}$ at $x=n+m-d-r$ equals $n+m-2r-(n+m-d-r) = d-r$. Then, the fact above indicates that the $y$-coordinate of $\lp{\mu,\lambda^{(1)},\lambda^{(2)}}$ at $x=n+m-d-r$ equals $-(d-r) = r-d$, because $\wid{\nu,\lambda^{(1)},\lambda^{(2)}} = n+m-d-r$ means that $\lp{\nu,\lambda^{(1)},\lambda^{(2)}}\mid_{0\le x\le n+m-d-r}$ attains its lowest point at $x=n+m-d-r$.
    Therefore, Lemma~\ref{lem:height-of-right-end-of-lattice-path} implies $r-d = n+m - 2\lvert\mu\rvert -(n+m-d-r)$ and hence $\lvert\mu\rvert = d$.

    It remains to show that the restricted lattice path $\lp{\mu,\lambda^{(1)},\lambda^{(2)}}\mid_{0\le x\le n+m-d-r}$ is weakly higher than the $x$-axis. This is immediate from the fact above.
\end{proof}

The bijections $\leftshadow{d,\lambda^{(1)},\lambda^{(2)}}$ for $0\le d\le r$ enable us to obtain a conciser positive combinatorial formula for $\grFrob(R(\ZZZ_{n,m,r});q)$ than Proposition~\ref{prop:bad-ver-pos-formula}. Surprisingly, the sum index set of this new formula coincides with the sum index set of the Schur expansion of $\Frob(R(\ZZZ_{n,m,r}))$ given by Lemma~\ref{lem:ungraded-frob} and Pieri's rule together. This observation means that the following new formula is a combinatorial refinement of $\Frob(R(\ZZZ_{n,m,r}))$.
\begin{theorem}\label{thm:good-ver-positive-formula}
    \[\grFrob(R(\ZZZ_{n,m,r});q) = \sum_{\substack{\lambda^{(1)}\vdash n \\ \lambda^{(2)}\vdash m}} \sum_{\mu\in\hori{r,\lambda^{(1)},\lambda^{(2)}}} q^{n+m-r-\wid{\mu,\lambda^{(1)},\lambda^{(2)}}}\cdot s_{\lambda^{(1)}}\otimes s_{\lambda^{(2)}}\]
\end{theorem}
\begin{proof}
    It suffices to show
    \[\Frob(R(\ZZZ_{n,m,r})_d) = \sum_{\substack{\lambda^{(1)}\vdash n \\ \lambda^{(2)}\vdash m}} \sum_{\substack{\mu\in\hori{r,\lambda^{(1)},\lambda^{(2)}}\\ \wid{\mu,\lambda^{(1)},\lambda^{(2)}}=n+m-d-r}}s_{\lambda^{(1)}}\otimes s_{\lambda^{(2)}}\]
    for $0\le d\le r$.
    Note that $\wid{\mu,\lambda^{(1)},\lambda^{(2)}}=n+m-d-r$ implies $\lambda^{(1)}_1\le n+m-d-r$ and $\lambda^{(2)}_1\le n+m-d-r$. As a result, it suffices to show
    \[\Frob(R(\ZZZ_{n,m,r})_d) = \sum_{\substack{\lambda^{(1)}\vdash n \\ \lambda^{(1)}_1\le n+m-d-r \\ \lambda^{(2)}\vdash m \\ \lambda^{(2)}_1\le n+m-d-r}} \sum_{\substack{\mu\in\hori{r,\lambda^{(1)},\lambda^{(2)}}\\ \wid{\mu,\lambda^{(1)},\lambda^{(2)}}=n+m-d-r}}s_{\lambda^{(1)}}\otimes s_{\lambda^{(2)}},\]
    which, by Proposition~\ref{prop:bad-ver-pos-formula}, is equivalent to
    \[\sum_{\substack{\lambda^{(1)}\vdash n \\ \lambda^{(1)}_1\le n+m-d-r \\ \lambda^{(2)}\vdash m \\ \lambda^{(2)}_1\le n+m-d-r}}\sum_{\mu\in\horipositive{d,\lambda^{(1)},\lambda^{(2)}}} s_{\lambda^{(1)}}\otimes s_{\lambda^{(2)}} = \sum_{\substack{\lambda^{(1)}\vdash n \\ \lambda^{(1)}_1\le n+m-d-r \\ \lambda^{(2)}\vdash m \\ \lambda^{(2)}_1\le n+m-d-r}} \sum_{\substack{\mu\in\hori{r,\lambda^{(1)},\lambda^{(2)}}\\ \wid{\mu,\lambda^{(1)},\lambda^{(2)}}=n+m-d-r}}s_{\lambda^{(1)}}\otimes s_{\lambda^{(2)}}.\]
    This equality holds since Lemma~\ref{lem:left-shadow-map-bijectivity} gives a bijection between the index set under the summations of both sides, so our proof is complete.
\end{proof}

\subsection{Application}\label{subsec:application-of-sign-free-formula}

Interestingly, Proposition~\ref{prop:bad-ver-pos-formula} indicates that lots of surjections $\twoheadrightarrow$ in Figure~\ref{fig:graded-surj} are actually isomorphisms. To see this fact, we need a technical result, which removes $r$ from the right-hand side of Proposition~\ref{prop:bad-ver-pos-formula} in some special cases.
\begin{lemma}\label{lem:surj-to-isom}
    For $0\le d\le r\le\min\{m,n\}$ such that $d\le\min\{m,n\}-r$, we have
    \[\Frob(R(\ZZZ_{n,m,r})_d) = \sum_{\substack{\lambda^{(1)}\vdash n \\  \lambda^{(2)}\vdash m }}\sum_{\mu\in\horipositive{d,\lambda^{(1)},\lambda^{(2)}}} s_{\lambda^{(1)}}\otimes s_{\lambda^{(2)}}.\]
\end{lemma}
\begin{proof}
    The assumption $d\le\min\{m,n\}-r$ indicates $n+m-d-r\ge\max\{m,n\}$, so for $\lambda^{(1)}\vdash n$ and $\lambda^{(2)}\vdash m$ we automatically have $\lambda_1^{(1)}\le n+m-d-r$ and $\lambda_1^{(2)}\le n+m-d-r$. In this case, we change nothing after throwing away these two inequalities under the summation in Proposition~\ref {prop:bad-ver-pos-formula}.
\end{proof}

Therefore, the sign-free formula in Proposition~\ref{prop:bad-ver-pos-formula} helps us obtain the following isomorphisms, which are unclear from the signed formula in Theorem~\ref{thm:grad-str}.
\begin{proposition}\label{prop:surj-to-isom}
    For $0\le d\le r<\min\{m,n\}$ such that $d\le\min\{m,n\}-r-1$, we have the $\symm_n\times\symm_m$-module isomorphism
    \[R(\ZZZ_{n,m,r})_d\cong R(\ZZZ_{n,m,r+1})_d.\]
\end{proposition}
\begin{proof}
    Note that the right-hand side of Lemma~\ref{lem:surj-to-isom} does not depend on $r$. Consequently, the application of Lemma~\ref{lem:surj-to-isom} to $R(\ZZZ_{n,m,r})_d$ and $R(\ZZZ_{n,m,r+1})_d$ yields what we desire.
\end{proof}
\begin{remark}\label{rmk:isom-ideal}
    Proposition~\ref{prop:surj-to-isom} is also immediate from Proposition~\ref{prop:ideal-equal}, since the degrees of the last three types of generators mentioned in Definition~\ref{def:ideal} are greater than $d$.
\end{remark}
\begin{remark}\label{rmk:surj-to-isom}
    Proposition~\ref{prop:surj-to-isom} indicates that all the surjections $\twoheadrightarrow$ under the straight line $d=\min\{m,n\}-r$ in Figure~\ref{fig:graded-surj} can be replaced with isomorphisms $\cong$. Take $\min\{m,n\}=6$ for example. Represent each of the modules in Figure~\ref{fig:graded-surj} using either $\circ$ or $\star$. Then all the stars in the same row of the diagram below can be connected by isomorphisms.
    
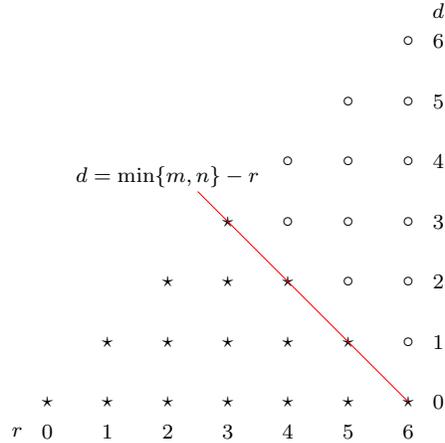
\begin{center}
        \begin{scriptsize}
            \begin{tikzpicture}[scale = 0.4]
                \node at (20,0) {$\star$};

                \node at (22,0) {$\star$};
                \node at (22,2) {$\star$};

                \node at (24,0) {$\star$};
                \node at (24,2) {$\star$};
                \node at (24,4) {$\star$};

                \node at (26,0) {$\star$};
                \node at (26,2) {$\star$};
                \node at (26,4) {$\star$};
                \node at (26,6) {$\star$};

                \node at (28,0) {$\star$};
                \node at (28,2) {$\star$};
                \node at (28,4) {$\star$};
                \node at (28,6) {$\circ$};
                \node at (28,8) {$\circ$};

                \node at (30,0) {$\star$};
                \node at (30,2) {$\star$};
                \node at (30,4) {$\circ$};
                \node at (30,6) {$\circ$};
                \node at (30,8) {$\circ$};
                \node at (30,10) {$\circ$};
                
                \node at (32,0) {$\star$};
                \node at (32,2) {$\circ$};
                \node at (32,4) {$\circ$};
                \node at (32,6) {$\circ$};
                \node at (32,8) {$\circ$};
                \node at (32,10) {$\circ$};
                \node at (32,12) {$\circ$};

                \node at (19,-1) {$r$};

                \node at (20,-1) {$0$};
                \node at (22,-1) {$1$};
                \node at (24,-1) {$2$};
                \node at (26,-1) {$3$};
                \node at (28,-1) {$4$};
                \node at (30,-1) {$5$};
                \node at (32,-1) {$6$};

                \node at (33,13) {$d$};
                \node at (33,12) {$6$};
                \node at (33,10) {$5$};
                \node at (33,8) {$4$};
                \node at (33,6) {$3$};
                \node at (33,4) {$2$};
                \node at (33,2) {$1$};
                \node at (33,0) {$0$};
                
                \draw[red] (32,0) -- (25,7);
                \node at (24,7.5) {$d=\min\{m,n\}-r$};
            \end{tikzpicture}
        \end{scriptsize}
    \end{center}
\end{remark}

Since the direct sum of all the modules in the $r$-th column of the diagram above exactly equals $R(\ZZZ_{n,m,r})\cong\CC[\ZZZ_{n,m,r}]$, we immediately deduce the module injections below.
\begin{corollary}\label{cor:C[Z]-injection}
    For $1\le r\le\frac{\min\{m,n\}}{2}$, there exists an $\symm_n\times\symm_m$-module embedding
    \[\CC[\ZZZ_{n,m,r-1}]\hookrightarrow\CC[\ZZZ_{n,m,r}].\]
\end{corollary}
\begin{remark}\label{rmk:C[Z]-injection}
    Though Corollary~\ref{cor:C[Z]-injection} only states the existence of such a module embedding, we can construct a concrete $\symm_n\times\symm_m$-equivariant map to give such a module embedding. Identifying $\CC[\ZZZ_{n,m,r^\prime}]$ with the space of all the formal sums $\sum_{\RRR\in\ZZZ_{n,m,r^\prime}}c_\RRR\cdot\RRR$ with complex coefficients $c_\RRR\in\CC$, we claim that
    \begin{align*}
        \CC[\ZZZ_{n,m,r-1}]&\longrightarrow\CC[\ZZZ_{n,m,r}] \\
        \tilde{\RRR}&\longmapsto \sum_{\tilde{\RRR}\subseteq\RRR\in\ZZZ_{n,m,r}}\RRR
    \end{align*}
    is injective. This is because the identification $\CC[\ZZZ_{n,m,r^\prime}]\cong\CC[\xxx_{n\times m}]/\II(\ZZZ_{n,m,r^\prime})$ identifies this map above with \eqref{eq:ungraded-map}, and \eqref{eq:ungraded-map} descends to a family of module surjections $R(\ZZZ_{n,m,r-1})_d\twoheadrightarrow R(\ZZZ_{n,m,r})_d$ ($0\le d\le r-1$) according to Corollary~\ref{cor:graded-surj}. However, these surjections $R(\ZZZ_{n,m,r-1})_d\twoheadrightarrow R(\ZZZ_{n,m,r})_d$ are in fact isomorphisms because of Proposition~\ref{prop:surj-to-isom}.
\end{remark}

\section{Conclusion}\label{sec:conclusion}
We raise an open problem concerning log-concavity. We say that a sequence $(a_1,a_2,\dots,a_N)$ of positive real numbers is \emph{log-concave} if $a_i^2\ge a_{i-1}\cdot a_{i+1}$ for all $1<i<N$. Chen conjectured \cite{chen2008logconcavityqlogconvexityconjectureslongest} that the sequence $(a_{n,1},a_{n,2},\dots,a_{n,n})$ is log-concave, where $a_{n,k}$ counts the permutations in $\symm_n$ with longest increasing subsequence of length $k$.

One way to generalize log-concavity is to incorporate $G$-equivariance with respect to the action of some group $G$. Let $(V_1,\dots,V_N)$ be a sequence of $G$-modules. It is \emph{$G$-log-concave} if we have $G$-module embeddings $V_{i-1}\otimes V_{i+1}\hookrightarrow V_{i}\otimes V_{i}$ for all $1<i<N$. For a graded $G$-module $V=\bigoplus_{d=0}^M V_d$,  we say that $V$ is \emph{$G$-log-concave} if the sequence $(V_0,\dots,V_M)$ is $G$-log-concave. Rhoades \cite{MR4821538} conjectured that the orbit harmonics ring $R(\symm_n)$ of the permutation matrix locus $\symm_n\subseteq\Mat_{n\times n}(\CC)$ is $\symm_n\times\symm_n$-log-concave, generalizing Chen's conjecture.

Interestingly, we have the locus identity $\ZZZ_{n,m,r} = \symm_n$ if $r=m=n$. Therefore, we further generalize Rhoades's conjecture as follows.
\begin{conjecture}\label{conj:log-concavity}
    For $0\le r\le\min\{m,n\}$, the orbit harmonics ring $R(\ZZZ_{n,m,r})$ is $\symm_n\times\symm_m$-log-concave.
\end{conjecture}
Conjecture~\ref{conj:log-concavity} has been verified by coding for $n\le 8,m\le 10$. Note that Conjecture~\ref{conj:log-concavity} has more flexibility than Rhoades's conjecture, as rook placement loci $\ZZZ_{n,m,r}$ live in $\Mat_{n\times m}(\CC)$ rather than $\Mat_{n\times n}(\CC)$. This flexibility may allow more potential induction strategies, making the log-concavity easier to prove. Furthermore, $R(\UZ_{n,m,r})$ is also conjectured to be $\symm_n\times\symm_m$-log-concave by \cite{liu2025extensionviennotsshadowrook}, and Proposition~\ref{prop:UZ-surj} connects $R(\ZZZ_{n,m,r})$ with $R(\UZ_{n,m,r})$. This connection may help us prove Conjecture~\ref{conj:log-concavity}.

Another open problem regards the standard monomial basis. Although we have found several graded character formulae for the graded $\symm_n\times\symm_m$-module $R(\ZZZ_{n,m,r})$, its standard monomial basis remains mysterious. Consequently, we ask for its standard monomial basis with respect to an appropriate monomial order of $\CC[\xxx_{n\times m}]$.
\begin{problem}\label{problem:basis}
    Find the standard monomial basis of $R(\ZZZ_{n,m,r})$ with respect to an appropriate monomial order.
\end{problem}
As Theorem~\ref{thm:good-ver-positive-formula} tells us the Hilbert series $\Hilb(R(\ZZZ_{n,m,r});q)$ of $R(\ZZZ_{n,m,r})$, it may give us some hint about Problem~\ref{problem:basis}. Moreover, the standard monomial basis of $R(\UZ_{n,m,r})$ is given by \cite{liu2025extensionviennotsshadowrook}, and Proposition~\ref{prop:UZ-surj} reveals the connection between $R(\ZZZ_{n,m,r})$ and $R(\UZ_{n,m,r})$, which may be helpful to the solution of Problem~\ref{problem:basis}.

One more interesting direction is to generalize our results to colored rook placements. Write the wreath product $\symm_{N,k}\coloneqq(\ZZ/r\ZZ)\wr\symm_N$ for the \emph{$k$-colored permutation group}. We embed $\symm_{N,k}$ into $\Mat_{N\times N}(\CC)$ by
\[\symm_{N,k} = \Big\{X\in\Mat_{N\times N}(\CC)\,:\,\begin{array}{cc}
     & \text{$X$ has exactly one nonzero entry in each row and column,} \\
     & \text{and nonzero entries of $X$ are $k$-th roots-of-unity}
\end{array}\Big\}.\]
Let $\ZZZ_{n,m,r}^{(k)}\subseteq\Mat_{n\times m}(\CC)$ be the \emph{$k$-colored rook placement locus} given by
\[\ZZZ_{n,m,r}^{(k)}\coloneqq\Bigg\{X\in\Mat_{n\times m}(\CC)\,:\,\begin{array}{cc}
     & \text{$X$ has at most one nonzero entry in each row and column,} \\
     & \text{nonzero entries of $X$ are $k$-th roots-of-unity,} \\
     & \text{and $X$ possesses exactly $r$ nonzero entries}
\end{array}\Bigg\}.\]
We ask for the graded module structure of $R(\ZZZ_{n,m,r}^{(k)})$.
\begin{problem}\label{problem:grad-str}
    Find the graded $\symm_{n,k}\times\symm_{m,k}$-module structure of $R(\ZZZ_{n,m,r}^{(k)})$.
\end{problem}

\section{Acknowledgements}\label{sec:acknowledgements}
We thank Brendon Rhoades for constructive suggestions about the content and structure of this paper.

\appendix
\section{The defining ideals for orbit harmonics on involution matrix loci}\label{sec:appendix}
We apply the same technique in the proof of Proposition~\ref{prop:ideal-equal} to another type of matrix loci studied by \cite{MR4887799}, namely the involution matrix loci $\MMM_{n,a}\subseteq\Mat_{n\times m}(\CC)$ given by
\[\MMM_{n,a}\coloneqq\{w\in\symm_n\,:\,\text{$w^2=1$ and $w$ has exactly $a$ fixed points}\}\]
where we identify each permutation in $\symm_n$ with its permutation matrix in $\Mat_{n\times n}(\CC)$. We will prove a concise generating set of the defining ideal $\gr\II(\MMM_{n,a})\subseteq\CC[\xxx_{n\times n}]$ in Proposition~\ref{prop:involution-ideal}.

We first state some necessary results in \cite{MR4887799}. Note that $\MMM_{n,a}$ is closed under the conjugation action of $\symm_n$ given by $g(w)=gwg^{-1}$ for $g\in\symm_n,w\in\MMM_{n,a}$. Therefore, $R(\MMM_{n,a})=\CC[\xxx_{n\times n}]/\gr\II(\MMM_{n,a})$ is also a graded $\symm_n$-module with the $\symm_n$-action generated by $g(x_{i,j}) = x_{g(i),g(j)}$. The graded module structure of $R(\MMM_{n,a})$ is given by \cite[Theorem 5.21]{MR4887799} as follows:
\begin{theorem}\label{thm:involution-grad-str}
    Suppose $a \equiv n \mod 2$. The graded Frobenius image of $R(\MMM_{n,a})$ is given by
    \[
        \grFrob(R(\MMM_{n,a});q) = \sum_{d \, = \, 0}^{(n-a)/2} \{
            h_d[h_2] \cdot h_{n-2d} - h_{d-1}[h_2] \cdot h_{n-2d+2}
        \}_{\lambda_1 \leq n-2d+a} \cdot q^d
    \]  
    where we interpret $h_{-1} := 0$.
\end{theorem}
Here $h_d[h_2]$ is a plethysm of symmetric functions (see \cite{MR3443860} for more details on plethysm). Fortunately, it has a pretty simple Schur expansion
\[h_d[h_2] = \sum_{\substack{\lambda\vdash 2d \\ \text{$\lambda$ is even}}} s_\lambda.\]

We need the monomials $\mmm(w)\in\CC[\xxx_{n\times n}]$ given by
\[\mmm(w) = \prod_{\substack{i\in[n]\\ i<w(i)}}x_{i,w(i)}\]
for all involutions $w\in\symm_n$ (i.e., $w^2=1$). They will play a crucial role in constructing spanning sets and generating sets.

We introduce an ideal $I_n^{(a)}\subseteq\CC[\xxx_{n\times n}]$ generated by
\begin{itemize}
    \item all sums $x_{i,1} + \dots + x_{i,n}$ of variables in a single row,
    \item all sums $x_{1,j} + \dots + x_{n,j}$ of variables in a single column,
    \item all products $x_{i,j} \cdot x_{i,j'}$ for $1\le i\le n,1\le j,j^\prime\le n$ of variables in a single row,
    \item all products $x_{i,j} \cdot x_{i',j}$ for $1\le i,i^\prime\le n,1\le j\le n$ of variables in a single column,
    \item all diagonally symmetric differences $x_{i,j} - x_{j,i}$ of variables,
    \item the diagonal sum $\sum_{i=1}^n x_{i,i}$,
    \item all products $\prod_{i\in S} x_{i,i}$ for $S\subseteq [n]$ with $\lvert S\rvert =a+1$, and
    \item all monomials $\mmm(w)$ for $w\in\bigsqcup_{a^\prime <a}\MMM_{n,a^\prime}$.
\end{itemize}
Our main goal is to show that $\gr\II(\MMM_{n,a}) = I_n^{(a)}$, but we can only know the containment immediately from \cite[Lemmas 5.5 and 5.9]{MR4887799} at first:
\begin{lemma}\label{lem:involution-containment}
    We have $I_n^{(a)}\subseteq \gr\II(\MMM_{n,a})$.
\end{lemma}

Now we give two results regarding spanning sets. The first spanning-set result is immediate from \cite[Theorem 3.3, Proposition 3.4]{MR4887799} and the fact that $\mmm(w)\in I_n^{(a)}$ for all $w\in\bigsqcup_{b<a}\MMM_{n,b}$.
\begin{lemma}\label{lem:involution-grad-span}
    The family of monomials
    \[\bigsqcup_{b\ge a}\{\mmm(w)\,:\,w\in\MMM_{n,b}\}\]
    descends to a spanning set of $\CC[\xxx_{n\times n}]/I_n^{(a)}$. In particular, we have $(\CC[\xxx_{n\times n}]/I_n^{(a)})_d =\{0\}$ for $d>\frac{n-a}{2}$.
\end{lemma}

The second spanning-set result is \cite[Lemma 5.10]{MR4887799}:
\begin{lemma}\label{lem:involution-ungrad-span}
    For any $0\le k\le d\le\lfloor n/2\rfloor$, the vector space $\CC[\xxx_{n\times n}]_{\le k}/(\II(\MMM_{n,n-2d})\cap\CC[\xxx_{n\times n}]_{\le k})$ is spanned by $\{\mmm(w)\,:\,w\in\MMM_{n,n-2k}\}$.
\end{lemma}

We can assign upper bounds to $\lambda_1$ for $V^\lambda$ appearing in $\CC[\xxx_{n\times n}]/I_n^{(a)}$ using \cite[Lemma 5.7]{MR4887799} and Lemma~\ref{lem:ann-length}, which immediately yields:
\begin{lemma}\label{lem:involution-length}
    For $0\le d\le\frac{n-a}{2}$, each $\symm_n$-irreducible $V^\lambda$ appearing in $(\CC[\xxx_{n\times n}]/I_n^{(a)})_d$ satisfies $\lambda_1\le n-2d+a$.
\end{lemma}

For convenience, we define a partial order $\prec$ on the set $\MMM_n\coloneqq\{w\in\symm_n\,:\,w^2=1\}$ of involutions in $\symm_n$. For $u,v\in\MMM_n$, we write $u\prec v$ if all the $2$-cycles in the cycle decomposition of $u$ also appear in the cycle decomposition of $v$. Therefore, $\prec$ is generated by the covering $\precdot$ on $\MMM_n$ where $u\precdot v$ if and only if the cycle decomposition of $v$ arises from adding a $2$-cycle to the cycle decomposition of $u$. We need the last technique result to show the ideal structure in Proposition~\ref{prop:involution-ideal}.
\begin{lemma}\label{lem:involution-ideal-pre}
    For $0\le d\le \lfloor n/2\rfloor$, we have
    \[\Frob((\CC[\xxx_{n\times n}]/I_n^{(n-2d)})_{d}) \le \Frob(R(\MMM_{n,n-2d})_d).\]
\end{lemma}
\begin{proof}
    For $w\in\MMM_n$, we define a function $\mathbf{1}_w\in\CC[\MMM_{n,n-2d}]$ by
    \[\mathbf{1}_w(u)\coloneqq\begin{cases}
        1, &\text{if $w\prec u$} \\
        0, &\text{otherwise}
    \end{cases}\]
    for all $u\in\MMM_{n,n-2d}$. For $0\le k\le d$, we write
    \[W_{d,k}\coloneqq \CC\cdot\{\mathbf{1}_w\,:\,w\in\MMM_{n,n-2k}\}\subseteq\CC[\MMM_{n,n-2d}].\]
    Consider the identification of $\symm_n$-modules
    \[\CC[\xxx_{n\times n}]/\II(\MMM_{n,n-2d})\cong \CC[\MMM_{n,n-2d}].\]
    According to Lemma~\ref{lem:involution-ungrad-span}, this identification identifies the submodule $\CC[\xxx_{n\times n}]_{\le k}/(\II(\MMM_{n,n-2d})\cap\CC[\xxx_{n\times n}]_{\le k})\subseteq\CC[\xxx_{n\times n}]/\II(\MMM_{n,n-2d})$ with the submodule
    $W_{d,k}\subseteq\CC[\MMM_{n,n-2d}]$ for $0\le k\le d$. Therefore, we have the module identification
    \[R(\MMM_{n,n-2d})_d\cong\frac{\CC[\xxx_{n\times n}]_{\le d}/(\II(\MMM_{n,n-2d})\cap\CC[\xxx_{n\times n}]_{\le d})}{\CC[\xxx_{n\times n}]_{\le d-1}/(\II(\MMM_{n,n-2d})\cap\CC[\xxx_{n\times n}]_{\le d-1})}\cong W_{d,d}/W_{d,d-1}.\]
    To show Lemma~\ref{lem:involution-ideal-pre}, it suffices to construct a module surjection
    \begin{align}\label{eq:involution-surj}
        \sigma_{d}\,:\, W_{d,d}/W_{d,d-1} \twoheadrightarrow (\CC[\xxx_{n\times n}]/I_n^{(n-2d)})_d.
    \end{align}

    We begin by construct a module homomorphism $W_{d,d}\rightarrow \CC[\xxx_{n\times n}]_d$. As $W_{d,d}=\CC[\MMM_{n,n-2d}]$ has a basis $\{\mathbf{1}_w\,:\,w\in\MMM_{n,n-2d}\}$ closed under the $\symm_n$-action, we define the $\symm_n$-module homomorphism
    \begin{align*}
        \zeta_d\,:\,W_{d,d}&\longrightarrow\CC[\xxx_{n\times n}]_d \\
        \mathbf{1}_w&\longmapsto\mmm(w).
    \end{align*}
    Note that for all $u\in\MMM_{n,n-2d+2}$ we have $\mathbf{1}_u = \sum_{u\precdot v}\mathbf{1}_v$ and thus
    \begin{align*}
        &\zeta_d(\mathbf{1}_u) = \sum_{u\precdot v}\zeta(\mathbf{1}_v) = \sum_{u\precdot v}\mmm(v)=\sum_{\substack{i<j \\ u(i)=i \\ u(j)=j}}\mmm(u)\cdot x_{i,j} =\mmm(u)\cdot\sum_{\substack{i<j \\ u(i)=i \\ u(j)=j}}x_{i,j}
        \equiv  \mmm(u)\cdot\sum_{i<j}x_{i,j} \\ \equiv & \frac{1}{2}\cdot\mmm(u)\cdot\sum_{\substack{i,j=1 \\ i\neq j}}^n x_{i,j} \equiv \frac{1}{2}\cdot\mmm(u)\cdot\sum_{i,j=1}^n x_{i,j} \equiv 0 \mod{I_n^{(n-2d)}}
    \end{align*}
    where the first $\equiv$ arises from the fact that $x_{i,j}\cdot x_{i^\prime,j^\prime}\in I_n^{(n-2d)}$ whenever $\{i,j\}\cap\{i^\prime,j^\prime\}\neq\varnothing$, the second $\equiv$ uses the fact that $x_{i,j}-x_{j,i}\in I_n^{(n-2d)}$, the third $\equiv$ is derived from the fact that $\sum_{i=1}^n x_{i,i}\in I_n^{(n-2d)}$, and the fourth $\equiv$ is deduced from the fact that all the row sums and column sums belong to $I_n^{(n-2d)}$. As a result, we know $\zeta_d(W_{d,d-1})\subseteq I_n^{(n-2d)}$, so $\zeta_d$ descends to the $\symm_n$-module homomorphism
    \[\sigma_{d}\,:\, W_{d,d}/W_{d,d-1} \twoheadrightarrow (\CC[\xxx_{n\times n}]/I_n^{(n-2d)})_d\]
    which is surjective by Lemma~\ref{lem:involution-grad-span}. Now we have finished the surjection~\eqref{eq:involution-surj}, completing the proof.
\end{proof}

Finally, we are ready to give a concise generating set of $\gr\II(\MMM_{n,a})$.
\begin{proposition}\label{prop:involution-ideal}
    We have $\gr\II(\MMM_{n,a}) = I_n^{(a)}$.
\end{proposition}
\begin{proof}
    By Lemma~\ref{lem:involution-containment}, we have an $\symm_n$-module surjection
    \[\CC[\xxx_{n\times n}]/I_n^{(a)}\twoheadrightarrow R(\MMM_{n,a}),\]
    so it suffices to force this surjection to be an isomorphism. Consequently, we only need to show
    \begin{align}\label{ineq:involution-frob}
        \Frob((\CC[\xxx_{n\times n}]/I_n^{(a)})_d) \le \Frob(R(\MMM_{n,a})_d)
    \end{align}
    for all $d\ge 0$.
    
    \noindent\textbf{Case (1):} If $d>\frac{n-a}{2}$, Theorem~\ref{thm:involution-grad-str} indicates that $\Frob(R(\MMM_{n,a})_d) =0$, and Lemma~\ref{lem:involution-grad-span} indicates that $\Frob((\CC[\xxx_{n\times n}]/I_n^{(a)})_d)=0$. Therefore, the inequality~\eqref{ineq:involution-frob} holds.
    
    \noindent\textbf{Case (2):} If $0\le d\le\frac{n-a}{2}$, Lemma~\ref{lem:involution-ideal-pre} and Theorem~\ref{thm:involution-grad-str} together yield
    \begin{align}\label{ineq:involution-frob'}
        \Frob((\CC[\xxx_{n\times n}]/I_n^{(n-2d)})_d) \le \{h_d[h_2]\cdot h_{n-2d} - h_{d-1}[h_2]\cdot h_{n-2d+2}\}_{\lambda_1\le 2(n-2d)}.
    \end{align}
    The inequality $n-2d\ge a$ yields the containment $(I_n^{(n-2d)})_d\subseteq (I_n^{(a)})_d$, implying the $\symm_n$-module surjection
    \[(\CC[\xxx_{n\times n}]/I_n^{(n-2d)})_d\twoheadrightarrow(\CC[\xxx_{n\times n}]/I_n^{(a)})_d\]
    and thereby
    \[\Frob((\CC[\xxx_{n\times n}]/I_n^{(a)})_d)\le\Frob((\CC[\xxx_{n\times n}]/I_n^{(n-2d)})_d).\]
    Appending \eqref{ineq:involution-frob'} to this inequality, we obtain
    \[\Frob((\CC[\xxx_{n\times n}]/I_n^{(a)})_d)\le\{h_d[h_2]\cdot h_{n-2d} - h_{d-1}[h_2]\cdot h_{n-2d+2}\}_{\lambda_1\le 2(n-2d)}\]
    which is strengthened by Lemma~\ref{lem:involution-length} to
    \[\Frob((\CC[\xxx_{n\times n}]/I_n^{(a)})_d)\le\{h_d[h_2]\cdot h_{n-2d} - h_{d-1}[h_2]\cdot h_{n-2d+2}\}_{\lambda_1\le n-2d+a},\]
    so the inequality~\eqref{ineq:involution-frob} still holds.

    To summarize, the inequality~\eqref{ineq:involution-frob} holds in both cases, finishing the proof.
\end{proof}

\printbibliography

\end{document}